\documentclass[12pt,leqno]{amsart}

\usepackage{amssymb,amsmath,amsthm,bm,transparent,color}
\usepackage{graphicx}
\usepackage{color}
\usepackage[textsize=tiny]{todonotes}
\usepackage[normalem]{ulem}
\usepackage[bookmarksopen,bookmarksdepth=3,colorlinks,citecolor=red,pagebackref,hypertexnames=true]{hyperref}
\usepackage[msc-links,nobysame,non-sorted-cites, initials]{amsrefs}
\usepackage[inline]{enumitem}
\usepackage{mathtools}
\mathtoolsset{showonlyrefs}
\usepackage{etoolbox}
\usepackage{chngcntr}

\definecolor{darkblue}{RGB}{0,0,160}

\usepackage[lining]{libertine}   
\usepackage{cabin}
\usepackage[libertine]{newtxmath}
\usepackage[T1]{fontenc}

\def\e{{\rm e}}

\def\d{{\rm d}}
\def\dist{{\rm dist}}

\def\R {\mathbb{R}}
\def\N {\mathbb{N}}

\def\tiles {{\mathbf{T}}}

\def\RR {{\mathcal R}}

\def\mass{\mathsf{mass}}

\def\Q {{\mathcal Q}}

\def\M {{\mathrm M}}

\def\RR {{\mathcal R}}
\def\Z {{\mathbb Z}}
\def\1 {{\mbox{\boldmath 1}}}

\def \l {\langle}

\def \r {\rangle}

\def \and{\quad\text{and}\quad}

\def\ind{\cic{1}}
\newcommand{\cic}{\bm}

\def \no#1#2#3 {{\bf #1} (#3), #2.}
\def \eds#1#2#3 {#1, #2, #3.}

 \def\Xint#1{\mathchoice
	   {\XXint\displaystyle\textstyle{#1}}%
	   {\XXint\textstyle\scriptstyle{#1}}%
	   {\XXint\scriptstyle\scriptscriptstyle{#1}}%
	   {\XXint\scriptscriptstyle\scriptscriptstyle{#1}}%
	   \!\int}
	 \def\XXint#1#2#3{{\setbox0=\hbox{$#1{#2#3}{\int}$}
	     \vcenter{\hbox{$#2#3$}}\kern-.5\wd0}}
	 
	 \def\avgint{\Xint-}

\counterwithout{subsubsection}{subsection}
\makeatletter
\let\c@equation\c@subsection

\let\c@figure\c@subsection

\let\c@subsubsection\c@subsection

\makeatother

\newcounter{counter}

 \numberwithin{equation}{section}
\numberwithin{counter2}{section}
\numberwithin{figure}{section}
\newtheorem{proposition}[equation]{Proposition}
\newtheorem{theorem}[counter]{Theorem}

\newtheorem{lemma}[equation]{Lemma} 
 
\theoremstyle{definition}
\newtheorem{definition}[equation]{Definition}
\newtheorem*{remark*}{Remark}
\newtheorem*{warn*}{A word of warning}

\newtheorem{remark}[equation]{Remark} 
\theoremstyle{plain}
\numberwithin{corollary}{counter}

\makeatletter
\let\save@mathaccent\mathaccent
\newcommand*\if@single[3]{%
  \setbox0\hbox{${\mathaccent"0362{#1}}^H$}%
  \setbox2\hbox{${\mathaccent"0362{\kern0pt#1}}^H$}%
  \ifdim\ht0=\ht2 #3\else #2\fi
  }
\newcommand*\rel@kern[1]{\kern#1\dimexpr\macc@kerna}
\newcommand*\widebar[1]{\@ifnextchar^{{\wide@bar{#1}{0}}}{\wide@bar{#1}{1}}}
\newcommand*\wide@bar[2]{\if@single{#1}{\wide@bar@{#1}{#2}{1}}{\wide@bar@{#1}{#2}{2}}}
\newcommand*\wide@bar@[3]{%
  \begingroup
  \def\mathaccent##1##2{%
    \let\mathaccent\save@mathaccent
    \if#32 \let\macc@nucleus\first@char \fi
    \setbox\z@\hbox{$\macc@style{\macc@nucleus}_{}$}%
    \setbox\tw@\hbox{$\macc@style{\macc@nucleus}{}_{}$}%
    \dimen@\wd\tw@
    \advance\dimen@-\wd\z@
    \divide\dimen@ 3
    \@tempdima\wd\tw@
    \advance\@tempdima-\scriptspace
    \divide\@tempdima 10
    \advance\dimen@-\@tempdima
    \ifdim\dimen@>\z@ \dimen@0pt\fi
    \rel@kern{0.6}\kern-\dimen@
    \if#31
      \overline{\rel@kern{-0.6}\kern\dimen@\macc@nucleus\rel@kern{0.4}\kern\dimen@}%
      \advance\dimen@0.4\dimexpr\macc@kerna
      \let\final@kern#2%
      \ifdim\dimen@<\z@ \let\final@kern1\fi
      \if\final@kern1 \kern-\dimen@\fi
    \else
      \overline{\rel@kern{-0.6}\kern\dimen@#1}%
    \fi
  }%
  \macc@depth\@ne
  \let\math@bgroup\@empty \let\math@egroup\macc@set@skewchar
  \mathsurround\z@ \frozen@everymath{\mathgroup\macc@group\relax}%
  \macc@set@skewchar\relax
  \let\mathaccentV\macc@nested@a
  \if#31
    \macc@nested@a\relax111{#1}%
  \else
    \def\gobble@till@marker##1\endmarker{}%
    \futurelet\first@char\gobble@till@marker#1\endmarker
    \ifcat\noexpand\first@char A\else
      \def\first@char{}%
    \fi
    \macc@nested@a\relax111{\first@char}%
  \fi
  \endgroup
}
\makeatother

\begin{document}

\setcounter{tocdepth}{1}

\title[Directional square functions]{Directional square functions}

\author[N. Accomazzo]{Natalia Accomazzo}
\address[N.\ Accomazzo]{Department of Mathematics, The University of British Columbia\\ \newline \indent Room 121, 1984 Mathematics Road, Vancouver, BC, Canada V6T 1Z2}
\email{\href{mailto:naccomazzo@math.ubc.ca}{\textnormal{naccomazzo@math.ubc.ca}}}

\author[F.\ Di Plinio]{Francesco Di Plinio} 
\address[F.\ Di Plinio]{Department of Mathematics, Washington University in Saint Louis\\ \newline \indent 1 Brookings Drive, Saint Louis, Mo 63130, USA}
\email{\href{mailto:francesco.diplinio@wustl.edu}{\textnormal{francesco.diplinio@wustl.edu}}}

\author[P.\ Hagelstein]{Paul Hagelstein}
\address[P. Hagelstein]{Department of Mathematics, Baylor University, Waco, Texas 76798}
\email{\href{mailto:paul_hagelstein@baylor.edu}{\textnormal{paul\!\hspace{.018in}\_\,hagelstein@baylor.edu}}}

\author[I.\ Parissis]{Ioannis Parissis}
\address[I.\ Parissis]{Departamento de Matem\'aticas, Universidad del Pa\'is Vasco, Aptdo. 644\\ \newline \indent  48080 Bilbao, Spain and Ikerbasque, Basque Foundation for Science, Bilbao, Spain}
\email{\href{mailto:	ioannis.parissis@ehu.es}{\textnormal{ioannis.parissis@ehu.es}}}

\author[L.\ Roncal]{Luz Roncal}
\address[L.\ Roncal]{BCAM - Basque Center for Applied Mathematics \\ \newline \indent 
48009 Bilbao, Spain, and Ikerbasque, Basque Foundation for Science, Bilbao, Spain\\ \newline \indent and Universidad del Pa\'is Vasco, Bilbao, Spain}
\email{\href{mailto:lroncal@bcamath.org}{\textnormal{lroncal@bcamath.org}}}

\thanks{N. Accomazzo is partially supported by the project PGC2018-094528-B-I00 (AEI/FEDER, UE) with acronym ``IHAIP'', project MTM-2017-82160-C2-2-P of the Ministerio de Econom\'ia y Competitividad (Spain), and grant T1247-19 of the Basque Government. F.\ Di Plinio is partially supported by the National Science Foundation under the grants DMS-1800628, DMS-2000510, DMS-2054863,. P. Hagelstein is partially supported by a grant from the Simons Foundation (\#521719 to Paul Hagelstein). I.\ Parissis is partially supported by the project PGC2018-094528-B-I00 (AEI/FEDER, UE) with acronym ``IHAIP'', grant T1247-19 of the Basque Government and IKERBASQUE. L. Roncal is supported by the Basque Government through the BERC 2018-2021 program, by the Spanish Ministry of Economy and Competitiveness MINECO: BCAM Severo Ochoa excellence accreditation SEV-2017-2018  and through project PID2020-113156GB-I00. She also acknowledges the RyC project RYC2018-025477-I and IKERBASQUE. Declarations of interest: none}



\maketitle

\begin{abstract} 
Quantitative formulations of Fefferman's counterexample for the ball multiplier are naturally linked to square function estimates for conical and directional multipliers. In this article we develop a novel framework for these square function estimates, based on a directional embedding theorem for Carleson sequences and multi-parameter time-frequency analysis techniques.   As applications we prove sharp or quantified bounds for Rubio de Francia type square functions of conical multipliers and of  multipliers adapted to rectangles pointing along $N$ directions. A suitable combination of these estimates yields a new and currently best-known  logarithmic bound for the Fourier restriction to an $N$-gon, improving on previous results of A. C\'ordoba.
 Our directional Carleson embedding extends to the weighted setting, yielding previously unknown weighted estimates for directional maximal functions and singular integrals.
\end{abstract}

\tableofcontents
 
\section{Motivation and main results} 
 The celebrated theorem of Charles Fefferman from \cite{FeffBall} shows that the ball multiplier is an  unbounded operator on $L^p(\R^n)$ for all $p\neq 2$ whenever $n\geq 2$. A well-known argument  originally due to Yves Meyer, \cite{DeGuzman}, exhibits the intimate relationship of the ball multiplier with vector-valued estimates for directional singular integrals along all possible directions. Fefferman proves in \cite{FeffBall} the impossibility of such estimates by testing these vector-valued inequalities on a Kakeya set.

 Besicovitch or Kakeya sets are compact sets in the Euclidean space that contain a line segment of unit length in every direction. Sets of this type with zero Lebesgue measure do exist. However, in two  dimensions, Kakeya sets  are necessarily of full Hausdorff dimension. The question of the Hausdorff dimension of Kakeya sets can be then formulated as a question of quantitative boundedness of the Kakeya maximal function, which is a maximal directional average along rectangles of fixed eccentricity and pointing along arbitrary directions.

 The importance of the ball multiplier  for the summation of higher dimensional Fourier series, as well as its intimate connection to Kakeya sets, have motivated a host of problems in harmonic analysis which have been driving relevant research since the 1970s. Finitary or smooth models of the ball multiplier such as the polygon multiplier and the Bochner-Riesz means quantify the failure of boundedness of the ball multiplier and formalize the close relation of these operators with directional maximal and singular averages.

 This paper is dedicated to the study of a variety of operators in the plane that are all connected in one way or another with the ball multiplier. Our point of view is through the analysis of directional operators mapping into $L^p(\R^2;\ell^q)$-spaces where the inner $\ell^q$-norm is taken with respect to the set of directions. Different values of $q$ are relevant in our analysis but the cases $q=2$ and $q=\infty$ are of particular interest. On one hand,  the case $q=\infty$ arises when considering maximal directional averages and the corresponding differentiation theory along directions; see \cite{BatTrans,CDR,DPPAlg,KatzDuke} for classical and recent work on the subject. On the other hand, the case $q=2$  is especially relevant for Meyer's argument that bounds the norm  of a vector-valued directional Hilbert transform by the norm of the ball multiplier. It also arises when dealing with square functions associated to conical or directional Fourier multipliers of the type
\[
f\mapsto \{C_j f: j=1,\ldots, N\}
\]
where each $C_j$ is adapted to a different coordinate pair and the $C_j$ have disjoint or well-separated Fourier support. These estimates are directional analogues of the celebrated square function estimate for Fourier restriction to families of disjoint cubes, due to Rubio de Francia \cite{RdF}, and they  appear naturally when seeking for quantitative estimates on the $N$-gon Fourier multiplier. 
 
 While such square function 	estimates have been considered previously in the literature, and usually approached directly via weighted norm inequalities, our treatment is novel and leads to improved and in certain cases sharp estimates in terms of the cardinality of the set of directions. It rests on a  new directional Carleson measure condition and corresponding embedding theorem,  which is subsequently applied to intrinsic directional square functions of time-frequency nature.    The link between the abstract Carleson embedding theorem and the    applications is provided by directional, one and two-parameter time-frequency analysis models. The latter  allow us to reduce  estimates for directional operators to those of the  corresponding intrinsic square functions involving directional wave packet coefficients. We  note that in the fixed coordinate system case, related square functions have appeared in Lacey's work \cite{LR}, while a single-scale directional square function similar to those of Section \ref{s:isf}  is present in \cite{DP+} by Guo, Thiele, Zorin-Kranich and the second author.

 Having clarified the context of our investigation, we turn to the detailed description of our main results and techniques.

\subsection*{A new approach to directional square functions}  While we address several types of square functions associated  to directional multipliers, our analysis of each relies on a common first step. This is an  $L^4$-square function inequality for abstract Carleson measures  associated with one and two-parameter collections of rectangles in $\R^2$, pointing along a finite set of $N$ directions; this setup is  presented in Section \ref{sec:carleson} and the  central result is Theorem \ref{thm:carleson}. Section \ref{sec:carleson} builds upon the proof technique first introduced  by Katz \cite{KatzDuke} and revisited by Bateman \cite{BatTrans} in the study of sharp weak $L^2$-bounds for maximal directional operators. Our main novel contributions are the formulation of an abstract directional Carleson condition which is flexible enough to be applied in the context of time-frequency square functions, and the realization that square functions in $L^4$ can be treated in a $TT^*$-like fashion. The advancements over \cite{BatTrans,KatzDuke} also include the possibility of handling two-parameter collections of rectangles.

In Section \ref{s:isf}, we verify that the  Carleson condition, which is a necessary assumption in the directional embedding of Theorem \ref{thm:carleson}, is satisfied by the intrinsic directional wave packet coefficients  associated with certain time-frequency tile configurations, and Theorem \ref{thm:carleson} may be thus applied to obtain sharp estimates for discrete time-frequency models of directional  Rubio de Francia square function (for instance). Establishing the Carleson condition requires a precise  control of spatial tails of the wave packets: this control is obtained by a  careful use of Journ\'e's product theory lemma. 

The estimates obtained for the time-frequency model square functions are then applied to three main families of operators described below. All of them are defined in terms of an underlying set of $N$ directions. As in Fefferman's counterexample for the ball multiplier the Kakeya set is the main obstruction for obtaining uniform estimates. Depending on the type of operator the usable estimates will be restricted in the range $2<p<4$ for square function estimates or in the range $3/4<p<4$ for the self-adjoint case of the polygon multiplier. The fact that the estimates should be logarithmic in $N$ in the $L^p$-ranges above is directed by the Besicovitch construction of the Kakeya set. It is easy to see that for $p$ outside this range the only available estimates are essentially trivial polynomial estimates. Further obstructions deter any estimates for  Rubio de Francia type square function in the range $p<2$ already in the one-directional case.

\subsection*{Sharp Rubio de Francia square function estimates in the directional setting}
Section \ref{sec:cone} concerns quantitative estimates of Rubio de Francia type for the square function associated with $N$ finitely overlapping cone multipliers, of both rough and smooth type. Beginning with the seminal article of Nagel, Stein and Wainger \cite{NSW}, square functions of this type are crucial in the theory of maximal operators, in particular  along lacunary directions, see for instance \cite{PR,SS}. In the case of  $N$ uniformly spaced    cones, logarithmic estimates with unspecified dependence were proved by A.\ C\'ordoba in \cite{CorGFA} using weighted theory. 

In order to make the discussion above more precise, and to give a flavor of the results of this paper, we introduce some basic notation.  Let $\uptau\subset (0,2\uppi)$ be an interval  and consider the corresponding smooth restriction to the frequency cone subtended by $\uptau$, namely 
\[
C_\uptau^\circ f(x)    \coloneqq \int_{0}^{2\uppi} \int_{0}^\infty  \widehat f(\varrho\e^{i\vartheta}) \upbeta_\uptau\left(\vartheta\right)  \e^{ix\cdot\varrho\e^{i\vartheta} } \, \varrho\d \varrho \d \vartheta,\qquad x\in\R^2,
\]
where $\upbeta_\uptau$ is a  smooth indicator on $\uptau$, namely it is supported in $\uptau$ and is identically one on the middle half of $\uptau$.

One of the main results of this paper is a quantitative estimate for a  square function associated with the  smooth conical multipliers of a finite collection of intervals with bounded overlap. In the statement of the theorem below $\ell^2 _{\cic{\uptau}}$ denotes the $\ell^2$-norm on the finite set of directions $\cic{\uptau}$.
\begin{theorem} \label{thm:smoothcones} Let $\cic{\uptau}=\{\uptau\}$ be a finite collection of  intervals in $[0,2\pi)$ with bounded overlap, namely
\[
\Bigl\| \sum_{\uptau\in \cic{\uptau}} \ind_{\uptau} \Bigr\|_\infty  \lesssim 1. 
\]
We then have the square function estimate
\begin{align*}
& \left\| \{C_\uptau^\circ f\}  \right\|_{L^p(\R^2; \ell^2_{\cic{\uptau}})} \lesssim_p (\log \# \cic{\uptau} )^{\frac{1}2-\frac1p}\|f\|_p 
\end{align*} 
for $2\leq p<4$,
as well as the restricted type analogue valid for all measurable sets $E$ 
\begin{align*}
&  \left\| \{C_\uptau^\circ(f\cic{1}_E)\}  \right\|_{L^4(\R^2; \ell^2_{\cic{\uptau}})}  \lesssim (\log \# \cic{\uptau} )^{\frac{1}4}|E|^{\frac14}\|f\|_\infty.
\end{align*}
The dependence on $\#\uptau$ in the estimates above is best possible.
\end{theorem}
The sharp estimate of Theorem~\ref{thm:smoothcones} above can be suitably bootstrapped in order to provide an estimate for rough conical frequency projections; the precise statement can be found in Theorem \ref{thm:wdcones} of Section~\ref{sec:cone}. The sharpness of the estimates in Theorem~\ref{thm:smoothcones} above is discussed in \S\ref{sec:cexconical}.
  
A similar square function estimate associated with disjoint rectangular directional frequency projections is presented in Section~\ref{sec:rdf}. This is a square function that is very close in spirit to the one originally considered by Rubio de Francia in \cite{RdF}, and especially to the two-parameter version of Journ\'e from \cite{Journe} and revisited by Lacey in \cite{LR}. The novel element is the  directional aspect which comes from the fact that the frequency rectangles are allowed to point along a set of $N$ different directions. Our method of proof can deal equally well with one-parameter rectangular projections or collections of arbitrary eccentricities. As before we prove a sharp -in terms of the number of directions-  estimate for the smooth square function associated with rectangular frequency projections along $N$ directions; this is the content of Theorem~\ref{thm:rdfsmooth}. The main term in the  upper bound of Theorem~\ref{thm:rdfsmooth} matches the  logarithmic lower bound associated with the Kakeya set.

\subsection*{The polygon multiplier} The square function estimates discussed above may be combined with suitable vector-valued estimates in the directional setting in order to obtain a quantitative estimate for the operator norm of the  $N$-gon multiplier, namely the  Fourier restriction to a regular $N$-gon $\mathcal{P}_N$,  
\begin{equation}
\label{e:Ngon}
T_{\mathcal{P}_N} f(x) \coloneqq \int_{\mathcal{P}_N} \widehat f(\xi) \e^{ix\cdot \xi} \, \d \xi, \qquad x\in\R^2.
\end{equation}
In Section~\ref{sec:polygon} we give the details and proof of the following quantitative estimate for the polygon multiplier.

\begin{theorem} \label{thm:polygon} Let $\mathcal P_{N}$ be a regular $N$-gon in $\R^2$ and $T_{\mathcal{P}_N}$ be the corresponding Fourier restriction operator defined in  \eqref{e:Ngon}.  We have the estimate
 \[
 \left\|T_{\mathcal P_N} : L^p(\R^2) \right\| \lesssim (\log N)^{4\left|\frac12-\frac1p\right|} , \qquad \frac43 <p <4.
 \]
\end{theorem}

We limit ourselves to treating the regular $N$-gon  case; however, it will be clear from the proof that this restriction may be significantly weakened by requiring instead a well-distribution type assumption on the arcs defining the polygon, similar to the one that is implicit in Theorem \ref{thm:smoothcones}.

Precise $L^p$-bounds for the $N$-gon multiplier as a function of $N$ quantify Fefferman's counterexample and so the failure of boundedness of the ball multiplier when $p\neq 2$. A logarithmic  type estimate for $T_{\mathcal{P}_N}$ was first obtained by A.\ C\'ordoba in \cite{CorPoly}. While the exact dependence in  \cite{CorPoly} is not explicitly tracked, the upper bound on the operator norm obtained in \cite{CorPoly} must be necessarily larger than $O(\log N)^{\frac54}$ for $p$ close to the endpoints of the relevant interval: see Remark \ref{rem:compcor} and \S\ref{sec:cexradial} for details.
  While  the dependence obtained in Theorem \ref{thm:polygon} is a significant improvement over previous results, it does not match the currently best known lower bound, which is the same as that for the Meyer lemma constant in Lemma~\ref{lem:meyer}  and \S\ref{sec:cexmeyer}.

\begin{remark*} Let $\updelta>0$ and  $T_j$ be a smooth frequency restriction  to one of the $O(\updelta^{-1})$ tangential $\updelta \times \updelta^2$ boxes covering the $\updelta^2$ neighborhood of $\mathbb S^1$. Unlike the sharp forward square function estimate we prove in this article,  the \emph{reverse square function} estimate
 \begin{equation}
\label{e:revsf} 
\|f\|_{p} \leq C_{p,\updelta}  \left\|  \left\{T_jf: 1\leq j \leq O{\textstyle \left(\frac 1\updelta \right)} \right\}   \right\|_{L^p(\R^2:\ell^2_j)},
\end{equation}
holds with $C_{4,\updelta}=O(1)$ at the endpoint $p=4$. For the proof of this $L^4$-decoupling estimate see \cites{CorPoly,Feff}.  An extension to the range $2<p<4$ is  at the moment only possible via vector-valued methods, which introduce the loss $C_{p,\updelta}=O(|\log \updelta|^{{1}/{2}-1/p} )$. In fact \eqref{e:revsf} with the loss $C_{p,\updelta}$ claimed above follows easily from Lemma~\ref{l:Tj}; the details are contained in Remark~\ref{rmrk:invsf}.

Reverse square function inequalities of the type \eqref{e:revsf} have been popularized by Wolff in his proof of local smoothing estimates in the large $p$ regime; see also the related works \cite{GS,LP,LW,PS}. 
We refer to Carbery's note \cite{CarbA} for a proof that the $p=2n/(n-1)$ case of the  $\mathbb S^{n-1}$ reverse square function estimate implies the corresponding $L^{n}(\R^n)$ Kakeya maximal inequality, as well as the Bochner-Riesz conjecture. In \cite{CarbA}, the author also asks whether a $\updelta$-free estimate holds in the  range $2<p<2n/(n-1)$.  At the moment this is not known in any dimension.

On a different but related note,   weakening  \eqref{e:revsf} by replacing the right hand side with the larger square function of $\|f_j\|_p$ yields a sample (weak) \emph{decoupling} inequality: a full range of sharp decoupling inequalities for hypersurfaces with curvature have been established starting from the recent, seminal paper by Bourgain and Demeter \cite{BD1}. In the case of $\mathbb S^1$,  the weak decoupling inequality holds in the wider range $2\leq p \leq 6$, with $C_\upvarepsilon\updelta^{-\upvarepsilon}$ type  bounds outside of $[2,4]$: our methods do not seem to provide insights on the quantitative character of weak decoupling in this wider range. 
\end{remark*}

\subsection*{Weighted estimates for the maximal directional function} The simplest example of application of the directional Carleson embedding theorem  is the adjoint of the directional maximal function; this was already noticed by Bateman  \cite{BatTrans}, re-elaborating on the approach of Katz \cite{KatzDuke}. By duality, the $L^2$-directional Carleson embedding theorem of Section~\ref{sec:carleson} yields the sharp bound
for the weak $(2,2)$ norm of the maximal Hardy-Littlewood maximal function $M_N$ along $N$ \emph{arbitrary directions}
\[
 \|M_N:L^2(\R^2)\to L^{2,\infty}(\R^2)\|\sim \sqrt{\log N};
\]
this result first appeared in the quoted article \cite{KatzDuke}  by Katz.

Theorem \ref{thm:carleson} may be extended to the directional weighted setting. We describe this extension in Section \ref{s:wcarleson}, see Theorem \ref{thm:wcarleson}, and derive several novel weighted estimates for directional maximal and singular integrals as an application. 

More specifically, our weighted Carleson embedding Theorem \ref{thm:wcarleson} yields  a Fefferman-Stein type inequality for the operator $M_N$ with sharp dependence on the number of directions; this result is the content of Theorem~\ref{thm:maxweight}. Specializing to $A_1$-weights in the directional setting yields the first sharp weighted result for the maximal function along arbitrary directions. Furthermore, Theorem \ref{thm:singweight} contains an $L^{2,\infty}(w)$-estimate for the maximal directional singular integrals along $N$ directions, for suitable directional weights $w$, with a quantified logarithmic dependence in $N$. This is a weighted counterpart of the results of \cite{Dem,DDP}.

 \subsection*{Acknowledgments} The authors are grateful to Ciprian Demeter and Jongchon Kim  for fruitful discussions on  reverse square function estimates, and for providing additional references on the subject.

\section{An $L^2$-inequality for directional Carleson sequences}\label{sec:carleson}
In this section we prove an abstract $L^2$-inequality for certain Carleson sequences adapted to sets of directions: the main result is Theorem \ref{thm:carleson} below.  The Carleson sequences we will consider are indexed by parallelograms with long side pointing in a given set of directions in $\mathbb{R}^2$, and possessing certain natural properties. The definitions below are motivated by the applications we have in mind, all of them lying in the realm of directional singular and averaging operators.

\subsection{Parallelograms and sheared grids}Fix a coordinate system and the associated horizontal and vertical projections
of $A\subset  \R^2$: 
\[
 \uppi_1(A)\coloneqq\left\{x\in\R:\, \{x\}\times \R \cap A \neq \varnothing\right\},\qquad \uppi_2(A)\coloneqq\left\{y\in\R:\,\R \times\{ y\} \cap A \neq \varnothing\right\}.
 \]
Fix a finite set of slopes $S\subset [-1,1]$. Throughout, we indicate by $N=\#S$ the number of elements of $S$. In general we will deal with sets of directions
\[
V\coloneqq \{(1,s):\, s\in S\},\quad V^\perp\coloneqq \{ (-s,1):\, s\in S\}.
\]
We will conflate the descriptions of directions in terms of slopes in $S$ and in terms of vectors in $V$ with no particular mention. 

For each $s\in S$ let 
\begin{equation}
\label{e:shearing}
A_s\coloneqq \left[ \begin{matrix} 1 & 0 
\\ 
s & 1 \end{matrix}\right]
\end{equation}
be the corresponding shearing matrix. A \emph{parallelogram along $s$} is the image $P=A_s(I\times J)$ of the rectangular box $I\times J$ in the fixed coordinate system with $|I|\geq |J|$. We denote the collection of parallelograms along $s$ by $\mathcal P_s^2$ and 
\[
\mathcal P_S ^2 \coloneqq \bigcup _{s\in S} \mathcal P_s^2.
\]
In order to describe the setup for our general result we introduce a collection of directional dyadic grids of parallelograms. In order to define these grids we consider the two-parameter product dyadic grid  
 \[
 \mathcal D^2 _{0}\coloneqq\left\{R=I\times J:\, I,J\in \mathcal D(\R),\/ |I|\geq |J| \right\}
 \] 
obtained by taking cartesian product of the standard dyadic grid $\mathcal D(\R)$ with itself; we note that we only consider the rectangles in $\mathcal D\times \mathcal D$ whose horizontal side is longer than their vertical one. Define the sheared grids
\[
\mathcal D^2 _{s}\coloneqq\left\{A_sR:\, R\in \mathcal D^2 _0 \right\},\quad s\in S,\qquad \mathcal D^2 _S\coloneqq \bigcup_{s\in S}\mathcal D^2 _s.
\]
We will also use the notation
\[
\mathcal D^2 _{s,k_1,k_2}\coloneqq \left\{A_sR:\, R=I\times J\in \mathcal D^2 _0,\, |I|=2^{-k_1},\, |J|=2^{-k_2} \right\},\quad s\in S,\qquad k_1,k_2 \in\Z, \quad k_1\leq k_2.
\]
Note that $\mathcal D^2 _s$ is a special subcollection of  $\mathcal P_s^2$. In particular,  $R\in \mathcal D_s ^2$ is a parallelogram oriented along $v=(1,s)$ with vertical sides parallel to the $y$-axis and such that $\uppi_1(R)$ is a standard dyadic interval. Furthermore our assumptions on $S$ and the definition of $\mathcal D_0 ^2$ imply that the parallelograms in $\mathcal D^2 _S$ have long side with slope $|s|\leq 1$ and a vertical short side. With a slight abuse of language we will continue referring to the rectangles in $\mathcal D_S ^2$ as \emph{dyadic}.

\begin{figure}[ht]\label{fig:RtoRs}
\centering
\def\svgwidth{330pt}
\begingroup%
  \makeatletter%
  \providecommand\color[2][]{%
    \errmessage{(Inkscape) Color is used for the text in Inkscape, but the package 'color.sty' is not loaded}%
    \renewcommand\color[2][]{}%
  }%
  \providecommand\transparent[1]{%
    \errmessage{(Inkscape) Transparency is used (non-zero) for the text in Inkscape, but the package 'transparent.sty' is not loaded}%
    \renewcommand\transparent[1]{}%
  }%
  \providecommand\rotatebox[2]{#2}%
  \newcommand*\fsize{\dimexpr\f@size pt\relax}%
  \newcommand*\lineheight[1]{\fontsize{\fsize}{#1\fsize}\selectfont}%
  \ifx\svgwidth\undefined%
    \setlength{\unitlength}{723.58294505bp}%
    \ifx\svgscale\undefined%
      \relax%
    \else%
      \setlength{\unitlength}{\unitlength * \real{\svgscale}}%
    \fi%
  \else%
    \setlength{\unitlength}{\svgwidth}%
  \fi%
  \global\let\svgwidth\undefined%
  \global\let\svgscale\undefined%
  \makeatother%
  \begin{picture}(1,0.49324622)%
    \lineheight{1}%
    \setlength\tabcolsep{0pt}%
    \put(0,0){\includegraphics[width=\unitlength,page=1]{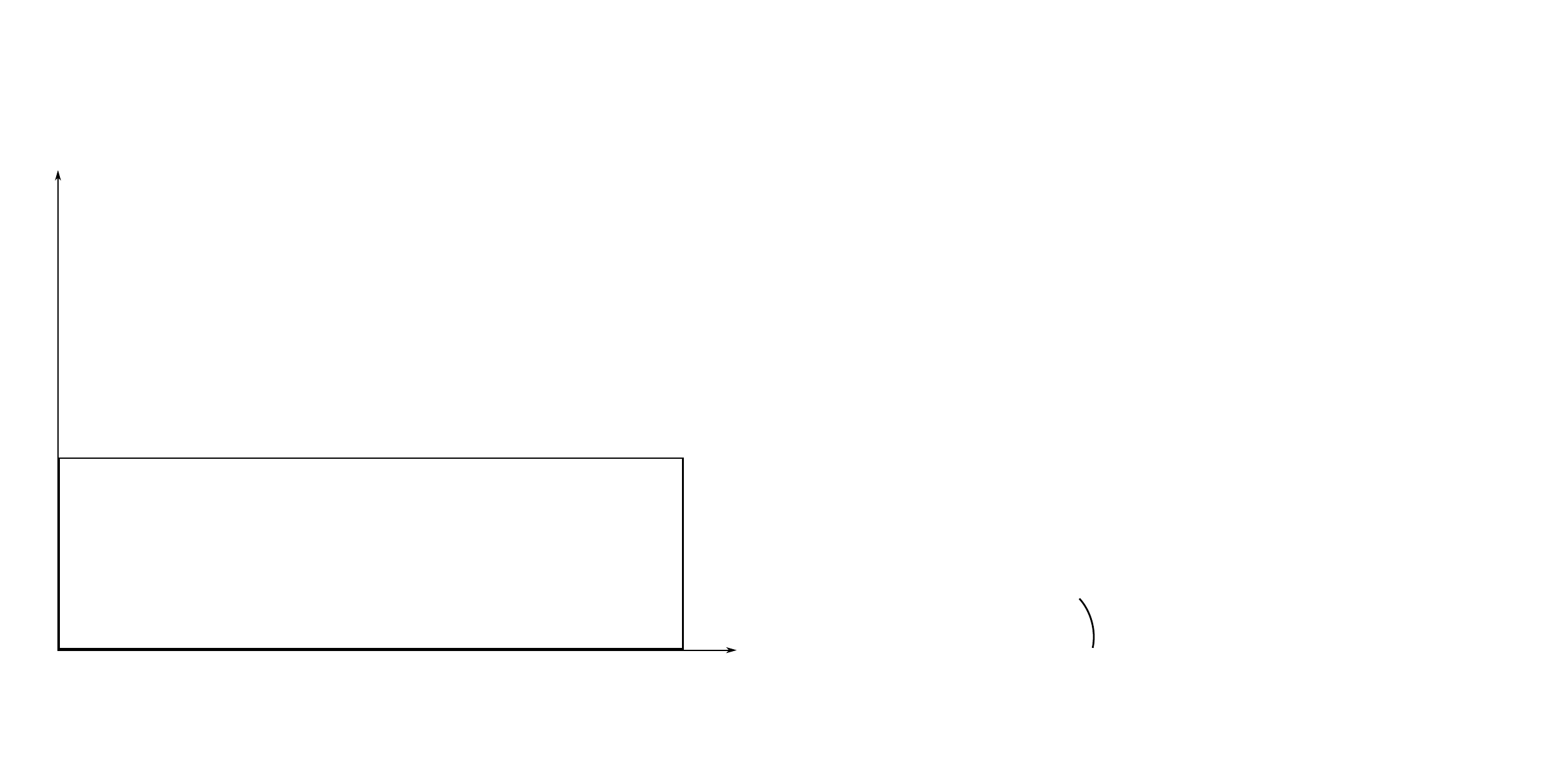}}%
    \put(0.10598895,0.1379423){\color[rgb]{0,0,0}\makebox(0,0)[lt]{\lineheight{0}\smash{\begin{tabular}[t]{l}$R=I\times[0,1]\in\mathcal D_0 ^2$\end{tabular}}}}%
    \put(0,0){\includegraphics[width=\unitlength,page=2]{grids.pdf}}%
    \put(0.66037296,0.17032085){\color[rgb]{0,0,0}\makebox(0,0)[lt]{\lineheight{0}\smash{\begin{tabular}[t]{l}$A_s R\in\mathcal D_s ^2$\end{tabular}}}}%
    \put(-0.00125515,0.19638162){\color[rgb]{0,0,0}\makebox(0,0)[lt]{\lineheight{0}\smash{\begin{tabular}[t]{l}$1$\end{tabular}}}}%
    \put(0.52627843,0.19559189){\color[rgb]{0,0,0}\makebox(0,0)[lt]{\lineheight{0}\smash{\begin{tabular}[t]{l}$1$\end{tabular}}}}%
    \put(0.53338587,0.07239541){\color[rgb]{0,0,0}\makebox(0,0)[lt]{\lineheight{0}\smash{\begin{tabular}[t]{l}$0$\end{tabular}}}}%
    \put(0.72181329,0.09434969){\color[rgb]{0,0,0}\makebox(0,0)[lt]{\lineheight{0}\smash{\begin{tabular}[t]{l}$ \tan \theta=s$\end{tabular}}}}%
    \put(0,0){\includegraphics[width=\unitlength,page=3]{grids.pdf}}%
    \put(0.4417783,0.31168086){\color[rgb]{0,0,0}\makebox(0,0)[lt]{\lineheight{0}\smash{\begin{tabular}[t]{l}$A_s$\end{tabular}}}}%
    \put(0.00032428,0.07358004){\color[rgb]{0,0,0}\makebox(0,0)[lt]{\lineheight{0}\smash{\begin{tabular}[t]{l}$0$\end{tabular}}}}%
    \put(0,0){\includegraphics[width=\unitlength,page=4]{grids.pdf}}%
    \put(0.22571024,0.00292328){\color[rgb]{0,0,0}\makebox(0,0)[lt]{\lineheight{0}\smash{\begin{tabular}[t]{l}$I$\end{tabular}}}}%
    \put(0,0){\includegraphics[width=\unitlength,page=5]{grids.pdf}}%
    \put(0.75654146,0.00529166){\color[rgb]{0,0,0}\makebox(0,0)[lt]{\lineheight{0}\smash{\begin{tabular}[t]{l}$I$\end{tabular}}}}%
  \end{picture}%
\endgroup%

\caption{The axis-parallel rectangle $R\in\mathcal D^2 _0$ is mapped to the slanted parallelogram $A_sR\in\mathcal D^2 _s$.}
\end{figure}

Several results in this paper will involve collections of  parallelograms $\mathcal R\subset \mathcal D^2 _S$. Writing $\mathcal R_s\coloneqq \mathcal R \cap\mathcal D_s ^2$ we have the natural decomposition of $\mathcal R $ into $\#S=N$ subcollections 
\[
\mathcal R  =\bigcup_{s\in S} \mathcal R_s.
\]
In general for any collection $\mathcal R$ of parallelograms we will use the notation
\[
\mathsf{sh}(\mathcal R)\coloneqq \bigcup_{R\in\mathcal R} R
\]
for the \emph{shadow} of the collection. Finally, for any collection of parallelograms $\mathcal R$ we define the corresponding maximal operator
\[
\mathrm{M}_{\mathcal R} f(x) \coloneqq \sup_{R\in \mathcal R} \l  |f| \r_R \cic{1}_R(x),\qquad f\in L^1 _{\mathrm{loc}}(\R^2),\quad x\in\R^2.
\]
We will also use the following notations for directional maximal functions:
\begin{equation}\label{eq:dirMv}
\M_v f(x)\coloneqq \sup_{r>0} \frac{1}{2r}\int_{-r} ^r |f(x+tv)|\,\d t,\qquad \M_jf(x)\coloneqq \M_{e_j}f(x), \quad j\in\{1,2\},\quad x\in\R^2.
\end{equation}

If $V\subset  \mathbb \R^2$ is a compact  set of directions with $0\notin V$, we  write 
\begin{equation}
\label{e:Mv}
\M_Vf\coloneqq \sup_{v\in V} \M_v f.
\end{equation}

In the definitions above and throughout the paper we use the notation 
\[
\l g \r_{E} = \avgint_E g \coloneqq \frac{1}{|E|}\int_E g(x)\,\d x
\]
whenever $g$ is a locally integrable function in $\R^2$ and $E\subset \R^2$ has finite measure.

\subsection{An embedding theorem for directional Carleson sequences}
In this section we will be dealing with Carleson-type sequences $a=\{a_R\}_{R\in\mathcal D_S ^2}$, indexed by dyadic parallelograms. In order to define them precisely we need a preliminary notion.
\begin{definition} Let $\mathcal L\subset  \mathcal P_S ^2$ be a collection of  parallelograms and let $s\in S$. We will say that $\mathcal L$ is subordinate to a collection $\mathcal T\subset \mathcal P_s^2$ if for each $L\in\mathcal L$ there exists $T\in\mathcal T$ such that $L\subset  T$; see Figure~\ref{fig:subordinate}.
\end{definition}

\begin{figure}[ht]
\centering
\def\svgwidth{330pt}
\begingroup%
  \makeatletter%
  \providecommand\color[2][]{%
    \errmessage{(Inkscape) Color is used for the text in Inkscape, but the package 'color.sty' is not loaded}%
    \renewcommand\color[2][]{}%
  }%
  \providecommand\transparent[1]{%
    \errmessage{(Inkscape) Transparency is used (non-zero) for the text in Inkscape, but the package 'transparent.sty' is not loaded}%
    \renewcommand\transparent[1]{}%
  }%
  \providecommand\rotatebox[2]{#2}%
  \ifx\svgwidth\undefined%
    \setlength{\unitlength}{311.20868bp}%
    \ifx\svgscale\undefined%
      \relax%
    \else%
      \setlength{\unitlength}{\unitlength * \real{\svgscale}}%
    \fi%
  \else%
    \setlength{\unitlength}{\svgwidth}%
  \fi%
  \global\let\svgwidth\undefined%
  \global\let\svgscale\undefined%
  \makeatother%
  \begin{picture}(1,0.44220777)%
    \put(0,0){\includegraphics[width=\unitlength,page=1]{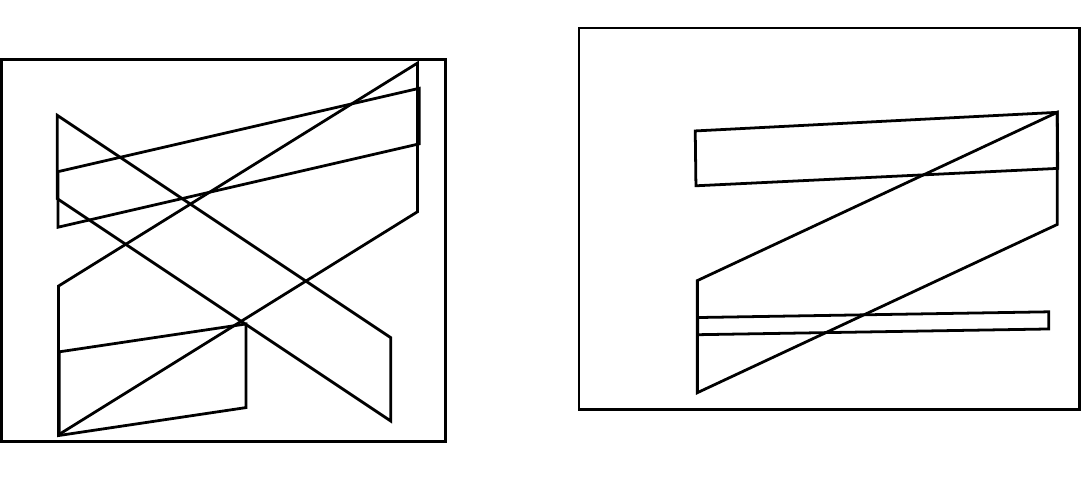}}%
    \put(0.09935607,0.34055067){\color[rgb]{0,0,0}\makebox(0,0)[lb]{\smash{$\mathcal L$}}}%
    \put(0,0){\includegraphics[width=\unitlength,page=2]{subordinate.pdf}}%
    \put(0.16325441,0.41877092){\color[rgb]{0,0,0}\makebox(0,0)[lb]{\smash{$T \in \mathcal T$}}}%
    \put(0.69867839,0.00453373){\color[rgb]{0,0,0}\makebox(0,0)[lb]{\smash{$T' \in \mathcal T$}}}%
    \put(0,0){\includegraphics[width=\unitlength,page=3]{subordinate.pdf}}%
    \put(0.55105117,0.28289539){\color[rgb]{0,0,0}\makebox(0,0)[lb]{\smash{$\mathcal L$}}}%
    \put(0,0){\includegraphics[width=\unitlength,page=4]{subordinate.pdf}}%
  \end{picture}%
\endgroup%
\caption{A collection $\mathcal L$ subordinate to a collection $\mathcal T\subset \mathcal P_0 ^2$.}
\label{fig:subordinate}
\end{figure}

It is important to stress that collections $\mathcal L$ are subordinate to rectangles $\mathcal T\subset \mathcal P_s^2$ \emph{having a fixed slope $s$}.
The Carleson sequences $a=\{a_R\}_{R\in\mathcal R}$ we will be considering will fall under the scope of the following definition. 

\begin{definition}\label{def:carleson}Let $a=\{a_R\}_{R\in\mathcal D_S ^2}$ be a sequence of nonnegative numbers. Then $a$ will be called an \emph{$L^\infty$-normalized} Carleson  sequence if for every $\mathcal L \subset \mathcal D_S ^2$ which is subordinate to some collection $\mathcal T \subset  \mathcal P^2_\uptau$ for some  fixed $\uptau\in S$, we have
	\[
\sum_{L\in\mathcal L}a_L\leq  |\mathsf{sh}(\mathcal T )|
	\]
and the quantity
	\[
	 \mathsf{mass}_a \coloneqq \sum_{R\in\mathcal D_S ^2} a_R 
	\]
	is finite.
Given a Carleson sequence $a=\{a_R: R\in\mathcal D_S ^2 \}$ and a collection $\mathcal R\subset  \mathcal D_S ^2$ we   define the corresponding \emph{balayage} 
\[
T_{\mathcal R}(a)(x)  \coloneqq \sum_{R\in\mathcal R} a_R \frac{\ind_R(x)}{|R|},\qquad x\in\R^2.
\]
We write $T(a)$ for $T_\mathcal R(a)$ when $\mathcal R=\mathcal D_S ^2$.
For $1\leq p \leq 2$ we then define the balayage norms   
\[
	\mathsf{mass}  _{a,p}   (\mathcal R)\coloneqq \left\|  T_{\mathcal R}(a) \right\|_{L^p}.
\]
Note that $\mathsf{mass}_{a,1} (\mathcal R)=\sum_{R\in\mathcal R} a_R\leq\mathsf{mass}_a$.
\end{definition}

\begin{remark}[Elementary properties of $\mathsf{mass}$]\label{rmrk:dircarl} Let $ \mathcal R  \subset  \mathcal D_\uptau ^2 $ for some fixed $\uptau\in S$. Then $\mathcal R$ is subordinate to itself and if $a$ is an $L^\infty$-normalized Carleson sequence we have
	\[
	\mathsf{mass}_{a,1} (\mathcal R)=\sum_{R\in\mathcal R }a_R\leq  |\mathsf{sh}(\mathcal R )|,\qquad \mathcal R \subset \mathcal D_\uptau ^2\quad\text{for some fixed}\,\, \uptau\in S.
	\]
Also, the very definition of $\mathsf{mass}$ and the $\log$-convexity of the $L^p$-norm imply 
\begin{equation}\label{eq:convex}
\mathsf{mass}_{a,p}  (\mathcal R)\leq \mathsf{mass}_{a,1} (\mathcal R) ^{1-\frac{2}{p'}}\mathsf{mass}_{a,2}  (\mathcal R) ^\frac{2}{p'} 
\end{equation}
for all $1\leq p \leq 2$, with $p'$ its dual exponent.
\end{remark}
We are now ready to state the main result of this section. The result below should be interpreted as a reverse H\"older-type bound for the balayages of directional Carleson sequences.
\begin{theorem} 
\label{thm:carleson}Let $S\subset[-1,1]$ be a finite set of $N$ slopes and $\mathcal R\subset  \mathcal D_S ^2$. Suppose that    the maximal operators $\{\M_{\mathcal R_s}: s\in S \}$ satisfy
\[
\sup_{s\in S}\big\| \M_{\mathcal R_s}: \, L^p\to L^{p,\infty}\big\|\lesssim (p')^\upgamma,\qquad p\to 1^+
\] 
for some $\upgamma\geq 0$. 
Then for every $L^\infty$-normalized Carleson sequence $a=\{a_R\}_{R\in\mathcal D_S ^2}$   
	\[
	\mathsf{mass}_{a,2}(\mathcal R)   \lesssim    (\log N)^{\frac12} \big((1+\upgamma)\log\log N\big)^{\frac\upgamma2}\mathsf{mass}_{a,1}(\mathcal R) ^\frac12.
	\]
\end{theorem}
The proof of Theorem \ref{thm:carleson} occupies the next subsection. The argument relies on several lemmata, whose proof is postponed to the final Subsection \ref{ss2:lp}.
\begin{remark} There are essentially two cases in the assumption of Theorem~\ref{thm:carleson} above. If for each $s\in S$ the family $\mathcal R_s$ happens to be a one-parameter family, then the corresponding maximal operator $\M_{\mathcal R_s}$ is of weak-type $(1,1)$, whence the assumption holds with $\upgamma=0$. In the generic case that $\mathcal R=\mathcal D_S ^2$ then for each $s$ the operator $\M_{\mathcal R_s}=\M_{\mathcal D^2 _s}$ is a skewed copy of the strong maximal function and the assumption holds with $\upgamma=1$. 	
\end{remark}

\subsection{Main line of proof of Theorem \ref{thm:carleson}} \label{ss2:mlp} Throughout the proof, we use  the following partial order between parallelograms $Q,R\in\mathcal D_S ^2$:
 \begin{equation}\label{eq:partorder}
 Q\leq R \overset{\text{def}}{\iff}  Q\cap R\neq \varnothing,\,\uppi_1(Q)\subseteq  \uppi_1(R).
 \end{equation}
Notice  that, since $Q,R\in\mathcal D_S ^2$,  we have  that $\uppi_1(R),\uppi_1(Q)$ belong to the standard dyadic grid $\mathcal D$  on $\R$.

It is convenient to encode the main inequality of Theorem \ref{thm:carleson} by means of  the following dimensionless quantity associated with a collection $\mathcal R\subset  \mathcal D_S ^2$ and a Carleson sequence $a=\{a_R\}_{R\in\mathcal D_S ^2}$
\[
{\mathsf U} _p (\mathcal R)  \coloneqq \sup_{\substack{\mathcal L\subset  \mathcal R\\ a=\{a_R\}}}\frac{\mass_{a,p}(\mathcal L)}{\mass_{a,1}(\mathcal L)^\frac1p} ,
\]
where the supremum is taken over all finite subcollections   $\mathcal L\subset  \mathcal R$ and all $L^\infty$-normalized Carleson sequences $a=\{a_R\}_{R\in\mathcal D_S ^2}$. 
There is  an easy, albeit lossy,  \emph{a priori} estimate for $\mathsf{U}_p(\mathcal R)$ for general $\mathcal R\subset  \mathcal D_S ^2$.
\begin{lemma}\label{l:apriori} Let $S\subset[-1,1]$ be a finite set of $N$ slopes and  $a=\{a_R\}_{R\in\mathcal R}$ be a normalized Carleson sequence as above. For every $\mathcal R\subset  \mathcal D_S ^2$ we have the estimate
	\[
   \mathsf U_p(\mathcal R)\lesssim N^{\frac{1}{p'}} \sup_{s\in S}\big\|\M_{\mathcal R_s}:L^{p'}\to L^{p',\infty}\big\|,\qquad 1<p<\infty.
	\]
\end{lemma}

Theorem \ref{thm:carleson} is then an easy consequence of the following bootstrap-type estimate. For an arbitrary finite collection of parallelograms $\mathcal R\subset \mathcal D_S ^2$  we will prove the estimate
\begin{equation}\label{eq:iter}
\mathsf U_2(\mathcal R) ^2 \lesssim (\log \mathsf U_2(\mathcal R) )^\upgamma \log N
\end{equation}
with absolute implicit constant. Note also that the boundedness assumption on $\M_{\mathcal R_s}$ for some $p<2$ and  Lemma~\ref{l:apriori}  yield the  \emph{a priori} estimate $\mathsf U_2 (\mathcal R) \lesssim N^{1/2}$. Inserting this  \emph{a priori} estimate into \eqref{eq:iter} and bootstrapping will then complete the proof of Theorem \ref{thm:carleson}. It thus suffices to prove \eqref{eq:iter} to obtain Theorem \ref{thm:carleson}.

The remainder of the section is dedicated to the proof of \eqref{eq:iter}. We begin by expanding the square of the $L^2$-norm of $T_\mathcal R(a)$ as follows:
\begin{equation}
\label{e:badnessnorm}
\mathsf{mass}_{a,2}(\mathcal R)^2 =
	\|T_{\mathcal R}(a)\|_2 ^2 \leq 2   \sum_{R\in\mathcal R} a_R \frac{1}{|R|}\int_{R} \sum_{\substack{{Q}\in\mathcal R\\{Q}\leq R}}a_{{Q}}\frac{\ind_{{Q}}}{|{Q}|}\eqqcolon 2 \sum_{R\in\mathcal R}a_R  B_R ^{\mathcal R}.
\end{equation}
For any $\mathcal L \subset  \mathcal R$ and $R\in\mathcal R$ we have implicitly defined
\begin{equation}
\label{eq:brL}
B_R ^{\mathcal L } \coloneqq  \frac{1}{|R|}\int_{R} \sum_{\substack{{Q}\in\mathcal L\\{Q}\leq R}} a_{{Q}}\frac{\ind_{{Q}}}{|Q|}.
\end{equation}

\begin{remark}\label{rmrk:BRbig} Observe that for any $\mathcal L\subset  \mathcal R$ and every fixed $s\in S$ we have
	\[
	\bigcup \left\{R\in\mathcal R_s:\, B_R ^{\mathcal L}>\uplambda\right\}\subset  \left\{x\in \R^2:\M_{\mathcal R_s}\bigg[\sum_{Q\in\mathcal L}  a_Q\frac{\ind_Q}{|Q|}\bigg](x)>\uplambda\right\}
	\]
which by our assumption on the weak $(p,p)$ norm of $\M_{\mathcal R_s}$ implies
\[
\sup_{s\in S}\Big|\bigcup \left\{R\in\mathcal R_s:\, B_R ^{\mathcal L}>\uplambda\right\}\Big|\lesssim (p')^\upgamma \frac{ \mathsf{mass}_{a,p}(\mathcal L) ^p}{\lambda^p},\qquad p\to 1^+.
\]
\end{remark}

For a numerical constant $\uplambda\geq 1$, to be chosen at the end of the proof, a nonnegative integer $k$ and $s\in S$ we consider subcollections of $\mathcal R_s$  as follows
\begin{equation}\label{eq:rsk}
\mathcal R_{s,k}\coloneq\left\{R:\, R\in\mathcal R_s,\: \uplambda k\leq  B_R ^{\mathcal R } < \uplambda(k+1)\right\},\qquad k\in\mathbb N,\quad s\in S .
\end{equation}
Using  \eqref{e:badnessnorm} we have
\begin{equation}\label{eq:main_split}
\begin{split}
\|T_{\mathcal R}(a)\|_2 ^2  & \lesssim \sum_{s\in S}\sum_{k=0} ^{  N} k\uplambda \sum_{R\in\mathcal R_{s,k}} a_R + N \sup_{s\in S} \Big[\sum_{k>N} k\uplambda \sum_{R\in\mathcal R_{s,k}} a_R\Big]
\\
&\qquad  \lesssim   \uplambda  (\log N) \mathsf{mass}_{a,1}(\mathcal R)+ \uplambda N \sum_{k> N} k \sup_{s\in S}   |\mathsf{sh}(\mathcal R_{s,k})|.
\end{split}
\end{equation}
Here $\uplambda>0$ is the constant   used to define the collections $\mathcal R_{s,k}$ and in the last lines we used the definition of a Carleson sequence and Remark~\ref{rmrk:dircarl}.

The following lemma encodes the exponential decay relation between    mass and $ B^{\mathcal L}_R$ and is in fact the main step of the proof  of Theorem~\ref{thm:carleson}.

\begin{lemma} \label{lem:iter} Let  $a=\{a_R:R\in \mathcal D_S ^2\}$ be an $L^\infty$-normalized Carleson sequence, $  S\subset [-1,1]$, and $\mathcal L,\mathcal R \subset  \mathcal D_S ^2$ with $\mathcal L \subseteq \mathcal R$. We assume that for some $p\in[1,2)$
	\[
A_p\coloneqq	\sup_{s\in S} \|\M_{\mathcal R_s}:L^p\to L^{p,\infty}\|<+\infty.
	\]
If  $\,\uplambda \geq C\max(1,A_p \mathsf{U}_2(\mathcal L) ^{\frac{2}{p'}})$ for a sufficiently large numerical constant $C>1$ then there exists $\mathcal L_1 \subset  \mathcal L $ such that:
\begin{itemize}
	\item [\emph{(i)}]  $\mathsf{mass}_{a,1}({\mathcal L_1})\leq  \frac{1}{2} \mathsf{mass}_{a,1}({\mathcal L}) $;
	\item[\emph{(ii)}] fixing $s\in S$ and denoting by
	 $\mathcal R_{s} '$  the collection of rectangles $R$ in $\mathcal R_s$ with $B^{\mathcal L} _R>\uplambda$, cf.\ \eqref{eq:brL}, we have that 
\[
B^{\mathcal L} _R \leq \uplambda +B^{\mathcal L_1} _{R}\qquad \forall R\in\mathcal R_s '.
\]
\end{itemize}
\end{lemma}

The final lemma we make use of in the argument translates the exponential decay of the mass of each $\mathcal R_{s,k}$  into exponential decay of the support size, which is what we need in the estimate \eqref{eq:main_split}.

\begin{lemma}\label{lem:shadow} Let $S\subset[-1,1]$ and define the collections $\mathcal R_{s,k}$ by \eqref{eq:rsk} with $\uplambda$ defined as in Lemma~\ref{lem:iter} for $\mathcal L= \mathcal R$
\[
\uplambda\coloneqq	 C \max\Big(1, A_p \mathsf{U} _2(\mathcal R) ^\frac{2}{p'}\Big).
\]	
We assume that the operators $\{\M_{\mathcal R_s}:\, s\in S\}$ map $L^p(\R^2)$ to $L^{p,\infty}(\R^2)$ uniformly with constant $A_{p}$. For $k\geq 1$ we then have the estimate
	\[
	|\mathsf{sh}(\mathcal R_{s,k})|\lesssim 2^{-k} \mathsf{mass}_{a,1}(\mathcal R) 
	\]
with absolute implicit constant.
\end{lemma}

With these lemmata in hand we now return to the proof of \eqref{eq:iter}. Substituting the estimate of Lemma~\ref{lem:shadow} into \eqref{eq:main_split} yields
\[
\|T_{\mathcal R}(a)\|_{2} ^2 \lesssim   \uplambda \mathsf{mass}_{a,1}(\mathcal R) \bigg[(\log N) + N\sum_{k\geq \log N}k2^{-k}\bigg]  \lesssim   \uplambda \mathsf{mass}_{a,1}(\mathcal R)(\log N).
\]
This was proved for an arbitrary collection $\mathcal R$ and so also for every $\mathcal L\subset  \mathcal R$. Thus the estimate above and our assumption $A_p\lesssim (p')^\upgamma$ imply
\[
\mathsf{U}_2(\mathcal R) ^2 \lesssim \uplambda  (\log N),\qquad \uplambda \gtrsim \max(1,(p')^\upgamma \mathsf{U}_2(\mathcal R) ^{\frac{2}{p'}}).
\]
Now observe that we can assume that $ \mathsf{U}_2(\mathcal R) \gtrsim 1$ otherwise there is nothing to prove. In this case we can take 
\[
\uplambda \simeq (p')^\upgamma \mathsf{U}_2(\mathcal R) ^\frac{2}{p'}
\]
for every $p>1$.  The choice  $p'\coloneqq (\log \mathsf{U}_2(\mathcal R))$ guarantees that $[\mathsf{U}_2(\mathcal R)]^{\frac{1}{p'}}\lesssim 1$   and leads to
\[
\mathsf{U}_2(\mathcal R) ^2 \lesssim (\log \mathsf{U}_2(\mathcal R))^\upgamma \log N.
\]
This is the desired estimate \eqref{eq:iter} and so the proof of Theorem~\ref{thm:carleson} is complete.

\subsection{Proof of the lemmata} \label{ss2:lp}

\begin{proof}[Proof of Lemma \ref{l:apriori}] We follow the proof of \cite{LR}*{Lemma 3.11}. Take $\RR$ to be some finite collection  and $\|g\|_{p'}=1$ such that 
\[
\Big\|\sum_{R\in\RR}a_R \frac{\ind_R}{|R|}\Big\|_p=\int \sum_{R\in\RR}a_R \frac{\ind_R}{|R|}g.
\]
Define $\RR':=\{R\in\RR: \, \l g\r_R > [cN/ \mass_{a,1}   (\RR)]^{1/{p'}}\}$ for some $c>1$  and $\RR'_s\coloneqq \RR'\cap \mathcal D_s ^2$ for $s\in S$. Then, 
\[
\begin{split}
\int \sum_{R\in\RR} a_R \frac{\ind_R}{|R|} g  &\le \sum_{R\in\RR\setminus\RR'}a_R  \l g\r_R+ \Big\|\sum_{R\in\RR'}a_R \frac{\ind_R}{|R|}\Big\|_p
\\
& \le  (cN)^{\frac{1}{p'}}(\sum_{R\in\RR } a_R)^{\frac1p}+ N \sup_{s\in S}\Big\|\sum_{R\in\RR'_s}a_R\frac{\ind_R}{|R|}\Big\|_p.
\end{split}\] 
This means
\[
\Big\|\sum_{R\in\RR}a_R \frac{\ind_R}{|R|}\Big\|_p\lesssim (c N)^{\frac{1}{p'}} \Big(1+\frac{N^{\frac1p}}{c^{\frac{1}{p'}}} \sup_{s\in S} \frac{\Big\|\sum_{R\in\RR'_s}a_R\frac{\ind_R}{|R|}\Big\|_p}{(\sum_{R\in\RR' _s} a_R)^{\frac1p}}	 \frac{(\sum_{R\in\RR' _s} a_R)^{\frac1p}}{(\sum_{R\in\RR _s} a_R)^{\frac1p}}\Big)(\sum_{R\in\RR } a_R)^{\frac1p}.
\]
We have proved that for an arbitrary collection $\mathcal R$ we have
\[
\mathsf{U}_p(\mathcal R) \leq (cN)^{\frac{1}{p'}}\bigg(1+\frac{N^{\frac1p}}{c^{\frac{1}{p'} }} \sup_s\mathsf{U}_p(\mathcal R_s ') \frac{\mass_{a,1}(\mathcal R_s ')^\frac{1}{p}}{\mass_{a,1}(\mathcal R)^{\frac1p}}\bigg).
\]

We claim that $\sup_{s\in S}\mathsf{U}_p(\RR_s ')\lesssim \sup_{s\in S}\|\M_{\RR_s}:L^{p'}\to L^{p',\infty}\|$. Assuming  this for a moment and using Remark~\ref{rmrk:dircarl} we can estimate
\[
\begin{split}
\sum_{R\in\RR' _s} a_R &\le |\mathsf{sh}(\RR'_s)|\le \big|\big\{\M_{\RR_s}(g)>(cN/\mass_{a,1}(\RR))^{1/{p'}}\big\}\big|
 \\
&\leq   \sup_{s\in S}\big\|\M_{\mathcal R_s}:L^{p'}\to L^{p',\infty} \big\|^{p'}
\frac{\mass_{a,1}(\RR)}{c N}.
\end{split}
\]
This proves the proposition upon choosing $c \gtrsim  \sup_{s\in S}\big\|\M_{\mathcal R_s}:L^{p'}\to L^{p',\infty} \big\|^{p'}$.

We have to prove the claim. Note that since $\RR' _s$ is a collection in a fixed direction the inequality $\mathsf{U}_{\RR'_s} \lesssim \sup_{s\in S}\|\M_{\RR_s}:L^{p'}\to L^{p',\infty}\|$ follows by the John-Nirenberg inequality in the product setting and Remark~\ref{rmrk:dircarl}; see \cite{LR}*{Lemma 3.11}. 
\end{proof}

\begin{proof}[Proof of Lemma \ref{lem:iter}]   By the invariance under shearing of our statement, we can work in the case $s=0$. Therefore,    $\mathcal R_0 '$ will stand for the collection of rectangles in $\mathcal R_0$ such that $B_R ^\mathcal L>\uplambda$, where $\uplambda\geq C$  and  $C>1$  will be  specified at the end of  the proof. We write $R=I_R \times L_R$ for $R\in \mathcal R_0$.

\subsubsection*{Inside-outside splitting}For $I\in \{\uppi_1(R): R \in \mathcal R_0'\}$  and any interval $K$ we define
	\[
	\begin{split}
		\mathcal L_{I,K} ^{\mathrm{in}} \coloneqq\{Q\in\mathcal L:Q\leq I\times K, \,\uppi_2(Q)\subset  3K\},\qquad \mathcal L_{I,K} ^{\mathrm{out}} \coloneqq\{Q\in\mathcal  L:Q\leq I\times K, \,\uppi_2(Q)\nsubseteq 3K\},
	\end{split}
	\]
where we recall that the definition of partial order $Q\leq R$ was given in \eqref{eq:partorder}. Set also
\[
B_{I,K} ^{\mathrm{in}}\coloneqq \avgint_{I\times K}  \sum_{Q\in \mathcal L_{I,K} ^{\mathrm{in}} }\frac{a_Q}{|Q|}\ind_{Q},
\quad
B_{I,K} ^{\mathrm{out}}\coloneqq \avgint_{I\times K}  \sum_{Q\in \mathcal L_{I,K} ^{\mathrm{out}} }\frac{a_Q}{|Q|}\ind_{Q}.
\]

We claim that if $K\subset \R$ is any interval then for all $\upalpha\in K$ we have
\begin{equation}
\label{eq:slice}
\begin{split}
	&\avgint_{I\times\{\upalpha\}} \sum_{Q\in\mathcal L^{\mathrm{out}} _{I,K} } a_Q \frac{\ind_Q}{|Q|}  =  \sum_{Q\in\mathcal L^{\mathrm{out}} _{I,K} } a_Q\frac{|Q\cap (I \times \{\upalpha\})|}{|Q|} 
	 \lesssim \avgint_{I\times 3K} \sum_{Q\in\mathcal L^{\mathrm{out}} _{I,K} } a_Q \frac{\ind_Q}{|Q|}.
\end{split}
\end{equation}
To see this note that in order for a $Q$-term appearing in the sum of the left hand side above  to be non-zero we must have
\[
\uppi_1(Q)\subset  I, \qquad \uppi_2(Q) \cap K \neq \varnothing, \qquad \uppi_2(Q) \cap \R \setminus 3 K  \neq \varnothing.
\]
Let us  write $\uptheta_Q=\arctan \sigma$ if $Q\in\mathcal D^2 _{\sigma}$ for some $\sigma \in S$. A computation then reveals that
\[
|Q\cap (I\times\{\upalpha\})|=\min(|J_Q|,\dist(\upalpha,\R\setminus \uppi_2(Q) ) \cot \uptheta_Q 
\] 
We also observe that $\uppi_2(Q) \cap (3K \setminus K)$ contains an interval $A=A(\upalpha)$ of length $|K|/3$, whence for all $\upalpha'\in A$ we have
\[
\dist(\upalpha,\R\setminus \uppi_2(Q) )\leq \dist(\upalpha,\upalpha') +\dist(\upalpha',\R\setminus \uppi_2(Q) ) \lesssim |K|+\dist(\upalpha',\R\setminus \uppi_2(Q) )\lesssim \dist(\upalpha',\R\setminus \uppi_2(Q) );
\]
see Figure~\ref{fig:slice}. This clearly implies that for every $\upalpha\in K$ we have
\[
|Q\cap (I\times\{\upalpha\})| \lesssim \avgint_{A} |Q\cap (I\times\{\upalpha'\})|\,\d \upalpha'\lesssim \avgint_{3K} |Q\cap (I\times\{\upalpha'\})|\,\d \upalpha'
\]
which proves the claim.

\subsubsection*{Smallness of the local average.} We now use the previously obtained \eqref{eq:slice} to prove (ii).   Let $\mathcal R_0 ^\star$ denote the family of parallelograms $R=I_R\times L_R\in\mathcal R_0 '$ such that $B_{I_R,L_R} ^{\mathrm{out}}>{\uplambda}$. For each such $R$ let $K_R$ be the maximal interval $K\in\{L_R,3L_R,\ldots,3^kL_R,\ldots\}$ such that $B_{I_R,K} ^{\mathrm {out}}>{\uplambda}$; the existence of the maximal interval $K_R$ is guaranteed for example by the  \emph{a priori} estimate of Lemma~\ref{l:apriori} and the assumption $R\in\mathcal R_0 ^\star$. Obviously $K_R\supseteq  L_R$ and $B_{I_R,3K_R} ^{\mathrm{out}}\leq \uplambda$.

We show that for $R\in\mathcal R_0 ^\star$ we have
\begin{equation}\label{eq:Ravg}
\avgint_R \sum_{Q\in\mathcal L^{\mathrm{out}} _{I_R,K_R} } a_Q \frac{\ind_Q}{|Q|} \leq \upkappa \uplambda
\end{equation}
for some numerical constant $\upkappa\geq 1$. 
\begin{figure}[t]
\centering
\def\svgwidth{330pt}
\begingroup%
  \makeatletter%
  \providecommand\color[2][]{%
    \errmessage{(Inkscape) Color is used for the text in Inkscape, but the package 'color.sty' is not loaded}%
    \renewcommand\color[2][]{}%
  }%
  \providecommand\transparent[1]{%
    \errmessage{(Inkscape) Transparency is used (non-zero) for the text in Inkscape, but the package 'transparent.sty' is not loaded}%
    \renewcommand\transparent[1]{}%
  }%
  \providecommand\rotatebox[2]{#2}%
  \newcommand*\fsize{\dimexpr\f@size pt\relax}%
  \newcommand*\lineheight[1]{\fontsize{\fsize}{#1\fsize}\selectfont}%
  \ifx\svgwidth\undefined%
    \setlength{\unitlength}{482.14248709bp}%
    \ifx\svgscale\undefined%
      \relax%
    \else%
      \setlength{\unitlength}{\unitlength * \real{\svgscale}}%
    \fi%
  \else%
    \setlength{\unitlength}{\svgwidth}%
  \fi%
  \global\let\svgwidth\undefined%
  \global\let\svgscale\undefined%
  \makeatother%
  \begin{picture}(1,0.77448698)%
    \lineheight{1}%
    \setlength\tabcolsep{0pt}%
    \put(0,0){\includegraphics[width=\unitlength,page=1]{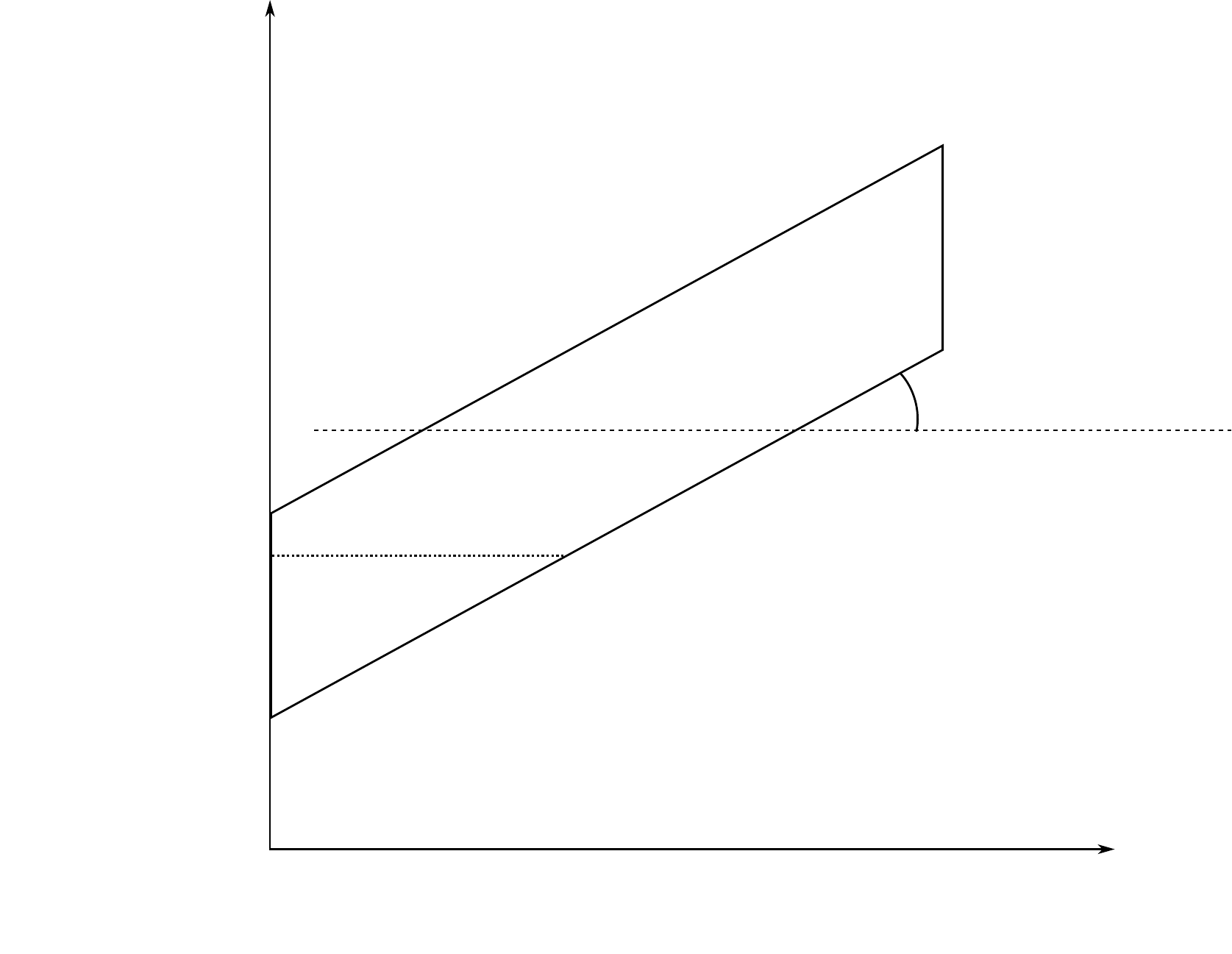}}%
    \put(0.17369807,0.31981504){\color[rgb]{0,0,0}\makebox(0,0)[lt]{\lineheight{0}\smash{\begin{tabular}[t]{l}$\upalpha'$\end{tabular}}}}%
    \put(0.75277102,0.44643419){\color[rgb]{0,0,0}\makebox(0,0)[lt]{\lineheight{0}\smash{\begin{tabular}[t]{l}$\uptheta_Q$\end{tabular}}}}%
    \put(0,0){\includegraphics[width=\unitlength,page=2]{rectangles.pdf}}%
    \put(0.61857882,0.73095941){\color[rgb]{0,0,0}\makebox(0,0)[lt]{\lineheight{0}\smash{\begin{tabular}[t]{l}$Q$\end{tabular}}}}%
    \put(0.46980148,0.09069451){\color[rgb]{0,0,0}\makebox(0,0)[lt]{\lineheight{0}\smash{\begin{tabular}[t]{l}$I$\end{tabular}}}}%
    \put(0,0){\includegraphics[width=\unitlength,page=3]{rectangles.pdf}}%
    \put(0.56117215,0.34727745){\color[rgb]{0,0,0}\makebox(0,0)[lt]{\lineheight{0}\smash{\begin{tabular}[t]{l}$R=I\times L$\end{tabular}}}}%
    \put(0,0){\includegraphics[width=\unitlength,page=4]{rectangles.pdf}}%
    \put(0.16946619,0.17990661){\color[rgb]{0,0,0}\makebox(0,0)[lt]{\lineheight{0}\smash{\begin{tabular}[t]{l}$K$\end{tabular}}}}%
    \put(0.2456552,0.34017265){\color[rgb]{0,0,0}\makebox(0,0)[lt]{\lineheight{0}\smash{\begin{tabular}[t]{l}$Q\cap (I\times\{\upalpha'\})$\end{tabular}}}}%
    \put(0,0){\includegraphics[width=\unitlength,page=5]{rectangles.pdf}}%
    \put(-0.00188368,0.19864912){\color[rgb]{0,0,0}\makebox(0,0)[lt]{\lineheight{0}\smash{\begin{tabular}[t]{l}$3K$\end{tabular}}}}%
    \put(0,0){\includegraphics[width=\unitlength,page=6]{rectangles.pdf}}%
    \put(0.22698852,0.2601726){\color[rgb]{0,0,0}\makebox(0,0)[lt]{\lineheight{0}\smash{\begin{tabular}[t]{l}$Q\cap (I\times\{\upalpha\})$\end{tabular}}}}%
    \put(0.17680919,0.24203717){\color[rgb]{0,0,0}\makebox(0,0)[lt]{\lineheight{0}\smash{\begin{tabular}[t]{l}$\upalpha$\end{tabular}}}}%
  \end{picture}%
\endgroup%

\caption{A rectangle $Q$ with angle $\uptheta_Q$ intersecting $R=I\times L\subset  I\times K$.}
\label{fig:slice}
\end{figure}
Indeed it is a consequence of \eqref{eq:slice} that
\[
\begin{split}
	 &\qquad \avgint_{I_R\times\{\upalpha\}} \sum_{Q\in\mathcal L^{\mathrm{out}} _{I_R,K_R} } a_Q \frac{\ind_Q}{|Q|} \lesssim \avgint_{I_R\times 3K_R} \sum_{Q\in\mathcal L^{\mathrm{out}} _{I_R,K_R} } a_Q \frac{\ind_Q}{|Q|}
	\\
	& \leq \avgint_{I_R\times 3K_R} \sum_{Q\in\mathcal L^{\mathrm{out}} _{I_R,3K_R} } a_Q\frac{\ind_Q}{|Q|} +\avgint_{I_R\times 3K_R} \sum_{Q\in\mathcal L^{\mathrm{out}} _{I_R, K_R}\setminus \mathcal L^{\mathrm{out}} _{I_R,3K_R}} a_Q \frac{\ind_Q}{|Q|}. 
	\end{split}
\]
The  first summand is estimated using the maximality of $K_R$
\[
\avgint_{I_R\times 3K_R} \sum_{Q\in\mathcal L^{\mathrm{out}} _{I_R,3K_R} } a_Q\frac{\ind_Q}{|Q|} =B_{I_R,3K_R} ^{\mathrm{out}}\leq \uplambda.
\]
The second summand can be further analyzed by observing that the cubes $Q$ appearing in the sum above satisfy $\uppi_1(Q)\subset  I$ and $\uppi_2(Q)\subset  9K_R$ since $Q\notin \mathcal L_{I_R,3K_R} ^{\mathrm{out}}$, that is, $\mathcal L^{\mathrm{out}} _{I_R,3K_R}\setminus \mathcal L^{\mathrm{out}} _{I_R,K_R}$ is subordinate to the singleton collection $\{I_R\times 9K_R\}$. Applying the Carleson sequence property
\[
\begin{split}
&\quad \avgint_{I_R\times 3K_R} \sum_{Q\in\mathcal L^{\mathrm{out}} _{I_R,K_R}\setminus \mathcal L^{\mathrm{out}} _{I_R,3K_R}} a_Q \frac{\ind_Q}{|Q|}\leq   \sum_{ Q\in\mathcal L_{I_R,K_R} ^{\mathrm{out}}\setminus \mathcal L_{I_R,3K_R} ^\mathrm{out} }  a_Q \frac{|Q\cap (I_R\times 3K_R)|}{|Q||I_R\times 3K_R|}\lesssim 1 \leq \uplambda
\end{split}
\]
by our assumption on $\uplambda$. Combining the estimates above shows that
\[
  \avgint_{I_R\times\{\upalpha\}} \sum_{Q\in\mathcal L^{\mathrm{out}} _{I_R,K_R} } a_Q \frac{\ind_Q}{|Q|} \lesssim \uplambda
\]
for all $\upalpha\in K_R$. Since $\pi_2(R)\subset  K$ this implies \eqref{eq:Ravg}.

Observe that if $R=I_R\times L_R\in\mathcal R_0 ' \setminus \mathcal R_0 ^\star$ then
\begin{equation}\label{eq:except}
B_{I_R,L_R} ^{\mathrm{out}}= \avgint_{I_R\times L_R} \sum_{Q\in\mathcal L^{\mathrm{out}} _{I_R,L_R} } a_Q \frac{\ind_Q}{|Q|} \leq\uplambda.
\end{equation}

\subsubsection*{Defining the subcollection $\mathcal L_1$} We set
\[
\mathcal L_1 {'} \coloneqq \bigcup_{R\in \mathcal R_0 ^\star} \mathcal L_{I_R,K_R} ^{\mathrm{in}},\quad \mathcal L_1 {''}\coloneqq \bigcup_{R\in\mathcal R_0 {'} \setminus \mathcal R_0 ^\star}  \mathcal L_{I_R,L_R} ^{\mathrm{in}},\qquad \mathcal L_1 \coloneqq \mathcal L_1 {'} \cup \mathcal L_1 {''}.
\]
Now note that for each $R\in \mathcal R_0 ^\star$  and $K=K_R\in \mathcal K_{\uppi_1(R)}$ we have that
\[
B_R ^{\mathcal L} \leq \avgint_R \sum_{Q\in\mathcal L_{I_R\times K_R} ^{\mathrm{out}}}a_Q\frac{\ind_Q}{|Q|}+\avgint_R \sum_{Q\in\mathcal L_{I_R\times K_R} ^{\mathrm{in}}}a_Q\frac{\ind_Q}{|Q|}\leq \upkappa \uplambda+B_R ^{\mathcal L_1}
\]
while for $R\in \mathcal R_0 ' \setminus \mathcal R_0 ^\star$ the same estimate holds using $L_R$ in place of $K_R$. It remains to show the desired estimate for $\mass_{a,1}(\mathcal L_1)$ in (i) of the lemma.

\subsubsection*{Smallness of  the mass $\,\mathsf{mass}_{a,1}(\mathcal L_1)$.} By the definition of the collections $\mathcal L_{I,K} ^{\mathrm{in}}$  we have that
\[
\mathsf{sh}(\mathcal L_1)\subset  \bigcup_{R\in\mathcal R_0 ^\star} I_R\times 3K_R \cup \bigcup_{R\in\mathcal R_0 '\setminus \mathcal R_0 ^\star} I_R\times 3L_R.
\]
If $K=K_R$ for some $R\in\mathcal R_0 ^\star$ we have by definition that $B_{I_R,K_R} ^{\mathrm{out}}>\uplambda$. On the other hand  for $R \in  \mathcal R_0 '\setminus \mathcal R_0 ^\star$ we have that $B_R ^{\mathcal L}=B_{I_R,L_R} ^{\mathcal L}>\uplambda$.

Define 
\[
E\coloneqq \bigg\{(x,y)\in \R^2:\, \M_v \bigg[  \sum_{Q\in\mathcal L} a_Q \frac{\ind_Q}{|Q|}\bigg](x,y)\geq \frac12 \uplambda\bigg\}
\]
where $\M_v=\M_{(1,s)}=\M_1$ is the directional Hardy-Littlewood maximal function acting on the direction $v=(1,s)=(1,0)$, see \eqref{eq:dirMv}, since we have assumed $s=0$. We will show that 
\[
\bigcup_{\substack{R\in\mathcal R_0^\star }} I_R\times 3K_R  \subset  \big\{(x,y)\in\R^2:\, \M_2 (\ind_E)(x,y)\geq C\big\}
\]
for a sufficiently small constant $C>0$, where $\M_2$ is as in \eqref{eq:dirMv}.  To this end let us define 
\[
\uppsi(\upalpha)\coloneqq   \frac{1}{|I_R|}\int_{I_R\times\{\upalpha\}} \sum_{Q\in\mathcal L^{\mathrm{out}} _{I_R,K_R} }a_Q \frac{\ind_Q}{|Q|}.
\]
Note that  
\[
\begin{split}
& \uplambda<B_{I_R,K_R} ^{\mathrm{out}}=\avgint_{K_R} \uppsi(\upalpha) \,\d \upalpha\leq \frac{1}{|K_R|}\int_{\{K_R:\,\uppsi(\upalpha)>\uplambda/2\}}\uppsi(\upalpha)\,\d \upalpha+\frac{\uplambda}{2} 
\leq \frac{c\uplambda}{|K_R|}|\{K_R:\, \uppsi(\upalpha)>\uplambda/2\}|+\frac\uplambda2
\end{split}
\]
which readily yields the existence of $K'\subset K_R$ with 
\[ 
|K_R|\lesssim |K'|, \qquad \inf_{x\in I_R}\inf_{y\in K'}\M_v \bigg[ \sum_{Q\in\mathcal L^{\mathrm{out}} _{I_R,K_R} } a_Q \frac{\ind_Q}{|Q|}\bigg](x,y)>\frac\uplambda 2.
\]
This in turn implies that $\M_2(\ind_E)\gtrsim 1$ on $I_R\times 3K_R$.  Now we can conclude
\begin{equation}\label{eq:measure}
	\begin{split}
	 &\quad   \bigg |\bigcup_{\substack{R\in\mathcal R_0^\star }}  I_R\times 3K_R\bigg| \leq \big|  \{\M_2(\ind_E)\gtrsim 1\}\big| \lesssim  |E|  \lesssim\frac{1}{\uplambda} \mathsf{mass}_{a,1}({\mathcal L})
\end{split}
\end{equation}
by the weak $(1,1)$ inequality of the directional Hardy-Littlewood maximal function $\M_{(1,0)}$.

On the other hand we have for the rectangles $R\in \mathcal R_0 '\setminus \mathcal R_0 ^\star$ that
\[
\bigcup_{R\in\mathcal R_0 '\setminus \mathcal R_0 ^\star} I_R\times 3L_R\subset \Big\{\M_{\mathcal R_0}\Big(\sum_{Q\in\mathcal L} a_Q\frac{\ind_Q}{|Q|}\Big)>\frac{\uplambda}{3} \big)\Big\}.
\]
Thus we get by the weak $(p,p)$ assumption for $\M_{\mathcal R_0}$  that
\[
\begin{split}
&\bigg|\bigcup_{R\in\mathcal R_0 '\setminus \mathcal R_0 ^\star} I_R\times 3L_R\bigg|\le \bigg|\Big\{\M_{\mathcal R_0}\big(\sum_{Q\in\mathcal L} a_Q\frac{\ind_Q}{|Q|}>\frac{\uplambda}{3} \big)\Big\}\bigg|\lesssim \frac{A_p ^p}{\uplambda ^p}\mass_{a,p}(\mathcal L)
\\
&\qquad\qquad\lesssim \frac{A_p ^p}{\uplambda^p}\mass_{a,1}(\mathcal L)\mathsf{U}_2(\mathcal L) ^{2(p-1)}.
\end{split}
\]
By the subordination property of $\mathcal L_1$ we get
\[
\mass_{a,1}(\mathcal L_1)\leq \bigg|\bigcup_{R\in\mathcal R_0 ^\star} I_R\times 3K_R \cup \bigcup_{R\in\mathcal R_0 '\setminus \mathcal R_0 ^\star} I_R\times 3L_R\bigg|\leq \frac12 \mass_{a,1}(\mathcal L)
\]
upon choosing $\uplambda \geq C \max(1,A_p\mathsf{U}_2 (\mathcal L)^\frac{2}{p'})$ with sufficiently large $C>1$.
\end{proof}

\begin{proof}[Proof of Lemma \ref{lem:shadow}] Fix $s\in S$ and choose $\uplambda$ in the definition of $\mathcal R_{s,k}$ to be the value given by Lemma \ref{lem:iter} with $\mathcal L=\mathcal R=\cup_{s\in S}\mathcal R_s$. Let  $j=0$ and $\mathcal L_0=\mathcal L_{j}\coloneqq \mathcal R$.    Construct $\mathcal L_{1}=\mathcal L_{j+1}\subset  \mathcal R$ such that $\mathsf{mass}_{a,1}({\mathcal L_1})\leq \frac 12 \mathsf{mass}_{a,1}(\mathcal L_0)$. Since $B_R ^{\mathcal L_0}>k\uplambda $ for all $R\in\mathcal R_{s,k}$ we have 
\[
\uplambda k<B_R ^{\mathcal L_0 }\leq \uplambda +B_R ^{\mathcal L_1}\implies B_R ^{\mathcal L_1}>\uplambda(k-1).
\]
Repeat  the procedure inductively with $j+1$ in place of $j$. When $j=k-1$  we have reached the collection $\mathcal L_{k-1}$ with $\mathsf{mass}_{a,1}({\mathcal L_{k-1}})\lesssim  2^{-k}\mathsf{mass}_{a,1}(\mathcal L_0)$  and $B_R ^{\mathcal L_{k-1}}>\uplambda$. This last condition and Remark~\ref{rmrk:BRbig} imply that
	\[
	\mathsf{sh}(\mathcal R_{s,k}) \subset  \bigg\{\M_{\mathcal R_s}\bigg[\sum_{Q\in\mathcal L_{k-1}}a_Q\frac{\ind_Q}{|Q|}\bigg]>\uplambda\bigg\}
	\]
and so, using \eqref{eq:convex},
\[
\begin{split}
& \big | \mathsf{sh}(\mathcal R_{s,k})\big|\leq \frac{A_p ^p}{\uplambda ^p}\mathsf{mass}_{a,p}({\mathcal L_{k-1}}) ^p \leq  \frac{A_p ^p}{\uplambda ^p} \mathsf{mass}_{a,1}(\mathcal L_{k-1}) ^{p-\frac{2p}{p'}} \mathsf{mass}_{a,2}(\mathcal L_{k-1}) ^{\frac{2p}{p'}}
\\
& \qquad \leq  2^{-k}\mathsf{mass}_{a,1}  (\mathcal L_0)  \frac{C A_p ^p  }{\uplambda^p}\Bigg(\frac{\mathsf{mass}_{a,2}(\mathcal L_0	)^2}{\mathsf{mass}_{a,1}(\mathcal L_0)}\Bigg)^{p-1} =  2^{-k}\mathsf{mass}_{a,1} (\mathcal L_0)   \frac{C A_p ^p  }{\uplambda^p}\mathsf{U}_2(\mathcal L_0)^{2(p-1)}
\end{split}
\]
and the lemma follows by the definition of $\uplambda$ since $\mathcal L_0=\mathcal R$.
\end{proof}

\section{A weighted Carleson embedding and applications to maximal directional operators} 
\label{s:wcarleson} In this section, we provide a weighted version of the directional Carleson embedding theorem. We then derive, as applications, novel weighted norm inequalities for maximal and singular directional operators. 

The proof of the weighted Carleson embedding follows the strategy used for Theorem~\ref{thm:carleson}, with suitable modifications. In order to simplify the presentation, we restrict our scope to collections of parallelograms $\mathcal R=\{\bigcup \mathcal R_s:s\in S\}$ with the property that the maximal operator $\M_{\mathcal R_s}$ associated to each collection $\mathcal R_s$ satisfies the appropriate weighted weak-$(1,1)$ inequality. This is the case, for instance, when the collections $\mathcal R_s$ are of the form 
\begin{equation}
\label{e:onepar2020}
\mathcal R_s \subset \mathcal D^2_{s,\cdot,k}\, , \qquad \mathcal D^2_{s,\cdot,k} \coloneqq
\bigcup_{k_1\leq k}
\mathcal D^2 _{s,k_1,k}
\end{equation}
for a fixed $k\in \mathbb Z$. In other words, the parallelograms in direction $s$ have fixed vertical sidelength and arbitrary eccentricity. 
\subsection{Directional weights} Let $S$ be a set of slopes and $w,u\in L^1 _{\mathrm{loc}}(\R^2)$ be nonnegative functions, which we refer to as \emph{weights} from now on. Our weight classes are related to the maximal operator 
\[\mathrm{M}_{S;2}\coloneqq \mathrm{M}_{V}\circ \M_{(0,1)},\] recalling that $M_V=M_{\{(1,s):s\in S\}}$ is the directional maximal operator defined in \eqref{e:Mv}.
 We introduce the two-weight directional constant
\[
[w,u]_{S}\coloneqq    \sup_{x\in \R^2} \frac{ \mathrm{M}_{S;2} w(x)}{u(x)}. 
\] 
We pause to point out some relevant examples of pairs $w,u$ with $[w,u]_{S}<\infty$.  Recall that for $p>2$, $\|\mathrm{M}_{S;2}\|_{p\to p} \lesssim (\log \#S)^{1/p}$;  this is actually a special case of Theorem \ref{thm:carleson} and interpolation. Therefore, if $g\geq 0$ belongs to the unit sphere of $L^p(\R^2)$,
\[
w\coloneqq \sum_{\ell=0}^\infty \frac{\mathrm{M}_{S;2}^{[\ell]} g}{2^\ell\left\|\mathrm{M}_{S;2}\right\|_{p\to p}^\ell} 
\] 
satisfies $[w,w]_{S}\leq 2 \|\mathrm{M}_{S;2} \|_{p\to p}$; here $T^{[\ell]}$ denotes $\ell$-fold composition of an operator $T$ with itself. We also highlight the relevance of $[w,u]_{S}$ in Theorem \ref{thm:wcarleson} below by noticing that
\begin{equation}
\label{e:weighted1bd}
\sup_{s\in S } \big\| \M_{\mathcal D^2_{s,\cdot,k}}: \, L^1(u)\to L^{1,\infty}(w)\big\| \lesssim  [w,u]_{S}
\end{equation}
with absolute implicit constant. This result is obtained via the classical Fefferman-Stein inequality in direction $s$ paired with the remark that $\M_{\mathcal D^2_{s,\cdot,k}} w \lesssim \M_{S; 2} w \leq   [w,u]_{S} u$.

\subsection{Weighted Carleson sequences} We begin with the weighted analogue of Definition~\ref{def:carleson}, which is given with respect to a fixed weight $w$. 
\begin{definition}\label{def:wcarleson}Let $a=\{a_R\}_{R\in\mathcal D_S ^2}$ be a sequence of nonnegative numbers. Then $a$ will be called an \emph{$L^\infty$-normalized} $w$-Carleson  sequence if for every $\mathcal L \subset \mathcal D_S ^2$ which is subordinate to some collection $\mathcal T \subset  \mathcal P^2_\uptau$ for some  fixed $\uptau\in S$, we have
	\[
\sum_{L\in\mathcal L}a_L\leq  w(\mathsf{sh}(\mathcal T )),
	\qquad
	 \mathsf{mass}_a \coloneqq \sum_{R\in\mathcal D_S ^2} a_R 
	<\infty.\]
\end{definition}

As before,  if $\mathcal R\subset \mathcal D_\uptau ^2$ for some fixed $\uptau\in S$ then $\mathcal R$ is subordinate to itself and
	\[
	\mathsf{mass}_{a,1} (\mathcal R)=\sum_{R\in\mathcal R }a_R\leq  w(\mathsf{sh}(\mathcal R )),\qquad \mathcal R \subset \mathcal D_\uptau ^2\quad\text{for some fixed}\,\, \uptau\in S.
	\]
Throughout this section all Carleson sequences and related quantities are taken with respect to some fixed weight $w$ which is suppressed from the notation.
We can now state our weighted Carleson embedding theorem.

\begin{theorem} 
\label{thm:wcarleson}Let $S\subset[-1,1]$ be a finite set of $N$ slopes  and $\mathcal R\subset \mathcal D_S ^2$. Let $w,u$ be weights with $[w,u]_{S}<\infty $ and such that  \[
\sup_{s\in S}\big\| \M_{\mathcal R_s}: \, L^1(u)\to L^{1,\infty}(w)\big\|\lesssim  [w,u]_S.
\] 
Then for every $L^\infty$-normalized $w$-Carleson sequence $a=\{a_R\}_{R\in\mathcal D_S ^2}$ we have
	\[
	\left(\int \left|T_{\mathcal R}(a)(x)\right|^2 \frac{\mathrm d x}{\M_{\mathcal R}u(x)}\right)^{\frac12}  \lesssim    (\log N)^{\frac12} [w,u]_S  \mathsf{mass}_{a,1}(\mathcal R) ^\frac12.
	\]
\end{theorem}

\subsection{Proof of Theorem~\ref{thm:wcarleson}} We follow the proof of Theorem~\ref{thm:carleson} and only highlight the differences to accommodate the weighted setting. Write $\upsigma\coloneqq [\M_{\mathcal R}u]^{-1}$. Expanding the $L^2(\upsigma)$-norm we have
\[
\|T_{\mathcal R}(a)\|_{L^2(\upsigma)}^2 \leq2\sum_{R\in\mathcal R}a_R  \sum_{\substack{Q\in\mathcal R\\Q\leq R}}a_Q \frac{\upsigma(Q\cap R)}{|Q||R|}.
\]
From the definition of $\upsigma$ we have that 
\[
\upsigma(Q\cap R)\leq  \frac{|Q\cap R|}{\inf_Q  \M_{\mathcal R}u}  \leq \frac{|Q|}{u(Q)}|Q\cap R	|
\]
whence
\[
\|T_{\mathcal R}(a)\|_{L^2(\upsigma)}^2 \leq2 \sum_{R\in\mathcal R}a_R  \avgint_R\sum_{\substack{Q\in\mathcal R\\Q\leq R}}a_Q\frac{\ind_Q}{u(Q)} \coloneqq 2 \sum_{R\in\mathcal R}a_R B_R ^{\mathcal R}
\]
where now for any $\mathcal L\subset  \mathcal R$ we have defined
\[
B_R ^{\mathcal L}\coloneqq \avgint_R \sum_{\substack{Q\in\mathcal L\\Q\leq R}} a_Q\frac{\ind_Q}{u(Q)}.
\]
Defining the families $\mathcal R_{s,k}$ for $s\in S$ and $k\in\N$ as in \eqref{eq:rsk} we then have the estimate
\[
\|T_{\mathcal R}(a)\|_{L^2(\upsigma)}^2\leq 2   \uplambda  \Big[(\log N) \mathsf{mass}_{a,1}(\mathcal R)+ N \sum_{k> \log N} k \sup_{s\in S}   w(\mathsf{sh}(\mathcal R_{s,k}))\Big].
\]
Again $\uplambda>0$ is a constant that will be determined later in the proof and in the last line we used the $w$-Carleson assumption for the sequence $a=\{a_R\}$ for rectangles in a fixed direction.

We need the weighted version of Lemma~\ref{lem:iter}, which is given under the standing assumptions of Theorem \ref{thm:wcarleson}.

\begin{lemma} \label{lem:witer} Let  $a=\{a_R:R\in \mathcal D_S ^2\}$ be an $L^\infty$-normalized $w$-Carleson sequence, $s\in S\subset [-1,1]$, and $\mathcal L,\mathcal R \subset  \mathcal D_S ^2 $ with $\mathcal L \subseteq \mathcal R$.   For every $\uplambda>C[w,u]_S$ where $C$ is a suitably chosen absolute constant, there exists $\mathcal L_1 \subset  \mathcal L $ such that:
\begin{itemize}
	\item [\emph{(i)}]  $\mathsf{mass}_{a,1}({\mathcal L_1})\leq  \frac{1}{2} \mathsf{mass}_{a,1}({\mathcal L}) $;
	\item[\emph{(ii)}] denoting by
	 $\mathcal R_{s} '$  the collection of rectangles $R$ in $\mathcal R_s$ with $B^{\mathcal L} _R>\uplambda$ we have that 
\[
B^{\mathcal L} _R \leq \uplambda +B^{\mathcal L_1} _{R}\qquad \forall R\in\mathcal R_s '.
\]
\end{itemize}
\end{lemma}

\begin{proof}   We can assume that $s=0$ and let $\mathcal R_0 '$ be the collection of rectangles in $\mathcal R_0$ such that $B_R ^\mathcal L>\uplambda$, where $\uplambda$ is as in the statement of the lemma and  $C$  will be  specified at the end of  the proof. For $I\in \{\uppi_1(R): R \in \mathcal R_0'\}$  and any interval $K\subset \R$ we define $\mathcal L_{I,K} ^{\mathrm{in}} $ and $\mathcal L_{I,K} ^{\mathrm{out}}$ as in the proof of Theorem~\ref{thm:carleson} but now we set
\[
B_{I,K} ^{\mathrm{in}}\coloneqq \avgint_{I\times K}   \sum_{Q\in \mathcal L_{I,K} ^{\mathrm{in}} } \frac{a_Q}{u(Q)}\ind_{Q},
\quad
B_{I,K} ^{\mathrm{out}}\coloneqq \avgint_{I\times K}  \sum_{Q\in \mathcal L_{I,K} ^{\mathrm{out}} }\frac{a_Q}{u(Q)}\ind_{Q}.
\]

We define   $\mathcal R_0 ''$ to be the subcollection of those $R=I\times L\in\mathcal R_0 '$ such that $B_{I,L} ^{\mathrm{out}}\leq \uplambda $. By linearity we get for each $R\in\mathcal R_0''$  that $B_R ^{\mathcal L}  \leq \uplambda+B_{I,L} ^{\mathrm{in}}\leq \uplambda +B_R ^{\mathcal L_1 ''}$ where 
\[
\mathcal L _1 {''} \coloneqq \bigcup_{R=I\times L\in\mathcal R_0 ''}\mathcal L_{I,L} ^{\mathrm{in}},\qquad 
\mathsf{sh}(\mathcal   L'' _1) \subset  \bigcup_{R=I\times L \in \mathcal R_0 ''} I\times 3L .
\]
Since $\mathcal R_0 ''\subset  \mathcal R_0 '$ we conclude as before that
\[
\begin{split}
&w \big (\mathsf{sh}(\mathcal L'' _1 )\big) \leq  w\bigg ( \bigcup_{R=I\times L\in\mathcal R_0 ''} I\times 3L  \bigg)\leq w\Big (\Big\{  \M_{\mathcal R_{0}} \Big(\sum_{Q\in\mathcal L}   \frac{a_Q\ind_Q}{u(Q)}\Big)>\frac{\uplambda}{3}\Big\}\Big) 
\\
&\qquad \lesssim  \frac{[w,u]_{S}}{\uplambda} \int_{\R^2} \sum_{Q\in\mathcal R} a_Q \frac{\ind_Q}{u(Q)}\,\d u = \frac{[w,u]_S}{\uplambda} \mass_{a,1}( \mathcal L)
\end{split}
\]
by the two-weight  weak type $(1,1)$ inequality for $\M_{\mathcal R_{s}}=\M_{\mathcal R_{0}}$. Now $\mathcal L'' _1$ is subordinate to the collection $\{I\times 3L:I\times L \in \mathcal R_0 '' \}$. Using the definition of a Carleson sequence we have
\[
\begin{split}
\sum_{Q\in\mathcal L'' _1} a_Q &\leq  w \bigg(\bigcup_{R=I\times L \in \mathcal R_0 ''} I\times 3L\bigg) \lesssim \frac{[w,u]_S}{\uplambda}\mass_{a,1}(\mathcal L)
\end{split}
\]
and so $ \mathsf{mass}_{a,1}({\mathcal L'' _1 })\lesssim [w,u]_S \mathsf{mass}_{a,1}({\mathcal L})/\uplambda$.

It remains to deal with parallelograms 
\[
R=I\times L 
\in \mathcal R^\star_0\coloneqq\mathcal R_0 '\setminus \mathcal R_0''\qquad B_{I,L} ^{\mathrm{out}}> \uplambda.
\] 
We define the maximal $K_R$ such that $B_{I,K_R} ^{\mathrm{out}}>\uplambda$ as before; the existence of this maximal interval can be guaranteed for example by assuming the collection $\mathcal R$ is finite. We have for each $R=I\times L\in \mathcal R^\star _0$ that  $B_{I,L} ^{\mathrm{out}}> \uplambda$  so  $K_R\supset  L$ and $B_{I,3K_R } ^{\mathrm{out}}\leq\uplambda$ by maximality.

Now using \eqref{eq:slice} we get that
\[
\Delta\coloneqq \sum_{Q\in\mathcal L^{\mathrm{out}} _{I,3K_R} } a_Q \frac{|Q\cap( I\times\{\upalpha\})|}{{u}(Q)|I|}\lesssim  \sum_{Q\in\mathcal L^{\mathrm{out}} _{I,3K_R} }a_Q \frac{|Q\cap (I\times 3K_R)|}{{u}(Q)||3K_R||I|}=\avgint_{I\times 3K_R}\sum_{Q\in \mathcal L^{\mathrm{out}} _{I,3K_R} }a_Q\frac{\ind_Q}{{u}(Q)}\lesssim \uplambda
\]
by the maximality of $K_R$. On the other hand
\[
\Xi\coloneqq \sum_{Q\in\mathcal L^{\mathrm{out}} _{I,K} \setminus \mathcal L^{\mathrm{out}} _{I,3K} }a_Q \avgint_{I\times\{\upalpha\}}   \frac{\ind_Q}{{u}(Q)} \lesssim  \sum_{Q\subset  I\times 9K} a_Q \frac{|Q\cap (I\times 3K)|}{|I\times 3K|{u}(Q)}.
\]
Since $\M_{\mathcal R_s} w\leq \M_V\M_2w\leq [w,u]_S {u}$ uniformly in $s$ we get that for $Q\subset  I\times 9K$
\[
{u}(Q) \gtrsim [w,u]_S ^{-1}\frac{w(I\times 9K)}{|I\times 9K|}|Q|
\]
and by this and the $w$-Carleson property for all $Q$s subordinate to $I\times 9K$ we get 
\[
\Xi\lesssim [w,u]_S \leq \uplambda
\]
provided $\uplambda \geq  [w,u]_S$.
We now define
\[
\mathcal L_1 {'} \coloneqq \bigcup_{R\in \mathcal R_0 ^\star} \mathcal L_{\uppi_1(R),K_R} ^{\mathrm{in}}
\]
so that 
\[
\mathsf{sh}(\mathcal L' _1)\subset  \bigcup_{R\in\mathcal R_0 ^\star} \uppi_1(R)\times K_R.
\]
Arguing as in the unweighted case of Theorem~\ref{thm:carleson} we can estimate
\[
w(\mathsf{sh}(\mathcal L' _1))\leq w\Bigg(\bigcup_{R\in\mathcal R_0 ^\star} \uppi_1(R)\times K_R\Bigg)\lesssim w(\{\M_2 (\ind_E)\gtrsim 1\})
\]
where
\[
E\coloneqq \bigg\{(x,y)\in \R^2:\, \M_v \bigg[  \sum_{Q\in\mathcal L} a_Q \frac{\ind_Q}{{u}(Q)}\bigg](x,y)\geq \frac12 \uplambda\bigg\}.
\]
In the definition of $E$ above we have that $\M_v=\M_{(1,s)}=\M_{1}$ since we have reduced to the case $v=(1,s)=(1,0)$. Using the subordination property of $\mathcal L' _1$ and the Fefferman-Stein inequality once in the direction $e_2$ for $\M_2$ and once in the direction $v=(1,s)=(1,0)$ for $\M_v$ we estimate
\[
\mass_{a,1}(\mathcal L' _1)\leq w\Bigg(\bigcup_{R\in\mathcal R_0 ^\star} \uppi_1(R)\times K_R\Bigg)\lesssim \frac{1}{\uplambda} \sum_{Q\in\mathcal L} a_Q\frac{\M_V\M_2w(Q)}{{u}(Q)}\leq \frac{[w,u]_S}{\uplambda}\mass_{a,1}(\mathcal L).
\]
We have thus proved the lemma upon setting $\mathcal L_1 \coloneqq \mathcal L_1 ^{''}\cup \mathcal L_1 '$ and choosing $\uplambda\geq C [w,u]_S$ for a sufficiently large numerical constant $C>1$.
\end{proof}

Repeating the steps in the proof of Lemma~\ref{lem:shadow} for $\uplambda$ as in the statement of Lemma~\ref{lem:witer} we get for the sets $\mathcal R_{s,k}$ defined with respect to this $\uplambda$ that
\[
w(\mathsf{sh}(\mathcal R_{s,k}))\lesssim 2^{-k} \mass_{a,1}(\mathcal R),
\]
and this completes the proof of Theorem \ref{thm:wcarleson}.

\subsection{Applications of Theorem \ref{thm:wcarleson}} The first corollary of Theorem \ref{thm:wcarleson} is a two-weighted estimate for the directional maximal operator $\M_V$ from \eqref{e:Mv}. 

\begin{theorem}\label{thm:maxweight}Let $V\subset \mathbb S^1$ be a finite set of $N$ slopes and $w$ be a weight on $\R^2$. Then 
\[
\left\|\M_V : {L^{2}(\widetilde{\M_{V}}w)}\to {L^{2,\infty}(w)}\right\|\lesssim \sqrt{\log N} 
 \qquad \widetilde{\M_{V}}  \coloneqq \M_{V} \circ  \M_{V} \circ \max\{\M_{(1,0)},\M_{(0,1)}\}.
\]
\end{theorem}

\begin{remark} In the proof below, we argue for almost horizontal $V$, and use $\M_{(0,1)}$ in place of   $\max\{\M_{(1,0)},\M_{(0,1)}\}$. The usage of  $\max\{\M_{(1,0)},\M_{(0,1)}\}$ is made necessary to make the statement of the theorem invariant under rotation of $V$.  
\end{remark}

\begin{proof}  By standard limiting arguments, it suffices to prove that for each $k\in \mathbb Z$ the estimate
\begin{equation}
\label{e:weightpf1}
\left\|\M_{\mathcal R} : {L^{2}(z)}\to {L^{2,\infty}(w)}\right\|\lesssim \sqrt{\log N}, \qquad z\coloneqq \M_{\mathcal R} \circ  \M_{V} \circ \M_{(0,1)}w,
\end{equation}
when $\mathcal R$ is a one-parameter collection as in \eqref{e:onepar2020}, holds uniformly in $k$.

For a nonnegative function $f\in\mathcal S(\R^2)$ let $Uf$ be a linearization of $\M_{\mathcal R}f$, namely
\[
\M_{\mathcal R}f(x)=U f(x)= \frac{1}{|R(x)|}\int_{R(x)}f(y)\,\d y = \sum_{R\in \mathcal R} \l f \r_R \cic{1}_{F_R}(x), \qquad F_R\coloneqq\{x\in R:\, R(x)=R \}.
\]   By duality, \eqref{e:weightpf1} turns into \begin{equation}
\label{e:weightpf2}
\|U^*(w\cic{1}_E) \|_{L^2(z^{-1}) }\lesssim \sqrt{\log N}  \sqrt{w(E)},\qquad \forall E\subset \R^2.
\end{equation}
We can easily calculate
\[
U^*(w\cic{1}_E)= \sum_{R\in\mathcal R} w(E\cap F_R)\frac{\ind_R}{|R|}
\]
and it is routine to check that $\{w(E\cap F_ R)\}_{R\in\mathcal R}$ is a $w$-Carleson sequence according to Definition~\ref{def:wcarleson}. The main point here is that the sets $\{E\cap F_R\}_{R\in\mathcal R}$ are by definition pairwise disjoint and $F_R\subseteq R$ for each $R\in\mathcal R$.

Setting $u\coloneqq   \M_{V} \circ \M_{{(0,1)}} w, $ if $S$ are the slopes of $V$,  it is clear that $[w,u]_S\lesssim 1$ and that $z^{-1}= (\M_{\mathcal R} u)^{-1}$. Therefore \eqref{e:weightpf2} follows from an application of Theorem \ref{thm:wcarleson}.
\end{proof}

We may in turn use Theorem \ref{thm:maxweight} to establish a  weighted norm inequality for maximal directional singular integrals with controlled dependence on the cardinality $\#V=N$. Similar considerations may be used to yield  weighted bounds for directional singular integrals in $L^p(\R^2)$ for $p>2$; we do not pursue this issue.

\begin{theorem}\label{thm:singweight} Let $K$ be a standard Calder\'on-Zygmund convolution kernel on $\,\R$ and $V\subset \mathbb S^1$ be a finite set of $N$ slopes. For $\,v\in V$ we define
\[
T_vf (x) = \sup_{\upepsilon>0}\left| \int_{\upepsilon<t<\frac{1}{\upepsilon}} f(x+tv) K(t) \,\d t\right|, \qquad T_V f(x)= \sup_{v\in V} |T_v f(x)|.
\] 
Let $w$ be a weight on $\R^2$ with $[w]_{A_1^V}\coloneqq \|{\M_{V}}w/w\|_\infty<\infty$. Then
\[
\left\|T_{V} :\,{L^{2}(w)}\to {L^{2,\infty}(w)}\right\|\lesssim (\log N)^{\frac32}  [w]_{A_1^V}^{\frac52}.
\]
\end{theorem}

We sketch the proof, which is a weighted modification of the arguments for  \cite[Theorem 1]{DDP}. Hunt's classical exponential good-$\uplambda$ inequality, see  \cite[Proposition 2.2]{DDP} for a proof, may be upgraded to 
\begin{equation}
\label{e:weightpf3}
w\left(\left\{ x\in \R^2 :T_vf(x) > 2\uplambda, \M_{v} f(x) \leq  \upgamma \uplambda \right\}\right) \lesssim \exp\left(-{\textstyle\frac{c}{\upgamma [w]_{A_1^V}}}\right) w\left(\left\{ x\in \R^2 :T_vf(x) > \uplambda\right\}\right) 
\end{equation}
by using that $[w]_{A_1^V}$ dominates the $A_\infty$ constant of the one-dimensional weight $t\mapsto w(x+tv)$ for all $x\in \R^2,v\in V$, together with Fubini's theorem. With \eqref{e:weightpf3} in hand, Theorem \ref{thm:singweight} follows from Theorem~\ref{thm:maxweight} via standard good-$\uplambda$ inequalities, selecting $(\upgamma)^{-1}\sim[w]_{A_1^V} \log N$. Note that the right hand-side of the estimate in the conclusion of Theorem~\ref{thm:maxweight} becomes $[w]_{A_1} ^\frac{3}{2} \sqrt{\log N}$ when the estimate is specified to $A_1 ^V$ weights as the ones we consider here.

\section{Tiles, adapted families, and intrinsic square functions} \label{s:isf}
 We define here some general notions of tiles and adapted families of wave-packets: definitions in this spirit have appeared in, among others \cite{BarrLac,DDP,LL2,LR,LL1}. These will be essential for the time-frequency analysis square functions we use in this paper in order to model the main operators of interest. After presenting these abstract definitions we show some general orthogonality estimates for wave packet coefficients.  We then detail how these notions are specialized in three particular cases of interest.

\subsection{Tiles and wavelet coefficients}Throughout this section we fix a finite set of slopes $S\subset [-1,1]$. Remember that alternatively we will refer to the set of vectors $V\coloneqq\{(1,s):\, s\in S\}$. A \emph{tile} is a set $t\coloneqq R_t \times \Omega_t \subset \R^2\times \R^2$ where $R_t \in\mathcal D_S ^2$ and $\Omega_t\subset \R^2$ is a measurable set, and $|R_t||\Omega_t|\gtrsim 1$.  We denote by $s(t)\in S$ the slope such that $R_t\in \mathcal D_{s(t)} ^2$, and then 
\[
R_t = A_{s(t)}(I_t \times J_t)\quad\text{with}\quad I_t\times J_t \in \mathcal D _0 ^2.
\]
We also use the notation $v_t\coloneqq (1,s(t))$. There are several different collections of tiles used in this paper, they will generically be denoted by $\tiles,\tiles_1,\tiles '$ or similar. Given any collection of tiles $\tiles$ we will use often use the notation $\mathcal R_{\tiles }\coloneq\{R_t:\, t\in \tiles\}$ to denote the collection of spatial components of the tiles in $\tiles$. The exact geometry of these tiles will be clear from context, however several estimates hold for generic collections of tiles as we will see in \S\ref{sec:orthocone}.

Let $t=R_t\times \Omega_t$ be a   tile  and $M\geq 2$. We denote by $\mathcal A_t^{M}$ the collection  of  Schwartz functions $\upphi$ on $\R^2$ such that:
\begin{itemize}
	\item [(i)]  $\mathrm{supp}(\widehat{\upphi})\subset \Omega_t$,
	\item[(ii)]  There holds
	\[
	\sup_{0\leq \upalpha,\upbeta\leq M} \sup_{x\in \R^2} 
	|R_t|^{\frac12} |I|^\upalpha|J|^\upbeta   \left( 1 + \frac{|x\cdot v_t|}{|I||v_t|}\right)^{M} \left( 1 +    \frac{|x\cdot e_2|}{|J|} \right)^{M}
	\left|\partial_{v_t}^\upalpha \partial_{e_2}^\upbeta \upphi (x+c_{R_t}) \right| \leq 1.
\]
\end{itemize}
In the above display  $c_{R_t}$ refers to  the center of  $R_t$ and $\partial_{v_t}(\cdot)\coloneqq \frac{v_t}{|v_t|}\cdot \nabla (\cdot)$. An immediate consequence of property (ii) is the normalization
\[
\sup_{\upphi \in \mathcal A_t^{M} }\|\upphi\|_{2} \lesssim 1.
\] 
We thus refer to $\mathcal A_t^{M}$ as the collection of \emph{$L^2$-normalized  wave packets adapted to $t$ of order $M$}. { For our purposes, it will suffice to work with moderate values of $M$, say $2^3\leq M \leq 2^{50}$. In fact, we use $M=M_0=2^{50}$ in the definition of the  \emph{intrinsic wavelet coefficient}   associated with the tile $t$ and the Schwartz  function $f$:\begin{equation}
\label{e:intwcdef}
a_{t}(f) \coloneqq \sup_{\upphi\in \mathcal A_t^{M_0} }| \l f,\upphi \r|^2,\qquad M_0=2^{50}.
\end{equation}}

This section is dedicated to square functions involving wavelet coefficients associated with particular collections of tiles which formally look like
\[
\Delta_\tiles(f)^2 \coloneqq \sum_{t\in\tiles} a_t(f)\frac{\ind_{R_t}}{|R_t|},\qquad\tiles\,\text{is a collection of tiles}.
\]
 We begin by proving some general global and local orthogonality estimates for collections of tiles with finitely overlapping frequency components. These estimates will be essential in showing that the sequence $\{a_t(f)\}_{t\in\tiles}$ is Carleson in the sense of Section~\ref{sec:carleson}, when $|f|\leq \ind_E$ for some measurable set $E\subset \R^2$ with $0<|E|<\infty$. This in turn will allow us to use the directional Carleson embedding of Theorem~\ref{thm:carleson} in order to conclude corresponding estimates for intrinsic square functions defined on collections of tiles.

\subsection{Orthogonality estimates for collections of tiles}\label{sec:orthocone}  We begin with an easy orthogonality estimate for wave packet coefficients. For completeness we present a sketch of proof which has a $TT^*$  flavor. The argument follows the lines of proof of \cite{LR}*{Proposition 3.3}.

\begin{lemma}\label{l:globalortho}Let $\mathbf T$ be a set of tiles such that $\sum_{t\in\tiles}\ind_{\Omega_t}\lesssim 1$, let $M\geq 2^3$ and $\{\upphi_t:t\in\tiles\}$ be such that $\upphi_t \in \mathcal A_{t}^{M} $ for all $t\in \tiles$. We have the estimate
\begin{equation}\label{e:ttstarlemma}
 \sum_{t\in \mathbf T} |\l f,\upphi_t \r|^2 \lesssim\|f\|_2^2, 
\end{equation}
and as a consequence
\[
 \sum_{t\in \mathbf T} a_{t}(f) \lesssim \|f\|_2^2.
\]
\end{lemma}

\begin{proof} Fix $M\geq 2^3$.
It suffices to prove that for $\|f\|_2=1$ and an arbitrary adapted family of wave packets  $\{\upphi_t:\, \upphi_t\in \mathcal A_t ^M,\, t\in \tiles\}$ there holds
\begin{equation}
\label{e:ttstar}
B\coloneqq \sum_{t\in \mathbf T} |\l f,\upphi_t \r|^2 \lesssim 1,
\end{equation}
Let us first fix some $\Omega\in \Omega(\tiles)\coloneqq \{\Omega_t:\, t\in\tiles\}$ and consider the family 
\[
\tiles(\{\Omega\})\coloneqq \{t\in\tiles:\,\Omega_t = \Omega\}.
\]
To prove \eqref{e:ttstar}, we introduce
\[
B_\Omega(g)\coloneqq \sum_{t\in \mathbf T(\{\Omega\})} |\l g,\upphi_t \r|^2,\qquad  S_\Omega(g)\coloneqq (\hat g\ind_\Omega)^\vee.
\]
We claim that $B_\Omega(g)\lesssim \|g\|_2 ^2$ for all $g$, uniformly in $\Omega\in\Omega(\tiles)$. Assuming the claim for a moment and remembering the finite overlap assumption on the frequency components of the tiles we have
\[
B=\sum_{\Omega\in\Omega(\tiles)}B_\Omega(S_\Omega f)\lesssim \sum_{\Omega\in\Omega(\tiles)}\|S_\Omega(f)\|_2 ^2\leq \bigg\|\sum_{\Omega\in\Omega(\tiles)}\ind_\Omega \bigg\|_\infty ^2 \| f\|_2 ^2\lesssim 1
\]
as desired. It thus suffices to prove the claim. To this end let
\[
P_\Omega(g)\coloneqq \sum_{t\in \mathbf T(\{\Omega\})} \l g,\upphi_t \r \upphi_t. 
\]
Then for any $g$ with $\|g\|_2=1$ we have that $B_\Omega(g)=\l P_\Omega(g),g \r \leq \|P_\Omega(g)\|_2 $ and it suffices to prove that $\|P_\Omega(g)\|_2^2\lesssim B_\Omega(g)$. A direct computation reveals that
\[
\|P_\Omega(g)\|_2^2 \leq B_\Omega(g)\sup_{t' \in \tiles(\{\Omega\})} \sum_{ t \in \tiles(\{\Omega\}) } |\l \upphi_t,\upphi_{t'} \r| \lesssim B
\]
where the second inequality in the last display above follows by  the polynomial  decay of the wave packets $\{\upphi_t:\, \, \Omega_t=\Omega\}$. This completes the proof of the lemma.
\end{proof}

We present below a localized orthogonality statement which is needed in order to verify that the coefficients $a_{t}(f)$ form a Carleson sequence in the sense of \S\ref{sec:carleson}. Verifying this Carleson condition relies on a variation of Journ\'e's lemma that can be found in \cite[Lemma 3.23]{CLMP}; we rephrase it here adjusted to our notation. In the statement of the lemma below we denote by $\mathrm{M}_{\mathcal P_s ^2}$ the \emph{maximal function} corresponding to the collection $\mathcal P_s ^2$ where $s\in S$ is a fixed slope. Note  that the proof in \cite{CLMP} corresponds to the case of slope $s=0$ but the general case $s\in S$ follows easily by a change of variables.  Remember here that we have $S\subset [-1,1]$.

In the statement of the lemma below two parallelograms are called \emph{incomparable} if none of them is contained in the other.

\begin{lemma} \label{l:journe} Let $s\in S$ be a slope and $\mathcal T\subset \mathcal  D_s ^2$ be a collection of pairwise incomparable parallelograms. Define
\[
\mathsf{sh}^\star(\mathcal T)\coloneqq \big\{\mathrm{M}_{\mathcal P_s ^2} \cic{1}_{\mathsf{sh}(\mathcal T)} > 2^{-6} \big\} 
\]
and for each $R \in \mathcal T$ let $u_R$ be the least integer $u$ such that $2^{u}R\not\subset \mathsf{sh}^\star(\mathcal T) .$ Then 
\[
\sum_{\substack{R\in \mathcal T\\ u_R=u}} |R| \lesssim 2^u| \mathsf{sh}(\mathcal T)|.
\]
\end{lemma}
With the suitable analogue of Journ\'e's lemma in hand  we are ready to state and prove the localized orthogonality condition for the coefficients $a_{t}(f)$.

\begin{lemma}\label{l:localortho} Let $s\in S$ be a slope, $\mathcal T\subset \mathcal P_s ^2$ be a given collection of parallelograms and $\tiles$ be a collection of tiles such that
	\[
	\mathcal R_{\tiles} \coloneqq \{R_t:\, t\in \tiles\}
	\]
is subordinate to   $\mathcal T$.  Then  we have
\[
\sum_{t\in \tiles} a_{t}(f)\lesssim \left| \mathsf{sh}(\mathcal T) \right| \|f\|_\infty^2.
\]
\end{lemma}

\begin{proof}  We first make a standard reduction that allows us to pass to a collection of dyadic rectangles. To do this we use that there exist at most $9^2$ shifted dyadic grids $\mathcal D_{s,j} ^2$ such that for each parallelogram $T\in \mathcal T$  there exists $\widetilde{T}\in\cup_j \mathcal D^2 _{s,j}$ with $T\subset  \widetilde{T}$ and $|T|\leq |\widetilde{T}|\lesssim |T|$; see for example \cite{HLP}. Now note that for each $\widetilde T\in \widetilde {\mathcal T}$ we have
	\[
	\frac{|T\cap \widetilde{T}|}{|\widetilde{T}|}\gtrsim 1,\qquad \mathsf{sh}(\widetilde {\mathcal T})\subset  \big\{\M_{\mathcal P_S ^2}(\ind_{\mathsf{sh}(\mathcal T)})\gtrsim 1\big\}
	\]
and so $|\mathsf{sh}(\widetilde{\mathcal T})|\lesssim |\mathsf{sh}(\mathcal T)|$. Now it is clear that we can replace $\mathcal T$ with the dyadic collection $\widetilde{\mathcal T}$ in the assumption.  Furthermore there is no loss in generality with assuming that $\mathcal T$ is a pairwise incomparable collection.  We do so in the rest of the proof and continue using the notation $\mathcal T$ assuming it is a dyadic collection.
 
Since $\mathcal R_{\tiles}$ is subordinate to $\mathcal T$ we have the decomposition
\[
\tiles=\bigcup_{T\in\mathcal T}\tiles(T), \qquad \tiles(T)\coloneqq \{t\in \tiles:\, R_t\subset  T\}.
\]

Now if $f$ is supported on $\mathsf{sh}^\star (\mathcal T)$ and $\upphi_t \in \mathcal A_t ^{M_0}$ for each $t\in\tiles$ then
\[
\sum_{t \in \tiles} |\l f,\upphi_t\r|^2 \lesssim \|f\|_2^2\leq |\mathsf{sh}^\star(\mathcal T)|\|f\|_\infty^2\lesssim  |\mathsf{sh}(\mathcal T)|\|f\|_\infty^2
\]
by Lemma \ref{l:globalortho}. We may thus assume that $f$ is supported outside $\mathsf{sh}^\star(\mathcal T) $. 
By Lemma \ref{l:journe} it then suffices to prove that
\[
\sum_{t\in \tiles(T)} |\l f, \upphi_t\r|^2 \lesssim 2^{-10u} |T| 
\]
whenever $u$ is the least integer such that $2^{u}T\not\subset \mathsf{sh}^\star(\mathcal T)$ and $\|f\|_\infty=1$. As $f$ is supported off $\mathsf{sh}^\star(\mathcal T)$ we have have for this choice of $u$ that
\[
f= \sum_{n \geq 0 } f_n, \qquad f_n\coloneqq f\cic{1}_{2^{u+n}T\setminus 2^{u+n-1} T}.
\] 
Let $z_T$ be the center of $T$ and suppose that $T=A_s(I_T\times J_T)$ with $I_T\times J_T \in \mathcal D_0 ^2$; remember that we write $v_s\coloneqq (1,s)$. Let 
\begin{equation}\label{eq:chiT}
\chi_T(x)\coloneqq\Big(1+ \frac{(x-z_T)\cdot v_s}{|I_T||v_s|}\Big)^{-20}(1+|J_T|^{-1} (x-z_T)\cdot e_2)^{-20}.
\end{equation}
Observe preliminarily that
\[
\|f_n\chi_T\|_\infty\lesssim 2^{-20 (u+n) }
\]
so that for any constant $c>0$ we have
\[
\begin{split}
\bigg(\sum_{t\in \tiles(T)} |\l f, \upphi_t\r|^2\bigg)^{\frac12} &\leq \sum_{n \geq 0} \bigg(\sum_{t\in \tiles(T)} |\l f_n, \upphi_t\r|^2\bigg)^{\frac12}= \sum_{n\geq 0}\bigg(\sum_{t\in\tiles(T)} |\l  f_nc^{-1}\chi_T, c\chi_T^{-1}\upphi_t\r|^2\bigg)^{\frac12}
\\ 
& \lesssim  \sum_{n\geq 0} \|f_n\chi_T\|_2 \lesssim \sum_{n\geq 0} \|f_n\chi_T\|_\infty |2^{u+n}T|^{\frac12} \lesssim 2^{-5u}|T|^{\frac12}
\end{split}
\]
as claimed. To pass to the second line we have used estimate \eqref{e:ttstarlemma} of Lemma \ref{l:globalortho} together with the easily verifiable fact that for each $ t\in \mathcal \tiles(T) $ the wave-packet $c\chi_T^{-1}\upphi_t$ is adapted to $t$ with order $M_0-20\geq 2^3$ provided the absolute constant $c$ is chosen small enough.
\end{proof}

\subsection{The intrinsic square function associated with rough frequency cones}\label{sec:conetiles} Let $s\in  S$ be our finite set of slopes. As usual we write $v_s\coloneqq(1,s)$ for $s\in S$ and $V\coloneqq\{v_s:\, s\in S\}$ and switch between the description of directions as slopes or vectors as desired with no particular mention. Now assume we are given a finitely overlapping collection of arcs $\{\upomega_s\}_{s\in S}$ with each $\upomega_s \subset\mathbb S^1$ centered at $(v_s/|v_s|)^\perp$. We will adopt the notation $\upomega_s \eqqcolon ((v_{s ^-}/|v_{s ^-}|)^\perp,(v_{s ^+}/|v_{s ^+}|)^\perp)$ assuming that the positive direction on the circle is counterclockwise and $s^-<s<s^+$.

For $s\in S$ we define the conical sectors
\begin{equation}
\label{e:freqsconical}
\Omega_{s,k}\coloneqq \left\{\xi \in \R^2:\,  2^{k-1} < |\xi| < 2^{k+1}, \, \frac{\xi}{|\xi|} \in \upomega_s \right\}, \qquad k \in \mathbb Z; 
\end{equation}
these are an overlapping cover of the cone
\[
C_s\coloneqq \left\{\xi\in \R^2\setminus\{0\}:\, \frac{\xi}{|\xi|} \in \upomega_s\right\}
\]
with   $k\in \mathbb Z$ playing the role of the annular parameter. Each sector $\Omega_{s,k}$ is strictly contained in the cone $C_s$.

For each $s\in S$ let $\ell_s\in \mathbb Z$ be chosen such that $2^{-\ell_s}<|\upomega_s|\leq 2^{-\ell_s+1}$. We perform a further discretization of each conical sector $\Omega_{s,k}$ by considering Whitney-type decompositions with respect to the distance to the lines determined by the boundary rays $r_{s^-}$ and $r_{s^+}$; here $r_{s^+}$ denotes the ray emanating from the origin in the direction of $v_{s^+} ^\perp$ and similarly for $r_{s^-}$. For each sector $\Omega_{s,k}$ a central piece which we call $\Omega_{s,k,0}$ is left uncovered by these Whitney decompositions. This is merely a technical issue and we will treat these central pieces separately in what follows.

To make this precise let $s,k$ be fixed and define the regions
\begin{equation}
\begin{split}\label{eq:whitney}
\Omega_{s,k,m}\coloneqq\Big\{\xi\in\Omega_{s,k}:\, \frac13 2^{-|m|-1} \leq \frac{\mathrm{dist}(\xi,r_{s^+})}{|\upomega_s|} \leq \frac13 2^{-|m|+1} \Big\},\qquad m>0,
\\
\Omega_{s,k,m}\coloneqq\Big\{\xi\in\Omega_{s,k}:\,  \frac13 2^{-|m|-1} \leq \frac{\mathrm{dist}(\xi,r_{s^-}) }{|\upomega_s|} \leq \frac13 2^{-|m|+1} \Big\},\qquad m<0.
\end{split}
\end{equation}
The central part that was left uncovered corresponds to $m=0$ and is described as
\begin{equation}\label{eq:whitney0}
\Omega_{s,k,0}\coloneqq \Big\{\xi\in\Omega_{s,k}:\,    \min(\mathrm{dist}(\xi,r_{s^-}),\mathrm{dist}(\xi,r_{s^+}))  \geq \frac12 \frac13	 |\upomega_s|\Big\}.
\end{equation}
Notice that the collection $\{\Omega_{s,k,m}\}_{m\in\N}$ is a finitely overlapping cover of $\Omega_{s,k}$. Furthermore the family $\{\Omega_{s,k,m}\}_{s,k,m}$ has finite overlap as the cones $\{C_s\}_{s\in S}$ have finite overlap and for fixed $s$ the family $\{\Omega_{s,k,m}\}_{k,m}$ is Whitney both in $k$ and $m$.

These geometric considerations are depicted in Figure~\ref{fig:conetiles} below.
 
\begin{figure}[ht]
\centering
\def\svgwidth{330pt}
\begingroup%
  \makeatletter%
  \providecommand\color[2][]{%
    \errmessage{(Inkscape) Color is used for the text in Inkscape, but the package 'color.sty' is not loaded}%
    \renewcommand\color[2][]{}%
  }%
  \providecommand\transparent[1]{%
    \errmessage{(Inkscape) Transparency is used (non-zero) for the text in Inkscape, but the package 'transparent.sty' is not loaded}%
    \renewcommand\transparent[1]{}%
  }%
  \providecommand\rotatebox[2]{#2}%
  \newcommand*\fsize{\dimexpr\f@size pt\relax}%
  \newcommand*\lineheight[1]{\fontsize{\fsize}{#1\fsize}\selectfont}%
  \ifx\svgwidth\undefined%
    \setlength{\unitlength}{955.56878988bp}%
    \ifx\svgscale\undefined%
      \relax%
    \else%
      \setlength{\unitlength}{\unitlength * \real{\svgscale}}%
    \fi%
  \else%
    \setlength{\unitlength}{\svgwidth}%
  \fi%
  \global\let\svgwidth\undefined%
  \global\let\svgscale\undefined%
  \makeatother%
  \begin{picture}(1,0.52752546)%
    \lineheight{1}%
    \setlength\tabcolsep{0pt}%
    \put(0,0){\includegraphics[width=\unitlength,page=1]{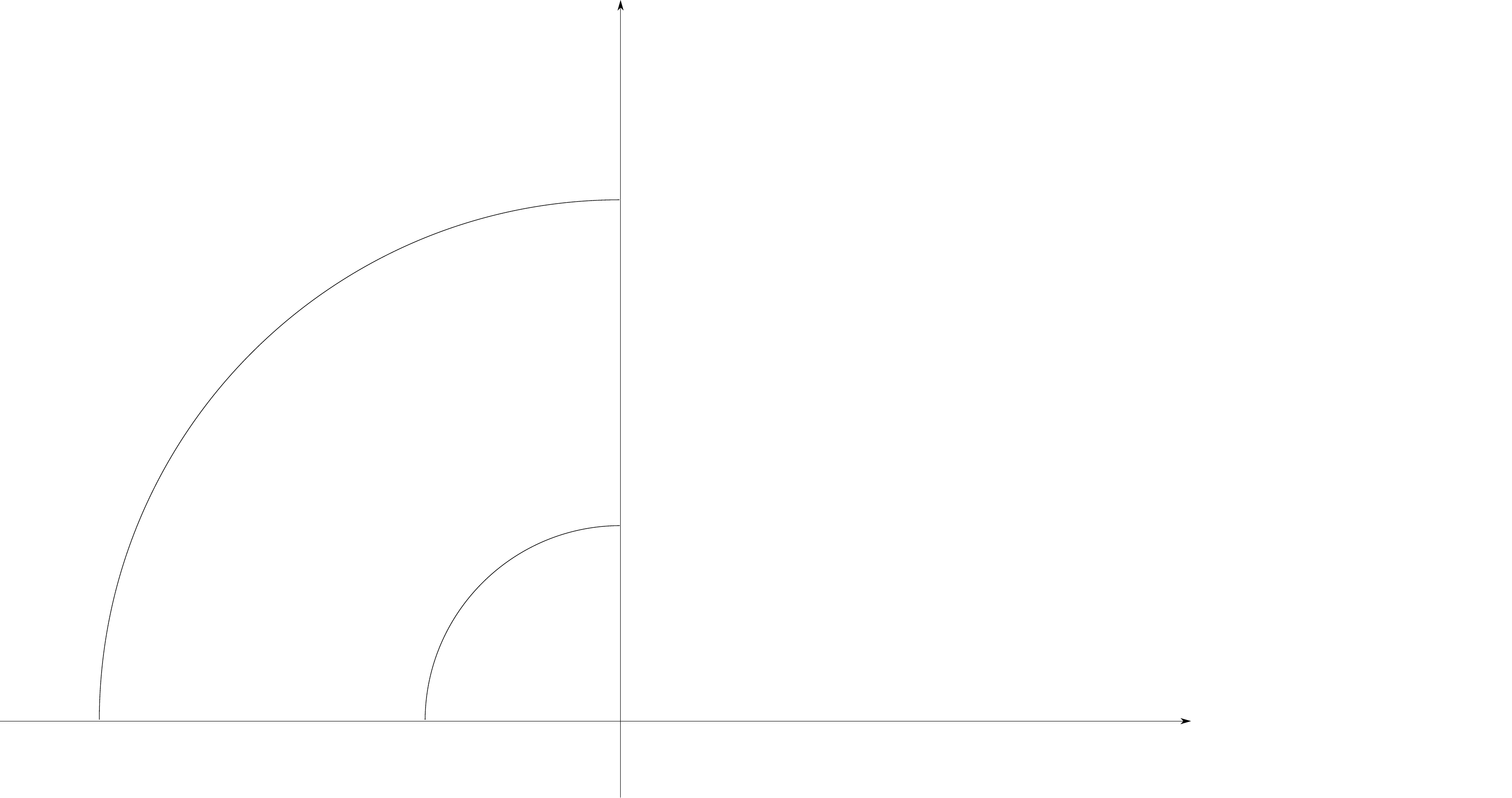}}%
    \put(0.73352601,0.39585592){\color[rgb]{0,0,0}\makebox(0,0)[lt]{\lineheight{0}\smash{\begin{tabular}[t]{l}$\scriptstyle{v_{s} =(1,s)}$\end{tabular}}}}%
    \put(0,0){\includegraphics[width=\unitlength,page=2]{conetiles.pdf}}%
    \put(0.65892369,0.09870227){\color[rgb]{0,0,0}\makebox(0,0)[lt]{\lineheight{0}\smash{\begin{tabular}[t]{l}$\scriptstyle{R_t \text{ dual to }  \Omega_{s,k,0}}$\end{tabular}}}}%
    \put(0.009892,0.02296516){\color[rgb]{0,0,0}\makebox(0,0)[lt]{\lineheight{0}\smash{\begin{tabular}[t]{l}$\scriptstyle{|\xi|=2^{k+1}}$\end{tabular}}}}%
    \put(0.24263737,0.02400128){\color[rgb]{0,0,0}\makebox(0,0)[lt]{\lineheight{0}\smash{\begin{tabular}[t]{l}$\scriptstyle{|\xi|=2^{k-1}}$\end{tabular}}}}%
    \put(0,0){\includegraphics[width=\unitlength,page=3]{conetiles.pdf}}%
    \put(0.20754137,0.39357333){\color[rgb]{0,0,0}\makebox(0,0)[lt]{\lineheight{0}\smash{\begin{tabular}[t]{l}$\scriptstyle{2^{k+1}|\omega_s|\eqsim 2^{k-\ell_s}}$\end{tabular}}}}%
    \put(0,0){\includegraphics[width=\unitlength,page=4]{conetiles.pdf}}%
    \put(0.32090948,0.33414535){\color[rgb]{0,0,0}\makebox(0,0)[lt]{\lineheight{0}\smash{\begin{tabular}[t]{l}$\scriptstyle{\Omega_{s,k}}$\end{tabular}}}}%
    \put(0.0093418,0.22583676){\color[rgb]{0,0,0}\makebox(0,0)[lt]{\lineheight{0}\smash{\begin{tabular}[t]{l}$\scriptstyle{r_{s^+}}$\end{tabular}}}}%
    \put(0.13727607,0.38045671){\color[rgb]{0,0,0}\makebox(0,0)[lt]{\lineheight{0}\smash{\begin{tabular}[t]{l}$\scriptstyle{r_{s^-}}$\end{tabular}}}}%
    \put(0,0){\includegraphics[width=\unitlength,page=5]{conetiles.pdf}}%
    \put(0.32247923,0.27371016){\color[rgb]{0,0,0}\makebox(0,0)[lt]{\lineheight{0}\smash{\begin{tabular}[t]{l}$\scriptstyle{\Omega_{s,k,0}}$\end{tabular}}}}%
    \put(0.02795961,0.3315946){\color[rgb]{0,0,0}\makebox(0,0)[lt]{\lineheight{0}\smash{\begin{tabular}[t]{l}$\scriptstyle{ v_{s}  ^\perp=(-s,1)}$\end{tabular}}}}%
    \put(0,0){\includegraphics[width=\unitlength,page=6]{conetiles.pdf}}%
    \put(0.07949608,0.10867378){\color[rgb]{0,0,0}\makebox(0,0)[lt]{\lineheight{0}\smash{\begin{tabular}[t]{l}$\scriptstyle{\Omega_{s,k,m},\,} \scriptscriptstyle{m}>0$\end{tabular}}}}%
    \put(0,0){\includegraphics[width=\unitlength,page=7]{conetiles.pdf}}%
  \end{picture}%
\endgroup%
\caption{The decomposition of the sector $\Omega_{s,k}$ into Whitney regions, and the spatial grid corresponding to the middle region $\Omega_{s,k,0}$.}
\label{fig:conetiles}
\end{figure}

The collection of tiles $\tiles$ corresponding to this decomposition is obtained as
\[
\tiles\coloneqq  \bigcup_{s\in S} \tiles_{s} ^-  \cup \tiles^0 _s \cup \tiles_{s} ^+
	\]
where 
\begin{alignat}{2}\label{eq:tilesskm}
\notag  \qquad & \tiles_{s} ^- \coloneqq\bigcup_{k\in\Z,\,m<0}\tiles_{s^-,k,m},\qquad  &&\tiles _{s^-,k,m}  \coloneqq \big\{t=R_t  \times \Omega_{s,k,m}: \,  R_t  \in \mathcal D_{s^-,k,k-\ell_{s}+|m|}\big\}, \quad m<0,
 \\
 \qquad & \tiles_{s} ^0 \coloneqq\bigcup_{k\in\Z}\tiles_{s,k,0}, \qquad  &&\tiles _{s,k,0}   \coloneqq \big\{t=R_t  \times \Omega_{s,k,0}: \,  R_t  \in \mathcal D_{s,k,k-\ell_{s} }\big\}, 
 \\
 \qquad & \tiles_{s} ^+\coloneqq\bigcup_{k\in\Z,\,m>0}\tiles_{s^+,k,m}, \qquad &&\tiles _{s^+,k,m} \coloneqq \big\{t=R_t  \times \Omega_{s,k,m}: \,  R_t  \in \mathcal D_{s^+,k,k-\ell_{s}+|m|}\big\}, \quad m>0.
\end{alignat}
We stress here that for each cone $C_s$ we introduce tiles in three possible directions $v_{{s^-}},v_s,v_{s^+}$. This turns out to be technical nuisance more than anything else as the total number of directions is still comparable to $\#S$, and our estimates will be uniform over all $S$ with the same cardinality. However in order  to avoid confusion we set
\begin{equation}\label{eq:s*}
S^*\coloneqq S\cup\{s^-:\, s\in S\}\cup\{s^+:\, s\in S\}\eqqcolon S^-\cup S\cup S^+.
\end{equation}
Note also that for fixed $s,k,m$ the choice of scales for $R_t$ yields that the tile $t=R_t \times  \Omega_{s,k,m}$  obeys the uncertainty principle in both radial and tangential directions. 

We then define the associated  intrinsic square function by
\begin{equation}
\label{e:intrinsicsf}
\begin{split}
&\Delta_{\tiles} (f) \coloneqq \Bigg( \sum_{t \in \mathbf T}a_{t}(f) \frac{\ind_{R_t}}{|R_t|} \Bigg)^\frac12,
\end{split}
\end{equation} 
where   the set of slopes $S$ are kept implicit in the notation. Here we remember that the notation $a_t(f)$ was introduced in \eqref{e:intwcdef}. Using the orthogonality estimates of \S\ref{sec:orthocone} as input for Theorem \ref{thm:carleson} we readily obtain the estimates of the following theorem. 

\begin{theorem} \label{t:isf} We have the estimates
\begin{align}\label{e:isfp}
& \big\|\Delta_{\tiles}: L^{p}(\R^2)   \big\|  \lesssim_p (\log \#S)^{\frac12-\frac1p} (\log\log\#S)^{\frac12-\frac1p}, \qquad 2\leq p<4, \\
& \sup_{E,f}
\label{e:isfrest} \frac{\left\|\Delta_{\tiles} (f\cic{1}_E)\right\|_{4}}{|E|^{\frac14}} \lesssim (\log \#S)^{\frac14} (\log\log\#S)^{\frac14},
\end{align}
where the supremum in the last display is taken over all measurable sets $E\subset \R^2$ of finite positive measure and all Schwartz functions $f$ on $\R^2$ with $\|f\|_\infty\leq 1$.
\end{theorem}

\begin{proof}[Proof of Theorem~\ref{t:isf}] First of all, observe that the case $p=2$ of \eqref{e:isfp} is exactly the conclusion of Lemma \ref{l:globalortho}. By restricted weak type interpolation it thus suffices to prove \eqref{e:isfrest} to obtain the remaining cases of \eqref{e:isfp}: we turn to the former task.

For convenience define $S^*\coloneqq S\cup\{s^-:\, s\in S\}\cup\{s^+:\, s\in S\}\eqqcolon S^-\cup S\cup S^+$; note that this is the actual set of slopes of tiles in $\tiles$. Let
\[
 \mathcal R_{\tiles}\coloneqq\{R_t:\,t\in\tiles\} \subset  \mathcal D_{S^*} ^2.
\]
Observe that we can write
\[
\Delta_{\tiles}(f\ind_E)^2 =\sum_{R\in \mathcal R_{\tiles}}\bigg(\sum_{t\in\tiles:\, R_t=R}a_t(f\ind_E)\bigg)\frac{\ind_{R }}{|R |}\eqqcolon \sum_{R\in \mathcal R_{\tiles}}a_R\frac{\ind_R}{|R|}
\]
where
\[
a\coloneqq\Big\{a_R= \sum_{t\in\tiles:\, R_t=R}a_t(f\ind_E): \, R \in \mathcal R_{\tiles}\Big\}.
\] 
We fix $E$ and $f$ as in the statement and we will obtain \eqref{e:isfrest} from an application of Theorem \ref{thm:carleson} to the Carleson sequence $a=\{a_R\}_{R \in \mathcal R_{\tiles}}$.

First, $\mathsf{mass}_a\lesssim |E|$ as a consequence of Lemma \ref{l:globalortho} since
\[
\sum_{R \in \mathcal R_{\tiles}}a_R 
 =\sum_{R \in \mathcal R_{\tiles}}\sum_{t\in\tiles:\, R_t=R}a_t(f\ind_E)=\sum_{t\in\tiles} a_t(f\ind_E)\lesssim\|f\ind_E\|_2 ^2\lesssim |E|.
\]

Further, the fact that $a$ is (a constant multiple of) an $L^\infty$-normalized Carleson sequence is a consequence of the  localized estimate of Lemma \ref{l:localortho}. To verify this we need to check the validity of Definition~\ref{def:carleson} for the sequence $a$ above. To that end let $\mathcal L\subset  \mathcal D_{S^*} ^2$ be a collection of parallelograms which is subordinate to $\mathcal T\subset  \mathcal D_\sigma ^2$ for some fixed $\sigma\in S^*$. Then
\[
\sum_{R\in\mathcal L} a_R = \sum_{R\in\mathcal L} \sum_{t\in\tiles:\, R_t=R}a_t(f\ind_E)=\sum_{t\in\tiles _{\mathcal L}} a_t(f\ind_E)
\]
where $\tiles _{\mathcal L}\coloneqq\{t\in \tiles:\, R_t\in\mathcal L\}.$ By Lemma~\ref{l:localortho} the right hand side of the display above can be estimated by a constant multiple of $|\mathsf{sh}(\mathcal T)| \|f\ind_E\|_{\infty} ^2\leq  |\mathsf{sh}(\mathcal T)|$. This shows the desired property in the definition of a Carleson sequence.

Finally if $\tiles_{\sigma}\coloneqq \{t\in\tiles:\, s(t)=\sigma\}$  for $\sigma\in S^*$ we have that
\[
\sup_{\sigma\in S^*}\big\| \M_{\mathcal R_{\tiles_{\sigma}}}: \, L^p(\R^2)\to L^{p,\infty}(\R^2)\big\|\lesssim p',\qquad p\to 1^+.
\] 
Indeed note that for fixed direction $\sigma\in S^*$ each maximal operator appearing in the estimate above is bounded by the strong maximal function in the coordinates $(v,e_2)$ with $v=(1,\sigma)$.

Now Theorem \ref{thm:carleson} applies  to the Carleson sequence $a=\{a_R\}_{R\in\mathcal R_{\tiles}}$ yielding
\[
 \|\Delta_{\tiles} (f\cic{1}_E)\|_4^4 = \|T_{R_{\tiles}}(a)\|_2^2 \lesssim (\log \#S^*)(\log\log \#S^*) \mathsf{mass}_a\lesssim   (\log \#S)(\log\log \#S) |E|
\]
which is the claimed estimate \eqref{e:isfrest} as $\#S^*\simeq \#S$. The proof of Theorem~\ref{t:isf} is thus complete.
\end{proof}

\subsection{The intrinsic square function associated with smooth frequency cones}\label{sec:smoothtiles} The tiles in the previous subsection were used to model rough frequency projections on a collection of essentially disjoint cones. Indeed note that all decompositions were of Whitney type with respect to all the singular sets of the corresponding rough multiplier. In the case of smooth frequency projections on cones we need a simplified collection of tiles that we briefly describe below.

Assuming $S$ is a finite set of slopes and the arcs $\{\upomega_s\}_{s\in S}$ on $\mathbb S^1$ have finite overlap as before we now define for $s\in S$ and $k\in\Z$ the collections
\begin{equation}\label{eq:smoothconetiles}
\tiles_{s,k}\coloneqq\big\{t=R_t\times \Omega_{s,k}:\,R_t\in\mathcal D_{s,k-\ell_s,k} \big\},\qquad \tiles_s\coloneqq\bigcup_{k\in\Z} \tiles_{s,k},\qquad \tiles\coloneqq \bigcup_{s\in S} \tiles_s,
\end{equation}
with $\Omega_{s,k}$ given by \eqref{e:freqsconical}. Here we also assume that $2^{-\ell_s}\leq |\upomega_s|\leq 2^{-\ell_s+1}$. Notice that each conical sector $\Omega_{s,k}$ now generates exactly one frequency component of possible tiles in contrast with the previous subsection where we need a whole Whitney collection for every $s$ and every $k$; in fact the tiles $\tiles_{s,k}$ are for all practical purposes the same as the tiles $\tiles_{s,k,0}$ considered in \S\ref{sec:conetiles}. It is of some importance to note here that for each fixed $s\in S$ the collection $\mathcal R_{\tiles}\coloneqq\{R_t:\, t\in\tiles\}$ consists of parallelograms of fixed eccentricity $2^{\ell_s}$ and thus the corresponding maximal operator $\mathrm{M}_{\mathcal R_{\tiles_s}}$ is of weak-type (1,1) uniformly in $s\in S$:
\[
\sup_{s\in S}\big\| \M_{\mathcal R_{\tiles_s}}: \, L^1(\R^2)\to L^{1,\infty}(\R^2)\big\|\lesssim 1.
\] 
The intrinsic square function $\Delta_\tiles$ is formally given as in \eqref{e:intrinsicsf} but defined with respect to the new collection of tiles defined in \eqref{eq:smoothconetiles}. A repetition of the arguments that led to the proof of Theorem~\ref{t:isf} yield the following.

\begin{theorem} \label{thm:smoothconeintrinsic} For $\tiles$ defined by \eqref{eq:smoothconetiles} we have the estimates
\[
\begin{split}
& \big\|\Delta_{\tiles}: L^{p}(\R^2)   \big\|  \lesssim_p (\log \#S)^{\frac12-\frac1p}, \qquad 2\leq p<4, \\
& \sup_{E,f} \frac{\left\|\Delta_{\tiles} (f\cic{1}_E)\right\|_{4}}{|E|^{\frac14}} \lesssim (\log \#S)^{\frac14},
\end{split}
\]
where the supremum in the last display is taken over all measurable sets $E\subset \R^2$ of finite positive measure and all Schwartz functions $f$ on $\R^2$ with $\|f\|_\infty\leq 1$.
\end{theorem}

\subsection{The intrinsic square function associated with rough frequency rectangles}\label{sec:intrinsicrdf}
The considerations in this subsection aim at providing the appropriate time-frequency analysis in order to deal with a Rubio de Francia type square function, given by frequency projections on disjoint rectangles in finitely many directions. The intrinsic setup is described by considering again a finite set of slopes $S$ and corresponding directions $V$. Suppose that we are given a finitely overlapping collection of rectangles $\mathcal F=\cup_{s\in S}\mathcal F_s$, consisting of rectangles which are tensor products of intervals in the coordinates $v,v^\perp$, $v=(1,s)$, for some $s\in S$. Namely a rectangle $F\in \mathcal F_s$ is a \emph{rotation} by $s$ of an axis-parallel rectangle. We stress that the rectangles in each collection $\mathcal F_s$ are generic two-parameter rectangles, namely their sides have independent lengths (there is no restriction on their eccentricity).

We also note that $\mathcal F_s$ consists of rectangles rather than parallelograms and this difference is important when one deals with rough frequency projections. Our techniques are sufficient to deal with the case of parallelograms as well but we just choose to detail the setup for the rectangular case. The interested reader will have no trouble adjusting the proof for variations of our main statement below for the case of parallelograms, or for the case that the families $\mathcal F_s$ are in fact one-parameter families.
 
Given $F\in\mathcal F_s$ we define a two-parameter Whitney discretization as follows. Let $F=\mathrm{rot}_s(I\times J)+y_F$ for some $y_F\in \R^2$, where $\mathrm{rot_s}$ denotes counterclockwise rotation by $s$ about the origin and $I\times J$ is an axis parallel rectangle centered at the origin. Note that $I=(-|I|/2,|I|/2)$ and similarly for $J$. Then we define for $(k_1,k_2)\in\N^2$, $k_1,k_2\neq 0$,
\[
W_{k_1,k_2}(F)\coloneqq \Big\{\xi\in I\times J:\,  \frac13 2^{-k_1-1} \leq \frac12- \frac{|\xi_1|}{|I|}\leq \frac 13 2^{-k_1+1} ,\, \frac13 2^{-k_2-1} \leq \frac12- \frac{|\xi_2|}{|J|}\leq \frac 13 2^{-k_2+1}   \Big\}.
\]
The definition has to be adjusted for $k_1=0$ or $k_2=0$. For example we define for $k_2\neq 0$
\[
W_{0,k_2}(F)\coloneqq \Big\{\xi\in I\times J:\,  |I|/2-|\xi_1|\geq  \frac12 \frac13|I|,\, \frac13 2^{-k_2-1}|J| \leq |J|/2- |\xi_2|\leq \frac 13 2^{-k_2+1}  |J|\Big\}
\]
and symmetrically for $k_1\neq 0$ and $k_2=0$. Finally
\[
W_{0,0}(F)\coloneqq \Big\{\xi\in I\times J:\,|I|/2-|\xi_1|\geq  \frac12 \frac13|I|,\, |J|/2-|\xi_2|\geq  \frac12 \frac13|J| \Big\}.
\]
Then for $k=(k_1,k_2)\in \N^2 $ we set $\Omega_{s,k_1,k_2}(F)\coloneqq\mathrm{rot}_s(W_{k_1,k_2}(F))+y_F$. 

We can define tiles for this system as follows. If $F\in\mathcal F_s$ for some $s\in S$ and $F=\mathrm{rot}_s(I\times J)+y_F$ with $I\times J$ as above, then we choose $\ell_I ^F,\ell_J ^F\in\Z$ such that $2^{\ell^F _I}<|I| \leq 2^{\ell^F _I+1}$ and $2^{\ell^F _J}<|J| \leq 2^{\ell^F _J+1}$. We will have
\begin{equation}\label{eq:rdftiles}
\tiles^{\mathcal F}\coloneqq \bigcup_{s\in S}\tiles_s ^{{\mathcal F}},\qquad \tiles_{s} ^{\mathcal F} \coloneqq\bigcup_{F\in\mathcal F_s} \tiles  _{s}(F),\qquad \tiles_s(F)\coloneqq \bigcup_{(k_1,k_2)\in\N^2} \tiles_{s,k_1,k_2}(F),\quad F\in\mathcal F_s,
\end{equation}
where 
\[
\tiles_{s,k_1,k_2}(F)\coloneqq\Big\{t=R_t\times \Omega_{s,k_1,k_2}(F):\, R_t\in\mathcal D_{s,-k_2+\ell_J ^F,-k_1+\ell_I ^F}\Big	\},\qquad  F\in\mathcal F_s.
\]
Note again that the tiles defined above obey the uncertainty principle in both $v,v^\perp$ for every fixed $v=(1,s)$ with $s\in S$.

The intrinsic square function associated with the collection $\mathcal F$ is denoted by $\Delta_{\tiles^{\mathcal F}}$ and formally has the same definition as \eqref{e:intrinsicsf}, where now the $\tiles$ are given by the collection $\tiles ^{\mathcal F}$ of \eqref{eq:rdftiles}. The corresponding theorem is the intrinsic analogue of a multiparameter directional Rubio de Francia square function estimate.

\begin{theorem} \label{thm:intrrdf} Let $\mathcal F$ be a finitely overlapping collection of two-parameter rectangles in directions given by $S$
	\[
	\Big \| \sum_{F\in\mathcal F}\ind_F \Big \|_\infty \lesssim 1.
	\]
Consider the collection of tiles $\tiles^\mathcal F$ defined in \eqref{eq:rdftiles} and $\Delta_{\tiles^{\mathcal F}}$ be the corresponding intrinsic square function. We have the estimates
\begin{align} 
& \big\|\Delta_{\tiles^{\mathcal F}}: L^{p}(\R^2)   \big\|  \lesssim_p (\log \#S)^{\frac12-\frac1p}(\log\log\#S)^{\frac12-\frac1p}, \qquad 2\leq p<4, \\
& \sup_{E,f}
  \frac{\left\|\Delta_{\tiles^{\mathcal F}} (f\cic{1}_E)\right\|_{4}}{|E|^{\frac14}} \lesssim (\log \#S)^{\frac14} (\log\log\#S)^{\frac14},
\end{align}
where the supremum in the last display is taken over all measurable sets $E\subset \R^2$ of finite positive measure and all Schwartz functions $f$ on $\R^2$ with $\|f\|_\infty\leq 1$.
\end{theorem}

\begin{remark}\label{rmrk:1paramrecs} As before, there is slight improvement in the case of one-parameter spatial components in each direction. More precisely suppose that $\mathcal F=\cup_{s\in S}\mathcal F_s$ is a given collection of disjoint rectangles in directions given by $S$. If for each $s\in S$ the family $\mathcal R_{\mathcal F_s}\coloneqq\{R_t:\, t\in\tiles_{\mathcal F_s}\}$ yields a weak-type $(1,1)$ maximal operator then the estimates of Theorem~\ref{thm:intrrdf} hold without the $\log\log$-terms.	
\end{remark}

\begin{remark}\label{rmrk:slanted} Suppose that $\mathcal R=\bigcup_{s\in S}\mathcal R_s\subset \mathcal P_S ^2$ is a family of parallelograms in directions given by $s$, namely we have that if $R\in\mathcal R_s$ then $R=A_s (I\times J)+y_R$ for some rectangle $I\times J$ in $\R^2$ with sides parallel to the coordinate axes and centered at $0$, and $y_R\in \R^2$. Now there is an obvious way to construct a Whitney partition of each $R\in\mathcal R$. Indeed we just define the frequency components 
	\[
	\Omega_{s,k_1,k_2}(R)\coloneqq A_s(W_{k_1,k_2}(I\times J))+y_R
	\]
with $W_{k_1,k_2}(I\times J)$ as constructed before. Then
	\[
	\tiles_{s,k_1,k_2}(R)\coloneqq\Big\{R_t\times \Omega_{s,k_1,k_2}(R):\, R_t\in\mathcal D_{s,-k_2+\ell_J ^F,-k_1+\ell_I ^F}\Big	\},\qquad  R\in\mathcal R_s
\]
and $\tiles$ are given as in \eqref{eq:rdftiles}. With this definition there is a corresponding intrinsic square function $\Delta_{\tiles_\mathcal R}$ which satisfies the bounds of Theorem~\ref{thm:intrrdf}. The improvement of Remark~\ref{rmrk:1paramrecs} is also valid if $\mathcal R=\cup_{s\in S}\mathcal R_s$ and each $\mathcal R_s$ consists of rectangles of fixed eccentricity.
\end{remark}

The proof of Theorem~\ref{thm:intrrdf} relies again on the global and local orthogonality estimates of \S\ref{sec:orthocone} and a subsequent application of the directional Carleson embedding theorem, Theorem~\ref{thm:carleson}. We omit the details.

\section{Sharp bounds for  conical square  functions} \label{sec:cone} 
We begin this section by recalling the definition for the smooth conical frequency projections given in the introduction. Let $\uptau\subset (0,2\uppi)$ be an interval  and consider the corresponding rough cone multiplier
\[
C_\uptau f(x)   \coloneqq \int_{0}^{2\uppi} \int_{0}^\infty  \widehat f(\varrho\e^{i\vartheta}) \cic{1}_{\uptau}(\vartheta)  \e^{ix\cdot\varrho\e^{i\vartheta} } \, \varrho\d \varrho \d \vartheta,\qquad x\in\R^2,
\]
and its smooth analogue
\begin{equation}
\label{e:smoothcones}
C_\uptau^\circ f(x)    \coloneqq \int_{0}^{2\uppi} \int_{0}^\infty  \widehat f(\varrho\e^{i\vartheta}) \upbeta\left(\frac{\vartheta-c_\uptau}{|\uptau|/2}\right)  \e^{ix\cdot\varrho\e^{i\vartheta} } \, \varrho\d \varrho \d \vartheta,\qquad x\in\R^2.
\end{equation}
where $\upbeta$ is a smooth function on $\R$ supported on $[-1,1]$ and equal to 1 on $[-\frac12,\frac12]$ and  $c_\uptau$, $|\uptau|$ stand respectively for the center and  length of $\uptau$.

This section is dedicated to the proofs of two related  theorems concerning conical  square functions.  The first is a quantitative estimate for a  square function associated with the  smooth conical multipliers of a finite collection of intervals with bounded overlap given in Theorem~\ref{thm:smoothcones}, namely the estimates
\begin{align}
 \label{e:sfsmooth}& \left\| \{C_\uptau^\circ f\}  \right\|_{L^p(\R^2; \ell^2_{\cic{\uptau}})} \lesssim_p (\log \# \cic{\uptau} )^{\frac{1}2-\frac1p}\|f\|_p
\end{align}
for $2\leq p<4$,
as well as the restricted type analogue valid for all measurable sets $E$
\begin{align}
 \label{e:sfsmoothrs}&  \left\| \{C_\uptau^\circ(f\cic{1}_E)\}  \right\|_{L^4(\R^2; \ell^2_{\cic{\uptau}})}  \lesssim (\log \# \cic{\uptau} )^{\frac{1}4}|E|^{\frac14}\|f\|_\infty,
\end{align}
under the assumption of finite overlap 
\begin{equation} \label{e:bo}
\Bigl\| \sum_{\uptau\in \cic{\uptau}} \ind_{\uptau} \Bigr\|_\infty  \lesssim 1.
\end{equation}

The second theorem concerns an estimate for the rough conical square function for a collection of finitely overlapping cones $\cic{\uptau}$.

\begin{theorem} \label{thm:wdcones} Let $\cic{\uptau}$ be a finite collection intervals in $[0,2\pi)$ with finite overlap as in \eqref{e:bo}. Then the square function estimate 
\begin{align} \label{e:sfrough}
& \left\| \{C_\uptau  f\}  \right\|_{L^p(\R^2; \ell^2_{\cic{\uptau}})} \lesssim_p (\log \# \cic{\uptau} )^{1-\frac2p}(\log\log\#\cic{\uptau})^{\frac12-\frac1p}\|f\|_p.
\end{align}
holds for each $2\leq p<4.$
\end{theorem}
Theorem   \ref{thm:smoothcones} is sharp, in terms of   $\log \# \cic{\upomega}$-dependence, for all $2\leq p<4$ and for $p=4$ up to the restricted type.
Theorem \ref{thm:wdcones} improves on \cite[Theorem 1]{CorGFA}, where the dependence on cardinality  is unspecified.
Examples providing a lower bound  of $(\log \# \cic{\upomega} )^{\frac12-\frac1p}\|f\|_p$ for the left hand side of \eqref{e:sfrough}, and showing the sharpness of Theorem   \ref{thm:smoothcones}, are detailed in Section \ref{sec:cex}. 

The remainder of the section is articulated as follows. In the upcoming Subsection~\ref{sec:proofA} we  show Theorem \ref{thm:smoothcones}.  The subsequent subsection is dedicated to the proof of Theorem \ref{thm:wdcones}.

\subsection{Proof of Theorem \ref{thm:smoothcones}}\label{sec:proofA} We are given a finite collection of intervals $\upomega\in \cic{\upomega}$ having bounded overlap as in \eqref{e:bo}. By finite splitting we may reduce to the case of $ \upomega\in \cic{\upomega}$ being pairwise disjoint; we treat this case throughout.

The first step in the proof of Theorem \ref{thm:smoothcones} is a radial decoupling. Let $\uppsi$ be a smooth radial function on $\R^2$ with
\[
\ind_{[1,2]}(|\xi|)\leq  \uppsi(\xi) \leq \ind_{[2^{-1},2^{2}]}(|\xi|)
\]
and  define the Littlewood-Paley projection
\[
S_kf(x)\coloneqq \int \uppsi(2^{-k}\xi)\widehat f(\xi) \, \e^{ix\cdot \xi} \, \d \xi, \qquad x\in \R^2.
\] 
The following weighted Littlewood-Paley inequality is contained in \cite{BeHa}*{Proposition 4.1}.

\begin{proposition}[Bennett-Harrison, \cite{BeHa}]\label{prop:beha} Let $w$ be a non-negative locally integrable function. 
\[
	\int_{\R^2} | f|^2 w \lesssim \int_{\R^2} \sum_{k\in\mathbb Z} |S_k(f)|^2 \M^{[3]}w
\]
with implicit constant independent of $w,f$, where we recall that $\M^{[3]}$ denotes the three-fold iteration of the Hardy- Littlewood maximal function $\M$ with itself.
\end{proposition}

We may easily deduce the next lemma from the proposition.
\begin{lemma}\label{lem:wLP} For any $p\geq 2$ we have
\begin{equation}
\label{e:radialdec}
\left\| \{C_{\uptau}^\circ  f\}  \right\|_{L^p(\R^2; \ell^2_{\cic{\uptau}})} \lesssim \Big\|\Big(\sum_{k\in\mathbb Z,\, \uptau\in \cic{\uptau}} |  C_\uptau^\circ S_k(f)|^2 \Big)^\frac12\Big\|_{p}.
\end{equation}
\end{lemma}

\begin{proof} The case $p=2$ is trivial so we assume $p>2$. Letting $r\coloneqq \frac p2>1$ there exists some $w\in L^{r'}(\R^2)$ with $\|w\|_{r'}=1$ such that
	\[
\left\| \{C_{\uptau}^\circ  f\}  \right\|_{L^p(\R^2; \ell^2_{\cic{\uptau}})}^2= \sum_{\uptau\in \cic{\uptau}} \int_{\R^2} |C_\uptau^\circ f |^2 w\lesssim \sum_{k\in\mathbb Z,\, \uptau\in \cic{\uptau}}\int_{\R^2}|C_\uptau^\circ S_k(f)|^2 M^{[3]}w
	\]
and the lemma follows by H\"older's inequality and the boundedness of $M^{[3]}$ on $L^{r'}(\R^2)$.
\end{proof}

The second and final step of the proof of Theorem \ref{thm:smoothcones} is the reduction of the operator appearing in the right hand side of \eqref{e:radialdec} to the model operator of Theorem \ref{thm:smoothconeintrinsic}. 

In order to match the notation of \S\ref{sec:conetiles} we write $\{\omega_s\}_{s\in S}$ for the collection of arcs in $\mathbb S^1$ corresponding to the collection of intervals $\cic{\uptau}$, namely for $\uptau\in\cic{\uptau}$ we implicitly define $s=s_\uptau$ by means of $v_s ^\perp/|v_s ^\perp|\coloneqq \e^{ic_\uptau}={(1,s)}/{|(1,s)|}$. We set $S\coloneqq\{s_\uptau:\, \uptau\in\cic{\uptau}\}$ and define the corresponding arcs in $\mathbb S^1$ as
\[
\upomega_{s_\uptau}\coloneqq\{\e^{i\theta}:\, \theta\in\uptau\}.
\]
Now the cone $C_\uptau$ is the same thing as the cone $C_s$ and $\#S=\#\cic{\uptau}$. Similarly we write $C^\circ _\uptau=C^\circ _{s_\uptau}$ so the cones can now be indexed by $s\in S$. Define $\ell_s$ such that $2^{-\ell_s} \leq |\omega_s|\leq 2^{-\ell_s+1}$.

By finite splitting and rotational invariance there is no loss in generality with assuming  that $S\subset[-1,1]$. Notice that the support of the multiplier of $C_{s}^\circ   S_k$ is contained in the frequency sector $\Omega_{s,k}$ defined in \eqref{e:freqsconical}.  
By standard procedures of time-frequency analysis, as for example in \cite[Section 6]{DDP}, the operator $C_{s}^\circ   S_k$ can  be recovered by appropriate averages of operators
\begin{equation}
\mathsf{C}_{s,k} f \coloneqq 
\sum_{t \in \tiles_{s,k}} \langle f,\upphi_t\rangle \upphi_t
\end{equation}
where $\upphi_t\in \mathcal{A}_t^{8M_0}$ for all $t\in \tiles_{s_\upomega,k}$ and $\tiles_{s,k}$ is defined in \eqref{eq:smoothconetiles}. Here $M_0=2^{50}$ is as chosen in  \eqref{e:intwcdef}. Fixing $s,k$ for the moment we preliminarily observe that for each $\upnu\geq 1$ the collection $\mathcal R_{s,k}\coloneqq \mathcal R_{\tiles_{s,k}}= \{R_t:\,t\in\tiles_{s,k} \}$ can be partitioned  into subcollections $ \{\mathcal R_{s,k,\upnu}^j: \, 1\leq j \leq 2^{8\upnu}\}$ with the property that
 \[
R_1,R_2 \in  \mathcal R_{s,k,\upnu}^j \implies 2^{2\upnu+4}R_1 \cap 2^{2\upnu+4} R_2 =\varnothing.
 \]
We will also use below the Schwartz decay of $\upphi_t\in \mathcal A_{t}^{M_0}$ in the form
 \[
\sqrt{|R_t|} |\upphi_t  | \lesssim \cic{1}_{R_t} + \sum_{\upnu\geq 0} 2^{-8M_0\upnu}   \sum_{\substack{\uprho \in \mathcal{R}_{s,k} \\  \uprho \not \subset 2^{\upnu} R_t, \, \uprho \subset 2^{\upnu+1}R_t}}   \cic{1}_\uprho  .\]
Using Schwartz decay of $\upphi_t$ twice, in particular  to bound by an absolute constant the second factor obtained by Cauchy-Schwartz after the first step, we get
\[
\begin{split}
&\quad |\mathsf{C}_{s,k} f  |^2    \lesssim 	\left(\sum_{t \in \tiles_{s,k}} | \langle   f,\upphi_t\rangle|^2 \frac{|\upphi_t  |}{\sqrt{|R_t|}} \right)	\left(\sum_{t \in \tiles_{s,k}}   \sqrt{|R_t|} |\upphi_t| \right)  
\\ 
& \lesssim \sum_{t \in \tiles_{s,k}} | \langle   f,\upphi_t\rangle|^2 \frac{ \ind_{R_t}   }{ |R_t| }+\sum_{\upnu \geq 0}  2^{-8M_0\upnu}   \sum_{t \in \tiles_{s,k}} \sum_ {\substack{\uprho \in \mathcal{R}_{s,k} \\  \uprho \not \subset 2^{\upnu} R_t, \uprho \subset 2^{\upnu+1}R_t}}  | \langle   f,\upphi_t\rangle|^2\frac{ \cic{1}_{\uprho}}{|\uprho|}  
\\
&\leq \sum_{t\in \tiles_{s,k}} | \langle   f,\upphi_t\rangle|^2 \frac{ \ind_{R_t}   }{ |R_t| }+ \sum_{\upnu \geq 0}  2^{-8M_0\upnu}  \sum_{R\in\mathcal R_{s,k}}  \sum_{j=1}^{2^{8\upnu}} \sum_{\substack{\uprho \in \mathcal{R}_{s,k} \\  \uprho \not \subset 2^{\upnu} R_t, \uprho \subset 2^{\upnu+1}R_t}}   |\langle   f,\upphi_t\rangle|^2 \frac{\ind_\uprho}{|\uprho|}.
\end{split}
\]
Now for fixed $\upomega,k,\upnu,j$ and $t\in \tiles_{s,k}$ observe that there is at most one $\uprho=\uprho_{s,k,\upnu}^j(t)\in\mathcal R_{\upomega,k,\upnu} ^j$ such that $ \uprho \not\subset 2^{\upnu} R_t, \uprho \subset 2^{\upnu+1}R_t$. Thus the estimate above can be written in the form
\begin{equation}\label{eq:R+R}  
\quad |\mathsf{C}_{s,k} f  |^2    \lesssim \sum_{t\in \tiles_{s,k}} | \langle   f,\upphi_t\rangle|^2 \frac{ \ind_{R_t}   }{ |R_t| }+ \sum_{\upnu \geq 0} 2^{-8M_0 \upnu}\sum_{j=1}^{2^{8\upnu}}  \sum_{t\in \tiles_{s,k}}    |\langle   f,\upphi_{t}\rangle|^2 \frac{\ind_{\uprho_{s,k,\upnu}^j(t)}}{|\uprho_{s,k,\upnu}^j(t)|}.
\end{equation}

Observe that if  $t\in \tiles_{s,k}$
\[
\upphi_{t}\in\mathcal A_{t} ^{8M_0},\quad  \uprho  \in \mathcal R_{s,k},\quad \uprho \subset 2^{\upnu+1}R_t  \implies  2^{-4M\upnu}|\langle   f,\upphi_{t}\rangle|^2 \leq a_{t_{\uprho}}(f)
\]
where $t_\uprho=\uprho \times \Omega_{s,k} \in   \tiles_{s,k} $ is the unique tile with spatial localization given by  $\uprho$: this is because $2^{-4M\upnu}\upphi_t\in \mathcal A_{t_\uprho}^{M_0}$. We thus conclude that
\begin{equation}
\label{e:Cvkdec}
  |\mathsf{C}_{s,k} f|^2 \lesssim \sum_{t\in \tiles_{s,k}}  a_{t}(f) \frac{ \cic{1}_{R_t}}{|R_t|}.
\end{equation}
Comparing with the definition of $\Delta_\tiles$ given in \eqref{e:intrinsicsf} we may  summarize the  discussion in the  lemma below.

\begin{lemma} \label{l:modelsumconical}  Let $1<p<\infty$. Then
\[
\sup_{\|f\|_p=1} \Big\|\Big(\sum_{k\in\mathbb Z,\, \uptau\in \cic{\uptau}} |  C_\uptau^\circ S_k(f)|^2 \Big)^\frac12\Big\|_{p} \lesssim  \sup_{\|f\|_p=1} \big\|\Delta_{\tiles} (f) \big\|_{p}
\]
where 
\[
\tiles \coloneqq \bigcup_{s\in S}\bigcup_{k\in\Z} \tiles_{s,k}
\]
and $\tiles_{s,k}$ is defined in \eqref{eq:smoothconetiles}.
\end{lemma} 
The proof of the upper bound in Theorem \ref{thm:smoothcones} is then completed by juxtaposing the estimates of Lemmata \ref{lem:wLP} and \ref{l:modelsumconical} with  Theorem \ref{thm:smoothconeintrinsic}. For the optimality of the estimate see \S\ref{sec:cexconical}.

\subsection{Proof of Theorem~\ref{thm:wdcones}} \label{ss:rts} The proof of Theorem~\ref{thm:wdcones} is necessarily more involved than its smooth counterpart Theorem~\ref{thm:smoothcones}. In particular we need to decompose each cone not only in the radial direction as before, but also in the directions perpendicular to the singular boundary of each cone. We describe this procedure below.

Consider a collection of intervals $\cic{\uptau}=\{\uptau\}$ as in the statement. By the same correspondence as in the proof of Theorem~\ref{thm:smoothcones} we pass to a family $\{\upomega_s\}_{s\in S}$ consisting of finitely overlapping arcs on $\mathbb S^1$  centered at $v_s ^\perp/|v_s ^\perp|$ and corresponding cones $C_s$. Note that the sectors $\{\Omega_{s,k}\}_{s\in S,k\in\Z}$, defined in \eqref{e:freqsconical} form a finitely overlapping cover of $\cup_{s\in S}C_s$. We remember here that $v_s=(1,s)$ and the endpoint of the interval $\upomega_s$ are given by $(v_{s^-} ^\perp,v_{s^+} ^\perp)$, and that the positive direction is counterclockwise.

Now, for each fixed $s\in S$ the cover $\{\Omega_{s,k,m}\}_{(k,m)\in\Z^2}$ defined in \eqref{eq:whitney}, \eqref{eq:whitney0}, is a Whitney cover if $\Omega_{s,k}$ in the product sense: for each $\Omega_{s,k,m}$ the distance from the origin is comparable to $2^k$ and the distance to the boundary is comparable to $2^{-|m||\upomega_s|}$. 

The radial decomposition in $k$ will be taken care of by the Littlewood-Paley decomposition $\{S_k\}_{k\in\Z}$, defined as in the proof of Theorem~\ref{thm:wdcones}. Now for fixed $s,k$ we consider a smooth partition of unity subordinated to the cover $\{\Omega_{s,k,m}\}_{m\in\Z}$. Note that one can easily achieve that by choosing  $\{ \upvarphi_{s,m}\}_{m<0}$ to be a one-sided (contained in $C_s$) Littlewood-Paley decomposition in the negative direction $v^-=v_{s^-}$ , and constant in the direction $(v^-)^\perp$ when $m<0$, and similarly one can define $ \upvarphi_{s,m}$ when $m>0$, with respect to the positive direction $v^+$. The central piece $\Omega_{s,k,0}$ corresponds to $ \upvarphi_{s_0}$ defined implicitly as 
\[
\upvarphi_{s,0} =\ind_{C_s}- \sum_{m\in\Z}\upvarphi_{s,m} .
\]
Now the desired partition of unity is $\pi_{s,k,m}(\xi)\coloneqq \ind_{C_s}(\xi)\upvarphi_{s,m}(\xi)\uppsi_k(\xi)=\upvarphi_{s,m}(\xi)\uppsi_k(\xi)$, where $\uppsi_k\coloneqq \uppsi(2^{-k}\cdot)$ with the $\uppsi$ constructed in the proof of Theorem~\ref{thm:smoothcones}. Remember that $ S_k f \coloneqq (\uppsi_k\hat f)^\vee$ and let us define $ \Phi_{s,m}f  \coloneqq (\upvarphi_{s,m} \hat f)^\vee$.

An important step in the proof is the following square function estimate in $L^p(\R^2)$, with $2\leq p<4$, that decouples the Whitney pieces in every cone $C_s$. It comes at a loss in $N$ which appears to be inevitable because of the directional nature of the problem.

\begin{lemma}\label{lem:decoupling} Let $\{C_s\}_{s\in S}$ be a family of frequency cones, given by a family of finitely overlapping arcs $\cic{\upomega}\coloneqq\{\upomega_s\}_{s\in S}$ as above. For $2\leq p <4$ there holds
	\[
\big\|\{C_sf\}|\big\|_{L^p(\R^2;\ell^2 _{\cic{\upomega}})}\lesssim \frac{1}{4-p}(\log\#S)^{\frac12-\frac{1}{p}} \|\{S_k \Phi_{s,m} f\}\|_{L^p(\R^2;\ell^2 _{\cic{\upomega}\times\Z\times\Z\}})}.
	\]
\end{lemma}

\begin{proof} Observe that the desired estimate is trivial for $p=2$ so let us fix some $p\in (2,4)$. There exists some $g\in L^{q}$ with $q=(p/2)'=p/(p-2)$ such that
\[
A^2\coloneqq \big\|\{C_sf\}|\big\|_{L^p(\R^2;\ell^2 _{\upomega})} ^2 = \int_{\R^2}\sum_{s\in S}|C_s f|^2 g	
\]
and so by Proposition~\ref{prop:beha} we get 
\[
A^2\lesssim \sum_{k\in\Z} \sum_{s\in S} \int_{\R^2}|C_s S_k f|^2 \M^{[3]}g
\]
where we recall that $\M^{[3]}$ denotes three iterations of the Hardy-Littlewood maximal function $\M$. Fixing $s$ for a moment we use Proposition~\ref{prop:beha} in the directions $v_{s^-},v_s$ and $v_{s^+}$ to further estimate
\[
\int_{\R^2}|C_s f|^2 \M^{[3]}g\lesssim \sum_{m\in\Z}\sum_{\upvarepsilon\in\{-,0,+\}}\int_{\R^2}| S_k\Phi_{s,m} f|^2 \M_{v_{s^\upvarepsilon} } ^{[3]}\M^{[3]}g 
\]
where we adopted the convention $s^0\coloneqq s$ for brevity, and $\M_v$ is given by \eqref{eq:dirMv}. Remember also that $\Phi_{s,m}$ for $m>0$ corresponds to directions $s^+$ while $\Phi_{s,m}$ corresponds to directions $s^-$ for $m<0$, and to directions $s^0=s$ for $m=0$. Now for any $v\in \mathbb S^1$ and $r>1$ we have that
\[
\M_v ^{[3]} G \lesssim (r')^2 [\M_v G^r]^{\frac1r};
\]
see for example \cite{CP}. Thus $\M_{v_{s^\upvarepsilon} } ^{[3]}\M^{[3]}g \lesssim (r')^2 [\M_{V^*} [\M^{[3]} G]^r]^{\frac1r}$ where $\M_{V^*}f\coloneqq \sup_{v\in V^*}\M_v f$, where here we use $V^*\coloneqq \{(1,s):\, s\in S^*\}$ with $S^*$ as in \eqref{eq:s*}, and $\M_{V^*}f\coloneqq\sup_{w\in V^*}\M_w(f)$.

It is known \cite{KatzDuke} that $M_{V^*}$ maps $L^p(\R^2)$ to $L^p(\R^2)$ with a bound $(\log\#V^*)^{\frac1p}$  for $p>2$. As $p<4$ there exists a choice of $1<r<\frac{p}{2(p-2)}$ so that $\frac{p}{r(p-2)}>2$ and a theorem of Katz from \cite{KatzDuke} applies. Using this fact together with H\"older's inequality  proves the lemma.
\end{proof}

The proof of Theorem~\ref{thm:wdcones} can now be completed as follows. For each $(s,k,m)\in S\times\Z\times \Z$ the operator $S_k \Phi_{s,m}$ is a smooth frequency projection adapted to the rectangular box $\Omega_{s,k,m}$. Following the same procedure that led to \eqref{e:Cvkdec} in the proof of Theorem~\ref{thm:smoothcones} we can approximate each piece $S_k \Phi_{s,m} f$ by an operator of the form 
\[
\mathsf{C}_{s^\upvarepsilon,k,m} f \coloneqq \sum_{t \in \tiles_{s^\upvarepsilon,k,m}} \langle f,\upphi_t\rangle \upphi_t,\qquad |\mathsf{C}_{s^\upvarepsilon,k,m} f|^2\lesssim \sum_{t \in \tiles_{s^\upvarepsilon,k,m}} a_t(f)\frac{\ind_{R_t}}{|R_t|},
\]
where $s^\upvarepsilon$ follows the sign of $m$ and coincides with $s$ if $m=0$. The collections of tiles $\tiles_{s^\upvarepsilon,k,m}$ are the ones given in \eqref{eq:tilesskm}. Now Lemma~\ref{lem:decoupling} and Theorem~\ref{t:isf} are combined to complete the proof of Theorem~\ref{thm:wdcones}.

\section{Directional Rubio de Francia square functions}\label{sec:rdf}In his seminal paper \cite{RdF}, Rubio de Francia proved a one-sided Littlewood-Paley inequality for arbitrary intervals on the line. This estimate was later extended by Journ\'e, \cite{Journe}, to the  case of rectangles ($n$-dimensional intervals) in $\R^n$; a proof more akin to the arguments of the present paper appears in \cite{LR}. The aim of this subsection is to present a generalization of the one-sided Littlewood-Paley inequality to the case of rectangles in $\R^2$ with sides parallel to a given set of directions. The set of directions is to be finite, necessarily, because of Kakeya counterexamples. 

As in the case of cones of \S\ref{sec:cone} we will present two versions, one associated with smooth frequency projections and one with rough. To set things up let $S$ be a finite set of slopes and $V$ be the corresponding directions. We consider a family of rotated rectangles $\mathcal F$ as in \S\ref{sec:intrinsicrdf} where $\mathcal F=\cup_{s\in S}\mathcal F_s$. For each $s\in S$ a rectangle $F\in\mathcal F_s$ is a rotation by $s$ of  an axis parallel rectangle, so that the sides of $R$ are parallel to $(v,v^\perp)$ with $v=(1,s)$. We will write $F=\mathrm{rot}_s(I_F\times J_F)+y_F$ for some $y_F\in \R^2$ in order to identify the axes-parallel rectangle $I_F\times J_F$ producing $F$ by an $s$-rotation; this writing assumes that $I_F\times J_F$ is centered at the origin.

Now for each $F\in\mathcal F$ we consider the rough frequency projection
\[
P_Ff(x)\coloneqq \int_{\R^2} \hat f(\xi)\ind_F(\xi) e^{ix\cdot\xi}\,\d \xi,\qquad x\in \R^2,
\]
and its smooth analogue
\[
P_F ^\circ f(x)\coloneqq \int_{\R^2} \hat f(\xi) \upgamma_F(\xi)e^{ix\cdot\xi}\, \d\xi,\quad x\in\R^2,
\]
where $\upgamma_R$ is a smooth function on $\R^2$, supported in $R$, and identically $1$ on $\mathrm{rot_s}(\frac12 I \times \frac12 J)$.

We first state the smooth square function estimate.

\begin{theorem}\label{thm:rdfsmooth} Let $\mathcal F$ be a collection of rectangles in $\R^2$ with sides parallel to $(v,v^\perp)$ for some $v$ in a finite set of directions $V$. Assume that $\mathcal F$ has finite overlap. Then
\begin{align}
& \left\| \{P_F ^\circ f\}  \right\|_{L^p(\R^2; \ell^2_{\mathcal F})} \lesssim_p (\log \#V )^{\frac{1}2-\frac1p}(\log\log \#V )^{\frac{1}2-\frac1p}\|f\|_p 
\end{align} 
for $2\leq p<4$, as well as the restricted type analogue valid for all measurable sets $E$ 
\begin{align}
&  \left\| \{P_F ^\circ(f\cic{1}_E)\}  \right\|_{L^4(\R^2; \ell^2_{\mathcal F})}  \lesssim (\log \#V )^{\frac{1}4}(\log\log \#V )^{\frac14}|E|^{\frac14}\|f\|_\infty.
\end{align}
The dependence on $\#V$ in the estimates above is best possible up the doubly logarithmic term.	
\end{theorem}

\begin{remark}\label{rmrk:1paramrdf} We record a small improvement of the estimates above in some special cases. Suppose that for fixed $s\in S$ all the rectangles  $F\in\mathcal F_s$ have one side-length fixed, or that they have fixed eccentricity. In both these cases the collections of spatial components of the tiles needed to discretize these operators, $\mathcal R_{\tiles_s ^\mathcal F}\coloneqq\{R_t:\, t\in\tiles_s ^{\mathcal F}\}$, with $\tiles_s$ as in \eqref{eq:rdftiles}, give rise to maximal operators that are of weak-type $(1,1)$. Then Remark~\ref{rmrk:1paramrecs} shows that the estimates of Theorem~\ref{thm:rdfsmooth} hold without the doubly logarithmic terms, and as shown in \S\ref{sec:rdfcex} this is best possible.	
\end{remark}

The rough version of this Rubio de Francia type theorem is slightly worse in terms of the dependence on the number of directions. The reason for that is that, as in the case of conical projections, passing from rough to smooth in the directional setting incurs a loss of logarithmic terms, essentially originating in the corresponding maximal function bound.

\begin{theorem}\label{thm:rdfrough} Let $\mathcal F$ be a collection of rectangles in $\R^2$ with sides parallel to $(v,v^\perp)$ for some $v$ in a finite set of directions $V$. Assume that $\mathcal F$ has finite overlap. Then the following square function estimate holds for $2\leq p<4$
\begin{align}
& \left\| \{P_F  f\}  \right\|_{L^p(\R^2; \ell^2_{\mathcal F})} \lesssim_p (\log \#V )^{\frac{3}{2}-\frac3p} (\log\log\#V)^{\frac12-\frac1p}\|f\|_p .
\end{align} 
\end{theorem}

The proofs of these theorems follow the by now familiar path of introducing local Littlewood-Paley decompositions on each multiplier, approximating with time-frequency analysis operators, establishing a directional Carleson condition on the wave-packet coefficients and finally applying Theorem~\ref{thm:carleson}. We will very briefly comment on the proofs below.

\begin{proof}[Proof of Theorem~\ref{thm:rdfrough} and Theorem~\ref{thm:rdfsmooth}] We first sketch the proof of Theorem~\ref{thm:rdfrough} which is slightly more involved. The first step here is a decoupling lemma which is completely analogous to Lemma~\ref{lem:decoupling} with the difference that now we need to use two directional Littlewood-Paley decompositions while in the case of cones only one. This explains the extra logarithmic term of the statement.
	
Remember that $\mathcal F=\cup_s\mathcal F_s$ with $s=(1,v)$ for some $v\in V$; here $s$ gives the directions $(v,v^\perp)$ of the rectangles in $\mathcal F_s$. Using the finitely overlapping Whitney decomposition of \S\ref{sec:intrinsicrdf} we have for each $F\in\mathcal F_s$ a collection of tiles 
\[
\tiles_s(F)=\bigcup_{(k_1,k_2)\in\Z^2}\tiles_{s,k_1,k_2}(F)
\]
as in \eqref{eq:rdftiles}. Let us for a moment fix $s$ and $F\in\mathcal F_s$. The frequency components of the tiles in $\tiles_s(F)$ form a two-parameter Whitney decomposition of $F$, so let $\{\phi_{F,k_1,k_2}\}_{(k_1,k_2)\in\Z^2}$ be a smooth partition of unity subordinated to this cover and denote by $\Phi_{F,k_1,k_2}$ the Fourier multiplier with symbol $\phi_{F,k_1,k_2}$.

The promised analogue of Lemma~\ref{lem:decoupling} is the following estimate: for $2\leq p<4$ there holds
\begin{equation}\label{eq:2paramdecoupling}
\big\|\{P_F f\}|\big\|_{L^p(\R^2;\ell^2 _{\mathcal F})}\lesssim \frac{1}{(4-p)^2}(\log\#V)^{1-\frac{2}{p}} \|\{\Phi_{s,k_1,k_2} f\}\|_{L^p(\R^2;\ell^2 _{\mathcal F\times\Z\times\Z\}})}.
\end{equation}
The proof of this estimate is a two-parameter repetition of the proof of Lemma~\ref{lem:decoupling}, where one applies Proposition~\ref{prop:beha} once in the direction of $v$ and once in the direction of $v^\perp$. Using the familiar scheme we can approximate each $\Phi_{s,k_1,k_2}f$ by time-frequency analysis operators
\[
\mathsf{P}_{F,k_1,k_2}f\coloneqq \sum_{t\in \tiles_{s,k_1,k_2}(F)}\langle f,\upphi_t\rangle \upphi_t, \qquad|\mathsf{P}_{F,k_1,k_2}f|^2\lesssim \sum_{t\in \tiles_{s,k_1,k_2}(F)} a_t(f)\frac{\ind_{R_t}}{|R_t|}
\]
and by \eqref{eq:2paramdecoupling} the proof of Theorem~\ref{thm:rdfrough} follows by corresponding bounds for the intrinsic square function of Theorem~\ref{thm:intrrdf}, defined with respect to the tiles $\tiles^{\mathcal F}$ given by \eqref{eq:rdftiles}.

For Theorem~\ref{thm:rdfsmooth} things are a bit simpler as the decoupling step of \eqref{eq:2paramdecoupling} is not needed. Apart from that one needs to consider for each $F$ a new set of tiles which is very easy to define: If $F\in\mathcal F_s$ with $F=\mathrm{rot}_s(I_F\times J_F)+y_F$
\[
\tiles'(F)\coloneqq \big\{t=R_t\times F: \, R_t\in\mathcal D^2 _{s,\ell_J,\ell_I}\big\}
\]
and then $\tiles'\coloneqq \cup_{F\in\mathcal F}\tiles' (F)$. One can recover $P_F ^\circ$ by operators of the form
\[
\mathsf{P}^\circ _F f\coloneqq \sum_{t\in\tiles_s(F)} \langle f,\upphi_t\rangle \upphi_t,\qquad |\mathsf{P}^\circ _F f|^2\lesssim \sum_{t\in\tiles_s(F)} a_t(f)\frac{\ind_{R_t}}{|R_t|}
\]
as before. Using the orthogonality estimates of \S\ref{sec:orthocone} in Theorem~\ref{thm:carleson} yields the upper bound in Theorem~\ref{thm:rdfsmooth}. The optimality of the estimates in the statement of Theorem~\ref{thm:rdfsmooth} is discussed in \S\ref{sec:rdfcex}.
\end{proof}

\section{The multiplier problem for the polygon}\label{sec:polygon} Let $\mathcal P=\mathcal P_{N}$ be a regular $N$-gon and $T_{\mathcal{P}_N}$ be the corresponding Fourier restriction operator on $\mathcal P$
\[
T_{\mathcal P}f (x) \coloneqq  \int_{\R^2} \widehat f(\xi) \ind_{\mathcal P}(\xi) \e^{ix\cdot \xi} \, \d \xi,\qquad x\in\R^2.
\]
In this subsection we prove Theorem~\ref{thm:polygon}, namely we will prove the estimate
\[
\left\|T_{\mathcal P_N} : L^p(\R^2) \right\| \lesssim (\log N)^{4\left|\frac12-\frac1p\right|} , \qquad \frac43 <p <4.
\]
The idea is to reduce the multiplier problem for the polygon to the directional square function estimates of Theorem~\ref{thm:rdfsmooth} and combine those with vector-valued inequalities for directional averages and directional Hilbert transforms.

We introduce some notation. The large integer $N$ is fixed throughout and left implicit in the notation. By scaling, it will be enough to consider a regular polygon $\mathcal P$ with the following geometric properties: first, $\mathcal P$ has   vertices 
\[
\{v_j =\e^{i\vartheta_j}:\, 1\leq j \leq N+1 \},\qquad v_j\coloneqq\exp( 2\uppi j/ N ),
\] 
on the unit circle $\mathbb S^1$ with $\vartheta_1=\vartheta_{N+1}=0$ and oriented counterclockwise so that $  \vartheta_{j+1}- \vartheta_{j}>0$. The associated Fourier restriction operator is then defined by
\[
T_{\mathcal P}f \coloneqq (\ind_{\mathcal P} \hat f)^\vee.
\]
The proof of the estimate of Theorem \ref{thm:polygon} for $T_{\mathcal P}$ occupies the remainder of this section: by self-duality of the estimate it will suffice to consider the range $2\leq p<4$. 

\subsection{A preliminary decomposition} Let $N$ be a large positive integer and take $\upkappa$  such that $2^{\upkappa-1}< N\leq 2^{\upkappa}$. For each $-2\upkappa\leq k\leq 0$  consider a smooth radial multiplier $m_k$ which is supported on the annulus 
\[
A_k\coloneqq \Big\{\xi\in\R^2:\, 1-\frac{2^{-k-1}}{2^{2\upkappa} }<  |\xi| <1- \frac{2^{-k-5}}{2^{2\upkappa}} \Big\}
\]
and is identically $1$ on the smaller annulus 
\[
a_k\coloneqq \Big\{\xi\in\R^2:\, 1-\frac{2^{-k-2}}{2^{2\upkappa} }<  |\xi| <1- \frac{2^{-k-4}}{2^{2\upkappa}} \Big\}.
\]
Now consider the corresponding radial multiplier operators $T_k$ 
\[
T_k f \coloneqq (m_k \hat f)^\vee, \qquad m_{\upkappa} \coloneqq \sum_{k=-2\upkappa} ^0 m_k.
\]
We note that $m_\upkappa$ is supported in the annulus  
\[
\Big\{\xi\in\R^2:\, \frac{1}{2}<|\xi|<1-\frac{2^{-5}}{2^{2\upkappa}}\Big\}.
\]
 With this in mind let us consider radial functions $m_0,m_{\mathcal P}\in\mathcal S(\R^2)$ with $0\leq m_0,m_{\mathcal P}\leq 1$ such that 
\begin{equation}
\label{e:pfpoly10}
\Big(m_0 +m_\upkappa +m_{\mathcal P}\Big)\ind_{\mathcal P}=\ind_{\mathcal P},
\end{equation}
with the additional requirement that 
\begin{equation}
\label{e:pfpoly11}
\mathrm{supp}(m_{\mathcal P})\subset  A_{\mathcal P}\coloneqq \big\{\xi\in\R^2: \, 1-  2^{-2\upkappa-3} \leq|\xi| \leq1 + 2^{-2\upkappa-3} \big\}.
\end{equation}
Defining 
\[ 
\begin{split}
&\widehat{T_0 f}  \coloneqq \widehat f  m_{0},  \qquad \widehat{T_\upkappa f}   \coloneqq   \widehat f m_{\upkappa} , \qquad 
\widehat{O_{\mathcal P}f}    \coloneqq    \widehat f  m_{\mathcal P} \ind_{\mathcal P} ,
\end{split}
\] 
identity \eqref{e:pfpoly10} implies that $T_{\mathcal P}= T_0 + T_\upkappa+ O_{\mathcal P}$. Observing that $T_0$ is bounded on $p$ for all $1<p<\infty$ with bounds $O_p(1)$  we have
\begin{equation} \label{e:pfpoly60}
\|T_{\mathcal P}\|_{L^p(\R^2)} \lesssim_p  1 +  \|T_{\upkappa}  \|_{L^p(\R^2)}+\|O_{\mathcal P}  \|_{L^p(\R^2)},\qquad 1<p<\infty.
\end{equation}

\subsection{Estimating $T_\upkappa$}  \label{sec:inner}
We aim for the estimate 
\begin{equation}\label{eq:innerannuli}
\|T_{\upkappa} f\|_{p} \lesssim \upkappa^{4(\frac12-{\frac1p})} \|f\|_p, \qquad 2\leq p<4.
\end{equation}
The case $p=2$ is obvious whence it suffices to prove the restricted type version at the endpoint $p=4$
\begin{equation}
\label{e:pfpoly21}
\|T_{\upkappa} (f\cic\ind_E)\|_{4} \lesssim \upkappa|E|^{\frac14} \|f\|_\infty.  
\end{equation}
Now we have that for any $g$
\[
|T_{\upkappa}g| =\Big|\sum_{k=-2\upkappa} ^0 T_k g\Big|\lesssim \Big(\sum_{k=-2\upkappa} ^0 |T_k g|^4\Big)^{\frac14}\upkappa^{\frac34}
\]
and thus
\begin{equation}\label{eq:4/3}
\|T_{\upkappa}g\|_4 \lesssim\upkappa^{\frac34} \Big(\sum_{k=-2\upkappa} ^0 \|T_k g\|_4 ^4\Big)^{\frac14}.
\end{equation}
Let $\{\upomega_{j}:\, j\in J\}$ be the collection of   intervals on $\mathbb S^1$ centered at $v_j\coloneqq\exp\left({{2\uppi i j/}{N}}\right)$ and of length $2^{-\upkappa}$. Note that these intervals have finite overlap and their centers $v_j$ form a $\sim 1/N$-net on $\mathbb S^1$. Now let $\{\upbeta_j:\, j\in J\}$ be a smooth partition of unity subordinated to the finitely overlapping open cover $\{\upomega_{j}:\, j\in J\}$ so that each $\upbeta_j$ is supported in $\omega_j$. We can decompose each $T_k$ as
\[
(T_k f)^{\wedge}(\xi)=\sum_{j\in J} m_k(|\xi|)\upbeta_j \Big(\frac{\xi}{|\xi|}\Big)\hat f(\xi)\eqqcolon \sum_{j\in J} m_{j,k}(\xi)\hat f(\xi),\eqqcolon\sum_{j\in  J} (T_{j,k} f)^{\wedge}(\xi), \qquad \xi\in \R^2.
\]
For $s_j\in S$ and $-2\upkappa\leq k\leq 0  $ we define the conical sectors
\begin{equation}
\label{e:amaranth1}
\Omega_{j,k}\coloneqq \left\{\xi  \in \R^2:\, \xi\in A_k,\: \xi/|\xi|\in \upomega_j\right\} 
\end{equation}
and note that each one of the multipliers $m_{k,j}$ is supported in $\Omega_{j,k}$.	Each $\Omega_{j,k}$ is an annular sector around the circle of radius $1-2^{-k}/2^{2\upkappa}$ of width $\sim 2^{-k}/2^{2\upkappa}$, where $-2\upkappa  \leq k \leq 0$. It is a known observation, usually attributed to C\'ordoba, \cite{CorPoly}*{Theorem 2} or C. Fefferman, \cite{Feff}, that for such parameters we have
\begin{equation}\label{eq:cordobaoverlap}
\sum_{j,j'\in J} \ind_{\Omega_{j,k}+\Omega_{j',k} }\lesssim 1.
\end{equation}
This pointwise inequality and Plancherel's theorem allows us to decouple the pieces $T_{j,k}$ in $L^4$: for each fixed $k$ as above we have
\begin{equation}\label{eq:cordecoupling}
\|T_kf\|_4 \lesssim \Big\| \big(\sum_{j\in J}|T_{j,k}f|^2\big)^\frac{1}{2}\Big\|_4;
\end{equation}
see also the proof of Lemma~\ref{l:Tj} below for a vector-valued version of this estimate. Combining the last estimate with \eqref{eq:4/3} and dominating the $\ell^2$-norm by the $\ell^1$-norm yields
\[
\begin{split}
&\|T_{\upkappa}f\|_4 \lesssim \upkappa^{\frac34}\Bigg( \int_{\R^2}\sum_{k=-2\upkappa} ^0  \big(\sum_{j\in J}|T_{j,k}f|^2\big)^{2} \Bigg)^{\frac14}\lesssim \upkappa^{\frac34}\Bigg( \int_{\R^2}\bigg[\sum_{k=-2\upkappa} ^0  \big(\sum_{j\in J}|T_{j,k}f|^2\big)^{2} \bigg]^{\frac12 2} \Bigg)^{\frac14}
\\
&\qquad\qquad \leq \upkappa^{\frac34}\Bigg( \int_{\R^2}\bigg[\sum_{k=-2\upkappa} ^0   \sum_{j\in J}|T_{j,k}f|^2  \bigg]^{ 2} \Bigg)^{\frac14}\eqqcolon \upkappa^{\frac34} \| \Delta_{J,\upkappa}f\|_4
\end{split}
\]
with
\[
\Delta_{J,\upkappa}f\coloneqq \Big(\sum_{k=-2\upkappa} ^0   \sum_{j\in J}|T_{j,k}f|^2\Big)^{\frac12}.
\]
But now note that $\{T_{j,k}\}_{j,k}$ is a finitely overlapping family of smooth frequency projections on a family of rectangles in at most $\sim N$ directions. Furthermore all these rectangles have one side of fixed length since $|\upomega_j|=2^{-\upkappa}$ for all $j\in J$. So Theorem~\ref{thm:rdfsmooth} with the improvement of Remark~\ref{rmrk:1paramrdf} applies to yield
\begin{equation}\label{eq:sqfnctpol}
\|\Delta_{J,\upkappa}f\|_4 \lesssim (\log\#N)^{\frac14}\|f\|_\infty |E|^\frac{1}{4}\simeq \upkappa^{\frac14}\|f\|_\infty |E|^{\frac14}.
\end{equation}
The last two displays establish \eqref{e:pfpoly21} and thus \eqref{eq:innerannuli}.

\begin{remark} \label{rem:compcor}The term $T_{\upkappa}$ is also present in the argument of \cite{CorPoly}. Therein, an upper estimate of order $O(\upkappa^{\frac54})$ for $p$ near $4$ is obtained, by using the triangle inequality and the  bound $\sup\,\{\|T_{k}\|_{L^4(\R^2)}:\, {-2\upkappa\leq k\leq 0}\}\sim \upkappa^{\frac14}$ for the  smooth restriction to a single annulus.
\end{remark}

\subsection{Estimating $O_{\mathcal P}$}\label{sec:OP} In this subsection we will prove the estimate
\begin{equation}\label{eq:OP}
\|O_{\mathcal P} f \|_p\lesssim \upkappa^{4(\frac12-\frac1p)}\|f\|_p.
\end{equation}
Let $\Phi$ be a smooth radial function with support in the annular region $\{\xi\in\R^2:\, 1-c2^{-2\upkappa}<|\xi|< 1+c2^{-2\upkappa}\}$, where $c$ is a fixed small constant,  and satisfying $0\leq \Phi\leq 1$. Let $\{\upbeta_j:\, j\in J\}$ be a  partition of unity on $\mathbb S^1$ relative to  intervals $\upomega_j$ as in  \S\ref{sec:inner}.
Define the Fourier multiplier operators on $\R^2$
\begin{equation}
\label{e:pfpoly50}
\widehat{T_j f} (\xi) \coloneqq \Phi(\xi) \upbeta_j\left( \frac{\xi}{|\xi|} \right)\hat f(\xi),\qquad \xi \in \R^2.
\end{equation}
The operators $T_j$ satisfy  a square function estimate
\begin{equation}
\label{e:pfpoly40}
\begin{split}
&\left\| \{  T_j f\} \right\|_{L^p(\R^2; \ell^2_J)} \lesssim   \upkappa^{\frac12-\frac1p}  \|f\|_p , \qquad    2\leq p<4,\\
&\left\| \{  T_j (f\cic{1}_E)\} \right\|_{L^4(\R^2; \ell^2_J)} \lesssim  \upkappa^{\frac14 }|E|^{\frac14}\|f\|_\infty,
\end{split}
\end{equation}
which follows in the same way as \eqref{eq:sqfnctpol}, by using Theorem~\ref{thm:rdfsmooth} with the improvement of Remark~\ref{rmrk:1paramrdf}. They also obey a vector-valued estimate
\begin{equation}
\label{e:pfpoly41}
\begin{split}
&\left\| \{  T_j f_j\} \right\|_{L^p(\R^2; \ell^2_J)} \lesssim   \upkappa^{\frac12-\frac1p}  \left\| \{    f_j\} \right\|_{L^p(\R^2; \ell^2_J)} , \qquad    2\leq p<4,
\\
&\left\| \{  T_j (f_j\cic{1}_F)\} \right\|_{L^4(\R^2; \ell^2_J)} \lesssim  \upkappa^{\frac14 }|F|^{\frac14}\left\| \{   f_j\} \right\|_{L^\infty(\R^2; \ell^2_J)}.
\end{split}
\end{equation}
These estimates are easy to prove. Indeed note that it suffices to prove the endpoint restricted estimate at $p=4$. Using the Fefferman-Stein inequality for fixed $j\in J$ we can estimate for each function $g$ with $\|g\|_2=1$
\[
\begin{split}
& \int_{\R^2} \sum_{j\in J}|T_j(f_j\ind_F)|^2  g \lesssim \sum_{j\in J}\int_{\R^2} |f_j\ind_F|^2 \M_{j} g \leq \left\| \{   f_j\} \right\|_{L^\infty(\R^2; \ell^2_J)} ^2 \int_F \sup_{j\in J}\M_j g 
\\
&\qquad\qquad\lesssim |F|^{\frac 12} \big\|\sup_{j\in J}\M_j g \big\|_{L^{2,\infty}(\R^2)},
\end{split}
\]
where $\M_j$ is the Hardy-Littlewood maximal function with respect to the collection of parallelograms in $\mathcal D^2 _{s_j,-2\upkappa,-\upkappa}$ with $s_j$ defined through $(-s_j,1)\coloneqq v_j$. Now $\sup_{j\in J}M_j$ is the maximal directional maximal function and the number of directions involved in its definition is comparable to $N\sim 2^{\upkappa}$. Then the maximal theorem of Katz from \cite{KatzDuke} applies to give the estimate 
\[
\big\|\sup_{j\in J}\M_j g\big \|_{L^{2,\infty}(\R^2)}\lesssim  \upkappa ^{\frac12}.
\]
This proves the second of the estimates \eqref{e:pfpoly41} and thus both of them by interpolation.

In the estimate for $O_{\mathcal P}$ we will  also need the following decoupling result.

\begin{lemma} \label{l:Tj}
Let $2\leq p<4$. Then
\[
\Big\| \sum_{j} T_j f_j \Big\|_{p} \lesssim \upkappa^{\frac12-\frac1p}\big\| \{    f_j\} \big\|_{L^p(\R^2; \ell^2_J)}.
\]
\end{lemma}

\begin{proof} Note that the case $p=2$ of the conclusion is trivial due to the finite overlap of the supports of the multipliers of the operators $T_j$. Thus by vector-valued restricted type interpolation of the operator
\[
\{f_j\} \mapsto O(\{f_j\})\coloneqq \sum_{j\in J} T_j f_j
\]
it suffices to prove a restricted type $L^{4,1}\to L^{4}$ estimate:
\begin{equation}
\label{e:pfpoly31}
\left\| O(\{f_j\})\right\|_{4} \lesssim \upkappa^{\frac14}|E|^{\frac14}
\end{equation}
 for functions with $\|\{f_j\}\|_{\ell^2}\leq \cic{1}_E$.
To do so note that the finite overlap of the supports of  $\widehat{T_j f_j}*\widehat{T_k f_k}$ over $j,k$, as in \eqref{eq:cordobaoverlap}, entail  
\[
\left\| O(\{f_j\}) \right\|_{4} \lesssim \left\| \{  T_j f_j \} \right\|_{L^4(\R^2; \ell^2_J)}
\]
and the restricted type estimate \eqref{e:pfpoly31} follows from   \eqref{e:pfpoly41}.
\end{proof}

We come to the main argument for $O_{\mathcal P}$. 
Let $m_{\mathcal P}$ be as as in \eqref{e:pfpoly10}-\eqref{e:pfpoly11} and $T_j$ be the multiplier operators from \eqref{e:pfpoly50} corresponding to the choice $\Phi=m_{\mathcal P}$. Then obviously
\[
m_{\mathcal P} \hat f = \sum_{j\in J} \widehat {T_j f}.
\]
We may also tweak $\Phi$ and the partition of unity on $\mathbb S^1$ to obtain further multiplier operators $\widetilde T_j$ as in  \eqref{e:pfpoly50} and  such that the Fourier transform of the symbol of $\widetilde T_j$ equals one on the support of the symbol of  $T_j$. 
With these definitions in hand we estimate for $2<p<4$
\begin{equation}
\label{e:pfpoly32}
\begin{split}
\|O_{\mathcal P} f \|_p &= \Big\| \sum_{j} \widetilde T_j (T_j T_{\mathcal P} f) \Big\|_p
\lesssim  \upkappa^{\frac12-\frac1p}\left\| \{T_{\mathcal P}(T_j  f) \} \right\|_{L^p(\R^2; \ell^2_J)} 
\\
&= \upkappa^{\frac12-\frac1p}\left\| \{H_jH_{j+1}(T_j  f) \} \right\|_{L^p(\R^2; \ell^2_J)}.
\end{split}
\end{equation}
The first inequality is an application of Lemma~\ref{l:Tj} for $\widetilde T_j$. The last equality is obtained by observing that the polygon multiplier $T_{\mathcal P}$ on the support of each $T_j$ may be written as  a (sum of $O(1)$)  directional biparameter  multipliers $H_jH_{j+1}$ of iterated Hilbert transform type, where $H_j$ is a Hilbert transform along the direction $\nu_j$, which is the unit vector perpendicular to the $j$-th side of the polygon, and pointing inside the polygon; these are at most $\sim N$ such directions.

In order to complete our estimate for $O_{\mathcal P}$ we need the following Meyer-type lemma for directional Hilbert transforms of the form
\[
H_v f (x)\coloneqq \int_{\R^2} \hat f(\xi)\ind_{\{\xi\cdot v>0\}}e^{ix\cdot \xi}\,\d\xi,\qquad x\in\R^2.
\]

\begin{lemma}\label{lem:meyer} Let $V\subset \mathbb S^1$ be a finite set of directions and $H_v$ be the Hilbert transform in the direction $v$. Then for $\frac43< p<4$ we have
\[
\begin{split}
&\left\| \{ H_v f_v\} \right\|_{L^p(\R^2; \ell^2_V)} \lesssim   (\log\#V)^{\left|\frac12-\frac1p\right|}  \left\| \{    f_v\} \right\|_{L^p(\R^2; \ell^2_V)} .
\end{split}
\]
The dependence on $\#V$ is best possible.
\end{lemma}

\begin{proof} It suffices to prove the estimate for $2<p<4$. The proof is by  way of duality and uses the following inequality for the Hilbert transform: for $r>1$ and $w$ a non-negative locally integrable function we have
\[
\int_{\R^2} |H_v f|^2 w \lesssim \int_{\R^2} |f|^2 (\M_v |w|^r	)^\frac1r
\]
with $\M_v$ given by \eqref{eq:dirMv}. See for example \cite{CP} and the references therein. Using this we have for a suitable $g\in L^{(p/2)'}$ of norm one that
\[\begin{split}
& \left\| \{ H_v f_v\} \right\|_{L^p(\R^2; \ell^2_V)} ^2=\int_{\R^2}\sum_{v\in V} |H_vf_v|^2 g \lesssim \sum_{v\in V}\int_{\R^2} |f_v|^2 (\M_v |g|^r	)^\frac1r 
\\
&\qquad\qquad\lesssim \|\{f_v\}\|_{L^p(\R^2;\ell^2 _V)} ^2 \big\|(\M_V  |g|^r	)^\frac1r \big\|_{L^{(p/2)'}(\R^2)}
\end{split}
\] 
with $\M_V g\coloneqq \sup_{v\in V}\M_vg$. Now for $2<p<4$ there is a choice of $1<r<\frac{p}{2(p-2)}$ so that $\frac{p}{r(p-2)}>2$. This means that the maximal theorem of Katz from \cite{KatzDuke} applies again to give
\[
 \big\|(\M_V  |g|^r	)^\frac1r \big\|_{L^{(p/2)'}(\R^2)}\lesssim (\log\#V)^{1 -\frac2p} 
\]
and so the proof of the upper bound is complete. The optimality is discussed in \S\ref{sec:cexmeyer}.
\end{proof}

Let us now go back to the estimate for $O_{\mathcal P}$. The left hand side of \eqref{e:pfpoly32} contains a double Hilbert transform. By an iterated application of Lemma~\ref{lem:meyer} we thus have
\[
\left\| \{H_jH_{j+1}(T_j  f) \} \right\|_{L^p(\R^2; \ell^2_J)}\lesssim \upkappa ^{1-\frac2p}\left\| \{T_j  f) \} \right\|_{L^p(\R^2; \ell^2_J)}
\]
since the number of directions is $N=2^\upkappa$. The final estimate for the right hand side of the  display above is a direct application of \eqref{e:pfpoly40} which together with \eqref{e:pfpoly32} yields the estimate for $\|O_{\mathcal P} f \|_p$ claimed in \eqref{eq:OP}.

Now the decomposition \eqref{e:pfpoly60} together with the estimate of \S\ref{sec:inner} for $T_\upkappa$ and the estimate \eqref{eq:OP} for $O_{\mathcal P}$ complete the proof of Theorem~\ref{thm:polygon}.

\begin{remark}\label{rmrk:invsf} Consider a function $f$ in $\R^2$ such that $\mathrm{supp}(\hat f)\subseteq A_\updelta$ where $A_{\updelta}$ is an annulus of width $\updelta^2$ around $\mathbb S^1$. Decomposing $A_{\updelta}$ into a union of $O(1/\updelta)$ finitely overlapping annular boxes of radial width $\updelta^2$ and tangential width $\updelta$ we  can write $f=\sum_{j\in J} T_j f$ where each $T_j$ is a smooth frequency projection onto one of these annular boxes, indexed by $j$. Then if $\widetilde T_j$ is a multiplier operator whose symbol is identically one on the frequency support of $T_j f$ and supported on a slightly larger box, we can write $f=\sum_j \widetilde{T_j}T_j f$, as in \eqref{e:pfpoly32} above. Then Lemma~\ref{l:Tj} yields
	\[
	\|f\|_{L^p(\R^2)} \lesssim (\log (1/\updelta))^{\frac{1}{2}-\frac{1}{p}} \|\{T_jf\}\|_{L^p(\R^2;\ell^2 _J)}.
	\]
This is the inverse square function estimate claimed in the remark after Theorem~\ref{thm:polygon} in the introduction.
\end{remark}

\section{Lower bounds and concluding remarks}\label{sec:cex}

\subsection{Sharpness of Meyer's lemma}\label{sec:cexmeyer}
We briefly sketch the quantitative form of Fefferman's counterexample \cite{FeffBall} proving the sharpness of Lemma~\ref{lem:meyer}. Let $N$ be a large dyadic integer. Using a standard Besicovitch-type construction we produce rectangles $\{R_j: \, j=1,\ldots, N\}$ with sidelengths $1\times \frac1N$, so that the long side of $R_j$ is oriented along $v_j\coloneqq \exp(2\uppi i j/N)$. Now we consider the set $E$ to be the union of these rectangles and
\[
\left|E\coloneqq \bigcup_{j=1}^N R_j\right| \lesssim \frac{1}{\log N}.
\] 
Denoting by $\widetilde {R_j}$ the $2$-translate of $R_j$ in the direction of $v_j$ we gather that $ \{\widetilde {R_j}:j=1,\ldots,N\}$ is a pairwise disjoint collection. Furthermore if  $H_j$ is the Hilbert transform in direction $v_j$, there holds
\[
|H_j \cic{1}_{R_j}| \geq c \cic{1}_{\widetilde {R_j}}.
\]
 Therefore for all $1<p<\infty$
\[
\bigg\|\Big(\sum_{j=1}^N |H_j \cic{1}_{R_j}|^2\Big)^{\frac12} \bigg\|_{p} \geq c \bigg| \bigcup_{j=1}^N  \widetilde {R_j} \bigg|^{\frac1p} \geq c
\] 
while for $p\leq 2$
\[
\bigg\|\Big(\sum_{j=1}^N |\cic{1}_{R_j}|^2\Big)^{\frac12} \bigg\|_{p} \leq \bigg(\sum_{j=1}^N |R_j| \bigg)^{\frac12} |E|^{\frac1p-\frac12} \lesssim (\log N)^{\frac{1}{2}-\frac{1}{p}}.
\]
Self-duality of the square function estimate then entails the optimality of the estimate of Lemma~\ref{lem:meyer}.

\subsection{Sharpness of the directional square function bound}\label{sec:rdfcex} In this subsection we prove that the bound of Theorem~\ref{thm:rdfrough} is best possible, up to the doubly logarithmic terms. In particular we prove that the bound of Remark~\ref{rmrk:1paramrdf} is best possible.

We begin by showing a lower bound for the rough square function estimate
\begin{equation}\label{eq:cexrdf}
\|\{P_F g\}\|_{L^p(\R^2;\ell^2 _{\mathcal F})} \leq \|\{P_F\}: L^p(\R^2)\to L^p(\R^2;\ell^2 _{\mathcal F})\| \|g\|_p,\qquad 2\leq p <4,
\end{equation}
where the notations are as in \S\ref{sec:rdf}. Now as in Fefferman's argument in \cite{FeffBall} one can easily show that the estimate above implies the vector-valued inequality for directional averages, for directions corresponding to the directions of rectangles in $\mathcal F$. For this let $\#V=N$ where $V$ is the set of directions of rectangles in $\mathcal F$. Now consider functions $\{g_F\}_{F\in\mathcal F}$ with compact Fourier support; by modulating these function we can assume that $\mathrm{supp}(\widehat{g_F})\subset  B(c_F,A)$ for some $A>1$ and $\{c_F\}_{F\in\mathcal F}$ a $100AN$-net in $\R^2$. Then if  $F$ is a rectangle centered at $c_F$ with short side $1$ parallel to a direction $v_F\in V$ and long side of length $N$ parallel to $v_F ^\perp$, then we have that $|P_Fg_F|=|A_{v_F}g_F|$ where $A_{v_F}$ is the averaging operator
\[
A_{v_F}f(x)\coloneqq 2N \int_{|t|\leq 1/2} \int_{N|s|<1} f(x-tv_F-sv_F ^\perp)\,\d t \,\d s,\qquad x\in \R^2.
\]
Note that this is a single-scale average with respect to rectangles of dimensions $1\times 1/N$ in the directions $v_F,v_F ^\perp$ respectively. Since the frequency supports of these functions are well-separated we gather that for all choices of signs $\upvarepsilon_F\in\{-1,1\}$ we have
\[
\sum_{T\in\mathcal F}  |P_T G|^2\coloneqq \sum_{T\in\mathcal F} \Big|P_T \big(\sum_{F\in \mathcal F} \upvarepsilon_F g_F \big)\Big|^2 = \sum_{T\in\mathcal F}|P_T g_T|^2.
\]
Thus applying \eqref{eq:cexrdf} with the function $G$ as above and averaging over random signs we get
\[
\big\| \{A_{v_F}g_F\}\big\|_{L^p(\R^2;\ell^2 _{\mathcal F})}\leq  \|\{P_F\}: L^p(\R^2)\to L^p(\R^2;\ell^2 _{\mathcal F})\|  \|\{g_F\}\|_{L^p(\R^2;\ell^2 _{\mathcal F})},\qquad 2\leq p <4.
\]
Now we just need to note that as in \S\ref{sec:cexmeyer} we have that 
\[
A_{v_F}\ind_{R_F}\gtrsim \ind_{\widetilde {R_F}}
\]
where $\{R_F\}_{F\in\mathcal F}$ are the rectangles used in the Besicovitch construction in \S\ref{sec:cexmeyer}. As before we get
\[
\big\|\{P_F\}: L^p(\R^2)\to L^p(\R^2;\ell^2 _{\mathcal F})\big\| \gtrsim (\log \#V)^{\frac12-\frac1p} .
\]
For $p<2$ the square function estimate \eqref{eq:cexrdf} is known to fail even in the case of a single directions; see for example the counterexample in \cite{RdF}*{\S1.5}.

One can use the same argument in order to show a lower bound for the norm of the smooth square function
\begin{equation}\label{eq:cexrdfsmmoth}
\|\{P_F ^{\circ} g\}\|_{L^p(\R^2;\ell^2 _{\mathcal F})} \leq \|\{P_F ^\circ\}: L^p(\R^2)\to L^p(\R^2;\ell^2 _{\mathcal F})\| \|g\|_p,\qquad 2\leq p <4.
\end{equation}
Indeed, following the exact same steps we can deduce a vector-valued inequality for smooth averages
\[
A_{v_F} ^\circ f(x)\coloneqq  \int_{\R}\int_{\R}  f(x-tv_F-sv_F ^\perp) \upgamma_F(t,s) \,\d t \,\d s,\qquad x\in \R^2,
\]
where $\upgamma_F$ is the smooth product bump function used in the definition of $P_F ^\circ$ in \S\ref{sec:rdf}. By a direct computation one easily shows the analogous lower bound $A_{v_F} ^\circ\ind_{R_F}\gtrsim \ind_{\widetilde {R_F}}$ for the rectangles of the Besicovitch construction and this completes the proof of the lower bound for smooth projections as well.

\subsection{Sharpness of C\'ordoba's bound for radial multipliers}\label{sec:cexradial} 
Firstly we remember the definition of each radial multiplier $P_\updelta$: Let $\Phi:\R\to\R$ be a smooth function which is supported in $[-1,1]$ and define
\[
P_{\updelta} f(x) \coloneqq\int_{\R^2} \widehat f(\xi) \Phi\big(\updelta^{-1} (1-|\xi|) \big) \e^{ix\cdot \xi}\, \d \xi,\qquad x\in \R^2.
\] 
These smooth radial multipliers were used extensively in \S\ref{sec:polygon}. In \cite{CordBR} C\'ordoba has proved the bound
\[
\|P_\updelta f\|_p \lesssim (\log1/\updelta)^{\left|\frac12-\frac1p\right|}\|f\|_p,\qquad \frac34\leq p \leq 4. 
\]
In fact the same bound is implicitly proved in \S\ref{sec:polygon} in a more refined form, but only in the open range $p\in(3/4,4)$ with weak-type analogues at the endpoints. More precisely we have discretized $P_{\updelta}$ into a sum of pieces $\{P_{\updelta,j}\}_{j\in J}$, where each $P_{\updelta,j}$ is a smooth projection onto an annular box of width $\updelta$ and length $\sqrt{\updelta}$, pointing along one of $N$ equispaced directions $\nu_j$. Then it follows from the considerations in \S\ref{sec:polygon} that
\begin{equation}\label{eq:vvaluedradial}
\begin{split}
\| \{P_{\updelta,j} f\}\|_{L^p(\R^2;\ell^2 _J)} &\lesssim \log(1/\updelta)^{\frac12-\frac1p}\|f\|_p,\qquad 2<p<4,
\\
\| \{P_{\updelta,j} f\ind_F\}\|_{L^4(\R^2;\ell^2 _J)} &\lesssim \log(1/\updelta)^{\frac14}\|f\|_\infty |F|^\frac14.
\end{split}
\end{equation}
Obviously one gets the same bound by duality for $4/3<p<2$ while the $L^2$-bound is trivial. Now these estimates imply C\'ordoba's estimate for $P_{\updelta}$ in the open range $(3/4,4)$ by the decoupling inequality \eqref{eq:cordecoupling}, also due to C\'ordoba. On the other hand C\'ordoba's estimate is sharp. Indeed one uses the same rescaling and modulation arguments as in the previous subsection in order to deduce a vector-valued inequality for smooth averages starting by C\'ordoba's estimate. Testing this vector-valued estimate against the rectangles of the Besicovitch construction proves the familiar lower bound for $P_\updelta$ and thus also shows the optimality of the estimates in \eqref{eq:vvaluedradial}. We omit the details.

\subsection{Lower bounds for the conical square function}\label{sec:cexconical} We conclude this section with a simple example that provides a lower bound for the operator norm of the conical square function $\|C_\upomega(f):{\ell^2_{\cic{\upomega}}}\| $ of Theorem~\ref{thm:wdcones} and the smooth conical square function $\|C_\upomega^\circ:\ell^2_{\cic{\upomega}}\|$ of Theorem~\ref{thm:smoothcones}. The considerations in this subsection also rely on the Besicovitch construction so we adopt again the notations of \S\ref{sec:cexmeyer} for the rectangles $\{R_j:\, 1\leq j \leq N\}$ and their union $E$. Let $H^+ _j$ denote the frequency projection in the half-space $\{\xi\in\R^2:\, \xi \cdot v_j>0\}$ where $v_j\coloneqq\exp(2\uppi i j/N)$. We begin by observing that
\begin{equation}
\label{e:cexbes1}
  H_j ^+ f - H_{j+1}  ^+ f = C_{j}P_+f -C_{j}P_{-}f,
\end{equation}
where $P_{+},P_{-}$ denote the rough frequency projections in the upper and lower half-space respectively and $C_{v_j}$ is the   multiplier associated with the cone bordered by  ${v_j,v_{j+1}}$. Since $H_j ^+$ is a linear combination of the identity with the usual directional Hilbert transform $H_j$ along $v_j$ we conclude that
\[
\Big\|\Big(\sum_{j=1} ^N |(H_{j+1}-H_j)f|^2\Big)^\frac12\Big\|_p\lesssim \big\| \{C_{j}\}:  L^p(\R^2)\to L^p(\R^2;\ell^2 _j)\big\|\, \|f\|_p,\qquad 2\leq p <4.
\]
Now note that for each fixed $1\leq k\leq N$ we have
\begin{equation}
\label{e:cexbes0}
\ind_{\widetilde{R_k}} \sum_j (H_{j }-H_{j+1}) \ind_{R_j}   =  \ind_{\widetilde{R_k}}   H_k\ind_{R_k} \gtrsim \ind_{\widetilde{R_k}} 
\end{equation}
if $\widetilde{R_k}$ is 
a sufficiently large translation of $R_k$ in the positive  direction $v_k$. Thus
\[
\begin{split}
\Big|\int \ind_{\cup_k \widetilde{R_k}} \sum_{j=1}^N  (H_{j+1}-H_j)\ind_{R_j}\Big| 	\gtrsim \Big|\sum_k \int_{\widetilde{R_k}}\ind_{\widetilde{R_k}}\Big|\simeq 1.
\end{split}
\]
On the other hand the left hand side of the display above is bounded by a constant multiple of
\[
\big\|  \{C_{j}\}:  L^p(\R^2)\to L^p(\R^2;\ell^2 _j)\big\| \, \Big\| \big(\sum_j \ind_{R_j}^2\big)^{\frac12}\Big\|_{p'}\lesssim \big\| C_V:\, L^p(\R^2)\to L^p(\R^2;\ell^2)\big\| (\log N)^{\frac{1}{2}-\frac{1}{p'}}
\]
for all $2\leq p<4$. We thus conclude that
\[
\big\| \{C_{j}\}:  L^p(\R^2)\to L^p(\R^2;\ell^2)\big\|   \gtrsim (\log N)^ {\frac12-\frac1p} ,\qquad 2\leq p <4.
\]

We explain how this   counterexample can be modified to get a lower  square function estimate for the smooth cone multipliers $C^\circ_\upomega$ from \eqref{e:smoothcones}  matching the upper bound of Theorem \ref{thm:smoothcones}. For $t\in \R$ write $v_j^t\coloneqq\exp(2\uppi i (j+t)/N)$ and let $H_{j}^t$ and $H_{j}^{t,+} $ be the directional Hilbert transform and analytic projection along $v_j^t$, respectively. Let $\updelta>0$ be a small parameter to be chosen later and for each $1\leq j\leq N$ let $\upomega_j$ be an interval of size $\updelta N^{-1}$ centered around $2\pi j/N$. Arguing as in \eqref{e:cexbes1}, 
\begin{equation}
  C_{\upomega_j}^\circ P_+f - C_{\upomega_j}^\circ P_{-}f= \avgint_{N|t|<\updelta} \upalpha\left({\textstyle\frac{Nt}{\updelta}}\right) \left( H_j ^{t,+} f - H_{j+1}^{t,+}   f  \right) \d t 
\end{equation}
for a suitable nonnegative averaging function $\upalpha$ which equals   $1$ on $[-\frac14,\frac14]$.  Now, if $\widetilde{R_k}$ is again a  sufficiently large translation of $R_k$ in the positive direction $v_k$ and $\updelta$ is chosen sufficiently small depending only on the translation amount, the analogue of  \eqref{e:cexbes0}
\begin{equation}
\ind_{\widetilde{R_k}}\inf_{N|t|<\updelta}  \sum_{j=1}^N(H_{j}^t   -H_{j+1}^t )  \ind_{R_j} = \ind_{\widetilde{R_k}}  \inf_{N|t|<\updelta} H_k^t[\ind_{R_k}] \gtrsim \ind_{\widetilde{R_k}}.
\end{equation}
 The lower bound for $\| \{C_{\upomega_j}\}:  L^p(\R^2)\to L^p(\R^2;\ell^2 _j)\|$ then follows exactly as in the previous case.

\bibliography{squaref}
 \bibliographystyle{amsplain}

\end{document}